\documentclass[12pt,fleqn,reqno]{article}
\usepackage{epsfig,amssymb,latexsym,amsmath,pifont,multicol,mathtools,amsmath,verbatim,amsthm,float,lipsum}
\usepackage{caption} 
\usepackage[utf8]{inputenc}
\usepackage[T1]{fontenc}

\usepackage[varg]{txfonts}
\let\mathbb=\varmathbb

\usepackage[sans]{dsfont}

\usepackage[left=2.5cm, right=2.5cm, top=2.5cm, bottom=2.5cm]{geometry}

\usepackage[%
cal=euler,
bb=fourier,
scr=zapfc,
]
{mathalfa}

\usepackage[font=small,labelfont=bf]{caption}
\usepackage{subfigure}
\usepackage{graphicx}
\graphicspath{{Figures/}} 
\usepackage{acronym}
\usepackage{latexsym}
\usepackage{paralist}
\usepackage{wasysym}
\usepackage{xspace}
\usepackage{framed}
\usepackage{palatino,pxfonts}
\usepackage{authblk}
\usepackage{enumitem}
\usepackage[numbers,sort&compress]{natbib}

\usepackage[dvipsnames,svgnames]{xcolor}
\colorlet{MyBlue}{DodgerBlue!75!Black}
\colorlet{MyGreen}{DarkGreen!95!Black}

\usepackage{hyperref}
\hypersetup{
colorlinks=true,
linktocpage=true,
pdfstartview=FitH,
breaklinks=true,
pdfpagemode=UseNone,
pageanchor=true,
pdfpagemode=UseOutlines,
plainpages=false,
bookmarksnumbered,
bookmarksopen=false,
bookmarksopenlevel=1,
hypertexnames=true,
pdfhighlight=/O,
urlcolor=MyBlue!60!black,linkcolor=MyBlue!70!black,citecolor=DarkGreen!70!black, % <--- for screen
pdftitle={},
pdfauthor={},
pdfsubject={},
pdfkeywords={},
pdfcreator={pdfLaTeX},
pdfproducer={LaTeX with hyperref}
}
\numberwithin{equation}{section}  %numberwithin goes before cleverefs when using hyperref
\usepackage[sort&compress,capitalize,nameinlink]{cleveref}
\crefname{example}{Ex.}{Exs.}

\crefrangeformat{equation}{\upshape(#3#1#4)\textendash(#5#2#6)}

\newcommand{\dd}{\:d}

\newcommand{\eps}{\varepsilon}

\newcommand{\dif}{\dd}

\DeclareMathOperator*{\argmin}{argmin}
\DeclareMathOperator*{\Argmin}{Argmin}
\DeclarePairedDelimiter{\ceil}{\lceil}{\rceil}

\DeclareMathOperator{\bd}{bd}
\DeclareMathOperator{\cl}{cl}

\DeclareMathOperator{\diag}{diag}

\DeclareMathOperator{\dom}{dom}

\DeclareMathOperator{\Int}{int}

\DeclareMathOperator{\image}{im}

\DeclareMathOperator{\tr}{tr}

\DeclareMathOperator{\Id}{Id}

\newcommand{\ct}{\mathtt{t}}

\newcommand{\bA}{{\mathbf A}}
\newcommand{\BB}{{\mathbf B}}
\newcommand{\bH}{\mathbf{H}}
\newcommand{\bI}{\mathbf{I}}
\newcommand{\bJ}{\mathbf{J}}
\newcommand{\XX}{\mathbf{X}}

\newcommand{\gbold}{\mathbf{g}}

\newcommand{\bS}{\mathbf{S}}

\newcommand{\bN}{\mathbf{N}}
\newcommand{\bZ}{\mathbf{Z}}
\newcommand{\bp}{\mathbf{p}}
\renewcommand{\iff}{\Leftrightarrow}

\renewcommand{\emptyset}{\varnothing}
\newcommand{\eqdef}{\triangleq}

\newcommand{\scrB}{\mathcal{B}}

\newcommand{\setE}{\mathsf{E}}
\newcommand{\scrF}{\mathcal{F}}

\newcommand{\scrH}{\mathcal{H}}

\newcommand{\setK}{\mathsf{K}}

\newcommand{\setL}{\mathsf{L}}

\newcommand{\scrP}{\mathcal{P}}

\newcommand{\scrT}{\mathcal{T}}

\newcommand{\scrW}{\mathcal{W}}

\newcommand{\setX}{\mathsf{X}}

\newcommand{\1}{\mathbf{1}}
\newcommand{\Rn}{\R^n}

\newcommand{\R}{\mathbb{R}}

\newcommand{\K}{\mathbb{K}}
\newcommand{\symm}{\mathbb{S}}
\newcommand{\ball}{\mathbb{B}}

\newcommand{\feas}{\mathsf{X}}							% for feasible region
\usepackage[ruled,vlined]{algorithm2e}
\usepackage{algpseudocode}								% for algorithm macros
\theoremstyle{plain}
\newtheorem{theorem}{Theorem}

\newtheorem*{corollary*}{Corollary}
\newtheorem{lemma}[theorem]{Lemma}
\newtheorem{proposition}[theorem]{Proposition}

\theoremstyle{definition}
\newtheorem{definition}[theorem]{Definition}
\newtheorem*{definition*}{Definition}
\newtheorem{assumption}{Assumption}
\newcommand{\close}{\hfill{\footnotesize$\Diamond$}}

\theoremstyle{remark}
\newtheorem{remark}{Remark}
\newtheorem*{remark*}{Remark}
\newtheorem*{notation*}{Notational remark}
\newtheorem{example}{Example}

\numberwithin{theorem}{section}
\numberwithin{remark}{section}
\numberwithin{example}{section}

\DeclarePairedDelimiter{\abs}{\lvert}{\rvert}

\DeclarePairedDelimiter{\inner}{\langle}{\rangle}
\DeclarePairedDelimiter{\norm}{\lVert}{\rVert}

\newacro{SC}[SC]{self-concordant}
\newacro{SCB}[SCB]{self-concordant barrier}
\newacro{SSB}[SSB]{self-scaled barrier}
\newacro{FOM}{First-order method}

\DeclareMathOperator{\AHBA}{\mathbf{AHBA}}
\DeclareMathOperator{\SAHBA}{\mathbf{SAHBA}}

\title{Hessian Barrier Algorithms for non-convex conic optimization}
\date{\today}

\author[1]{\small Pavel Dvurechensky}
\author[2]{\small Mathias Staudigl}

\affil[1]{\footnotesize Weierstrass Institute for Applied Analysis and Stochastics, Mohrenstr. 39, 10117 Berlin, Germany\\
(\href{mailto:pavel.Dvurechensky@wias-berlin.de}{pavel.dvurechensky@wias-berlin.de})}

\affil[2]{\footnotesize Maastricht University, Department of Advanced Computing Sciences, P.O. Box 616, NL\textendash 6200 MD Maastricht, The Netherlands\\
(\href{mailto:m.staudigl@maastrichtuniversity.nl}{m.staudigl@maastrichtuniversity.nl})}

\begin{document}
\maketitle
%---------------------------------------------
%%% ABSTRACT
%----------------------------------------------------------------------
\begin{abstract}
%----------------------------------------------------------------------
%%% Abtract
%----------------------------------------------------------------------
% !TEX root = ./HBAConicMain.tex
%
A key problem in mathematical imaging, signal processing and computational statistics is the minimization of non-convex objective functions that may be non-differentiable at the boundary of the feasible set. This paper proposes a new family of first- and second-order interior-point methods for non-convex optimization problems with linear and conic constraints, combining logarithmically homogeneous barriers with quadratic and cubic regularization respectively. Our approach is based on a potential-reduction mechanism and, under the Lipschitz continuity of the corresponding derivative with respect to the local barrier-induced norm, attains a suitably defined class of approximate first- or second-order KKT points with worst-case iteration complexity $O(\eps^{-2})$ (first-order) and $O(\eps^{-3/2})$ (second-order), respectively. Based on these findings, we develop new path-following schemes attaining the same complexity, modulo adjusting constants. These complexity bounds are known to be optimal in the unconstrained case, and our work shows that they are upper bounds in the case with complicated constraints as well. 
%A key feature of our methodology is the use of self-concordant barriers to construct strictly feasible iterates via a disciplined decomposition approach and without sacrificing on the iteration complexity of the method. 
To the best of our knowledge, this work is the first which achieves these worst-case complexity bounds under such weak conditions for general conic constrained non-convex optimization problems.

\end{abstract}

%*************************************************************
%*****    BODY TEXT
%*************************************************************
\renewcommand{\sharp}{\gamma}
\acresetall
\allowdisplaybreaks

%%%%%%%%%%%%%%%%%%%%%%%%%%%%%%%%%%%%%%%%%
%%%%%%%% INTRO%%%%%%%%%%%%%%%%%%%%%%%%%%%
%%%%%%%%%%%%%%%%%%%%%%%%%%%

\section{Introduction}
\label{sec:intro}
%----------------------------------------------------------------------
%%% Introduction
%----------------------------------------------------------------------
% !TEX root = ./HBAConicMain.tex
%
Let $\setE$ be a finite dimensional vector space with inner product $\inner{\cdot,\cdot}$ and norm $\norm{\cdot}$. In this paper we are concerned with solving constrained conic optimization problems of the form 
%%%%
\begin{equation}\label{eq:Opt}\tag{Opt}
 \min_{x} f(x)\quad \text{s.t.: } \bA x=b,\; x\in\bar{\setK}.
\end{equation}
%%%%%%
The main working assumption underlying our developments is as follows:
\begin{assumption}\label{ass:1}
\begin{enumerate}
\item $\bar{\setK}\subset\setE$ is a regular convex cone with nonempty interior $\setK$: $\bar{\setK}$ is closed convex, solid and pointed (i.e. contains no lines);
\item $\bA:\setE\to\R^{m}$ is a linear operator assigning each element $x\in\setE$ to a vector in $\R^m$ and having full rank\footnote{
Note that this assumption is not restrictive. If the linear operator maps a point $x$ to a lower-dimensional subset, then it is possible to eliminate redundant constraints, or we are working with an inconsistent system. The latter is excluded from our considerations, so in fact this assumption is without loss of generality. 
}
, i.e., $\image(\bA)=\R^{m}$, $b \in \R^{m}$; 
\item The feasible set $\bar{\setX}=\bar{\setK}\cap\setL$, where $\setL=\{x\in\setE\vert\bA x=b\}$, has nonempty relative interior denoted by $\setX=\setK\cap\setL$;
%\item The feasible set $\bar{\setX}\eqdef\bar{\setK}\cap\setL\neq\emptyset$, where $\setL\eqdef\{x\in\setE\vert\bA x=b\}$, is closed and convex;
%\item The relative interior of $\bar{\setX}$, denoted by $\setX\eqdef\setK\cap\setL$, is nonempty;
\item $f:\setE\to\R$ is possibly non-convex, continuous on $\bar{\setX}$ and continuously differentiable on $\setX$;
\item Problem \eqref{eq:Opt} admits a global solution. We let $f_{\min}(\setX)=\min\{f(x)\vert x\in\bar{\setX}\}$.
\end{enumerate}
\end{assumption}
%Note that $f$ is not assumed to be globally differentiable, but only over the relative interior of the feasible set. Problem \eqref{eq:Opt} contains many important classes of optimization problems as special cases. We summarize the three most-important ones below.
%%%%%%%NLP%%%%%%%%%%%%%%%
\begin{example}[NLP with non-negativity constraints]
\label{ex:NN}
For $\setE=\Rn$ and $\bar{\setK}\equiv \bar{\setK}_{\text{NN}}=\Rn_{+}$ we recover non-linear programming problems with linear equality constraints and non-negativity constraints: $\bar{\setX}=\{x\in\Rn\vert \bA x=b,\text{ and } x_{i}\geq 0\text{ for all }i=1,\ldots,n\}.$
\close
\end{example}
%%%%%%%%%SOC%%%%%
\begin{example}[Optimization over the Second-Order Cone]
\label{ex:SOC}
Consider $\setE=\R^{n+1}$ and $\bar{\setK}\equiv \bar{\setK}_{\text{SOC}}=\{x=(x_{0},\underline{x})\in\R\times\R^{n-1}\vert x_{0}\geq\norm{\underline{x}}_2\}$, the second-order cone (SOC). In this case problem \eqref{eq:Opt} becomes a non-linear second-order conic optimization problem. Such problems have a huge number of applications, including energy systems \cite{Mol19}, network localization \cite{Tse07}, among many others \cite{AliGol03}. \close
\end{example}
%%%%%%%%%%%SDP%%%%%%%%%
\begin{example}[Semi-definite programming]
\label{ex:SDP}
If $\setE=\symm^{n}$ is the space of real symmetric $n\times n$ matrices and $\bar{\setK}\equiv\bar{\setK}_{\text{SDP}}=\symm^{n}_{+}$ is the cone of positive semi-definite matrices, we obtain a non-linear semi-definite programming problem. Endow this space with the standard inner product $\inner{a,b}=\tr(ab)$. In this case, the linear operator $\bA$ assigns a matrix $x\in\symm^{n}$ to a vector $\bA x = [\inner{a_{1},x}, \ldots, \inner{a_{m},x}]^\top$.
%$\bA x=\left(\begin{array}{c} \inner{a_{1},x}\\
%\vdots\\
%\inner{a_{m},x}
%\end{array}\right)$. 
Such mathematical programs have received enormous attention due to the large number of applications in control theory, combinatorial optimization and engineering \cite{LauRen05,De-Klerk:2006aa,BenNem01}. \close
\end{example}
%%%%%%%%%%%%%%%%%%%%%%%%%%%
\subsection{Motivating applications}
\paragraph{Statistical estimation with non-convex regularization}
An important instance of \eqref{eq:Opt} is the composite optimization problem
\begin{equation}\label{eq:composite}
\min_{x}f(x)=\ell(x)+\lambda\sum_{i=1}^{n}\varphi(x_{i}^{p})\quad \text{s.t.: } x\in\bar{\setK}_{\text{NN}},
\end{equation}
where $\ell:\Rn\to\R$ is a smooth data fidelity function, $\varphi:\R\to\R$ is a convex function, $p\in(0,1)$, and $\lambda>0$ is a regularization parameter. A common use of this problem formulation is the regularized empirical risk-minimization problem in high-dimensional statistics, or the variational regularization technique in inverse problems. Common specifications for the regularizing function are $\varphi(s)=s$, or $\varphi(s)=s^{2/p}$. In the first case, we obtain 
$\sum_{i=1}^{n}\varphi(x^{p}_{i})=\sum_{i=1}^{n}x_{i}^{p}=\norm{x}^{p}_{p}$ on $\setK_{\text{NN}}$, whereas in the second case, we get $\sum_{i=1}^{n}\varphi(x^{p}_{i})=\sum_{i=1}^{n}x_{i}^{2}=\norm{x}^{2}_2$. Note that the first case yields the objective $f$ which is non-convex and non-differentiable at the boundary of the feasible set. It has been reported in imaging sciences that the use of such non-convex and non-differentiable regularizer has advantages in the restoration of piecewise constant images. \cite{BiaChe15} contains a nice survey of studies supporting this observation. Moreover, in variable selection, the $L_{p}$ penalty function with $p\in(0,1)$ owns the oracle property \cite{FanLi01} in statistics, while $L_{1}$ (called the LASSO) does not; problem \eqref{eq:composite} with $p\in(0,1)$ can be used for variable selection at the group and individual variable levels simultaneously, while the very same problem with $p=1$ can only work for individual variable selection \cite{HUANG:2009}. See \cite{GeJiaYe11,Chen:2014wx} for a complexity-theoretic analysis of this problem. 

\paragraph{Low rank matrix recovery}
Similar to the composite minimization problem \eqref{eq:composite}, there are many relevant optimization problems defined on matrix domains $\setE=\symm^{n}$, which are of the similar form, but now defined over a feasible set of the form $\bar{\setX}=\{x\in\setE\vert\bA x=b,x\in\bar{\setK}_{\text{SDP}}\}$. In particular, let us consider the composite model $f(x)=\ell(x)+r(x)$, with smooth loss function $\ell:\setE\to\R$, and with regularizer given in form of a matrix function $r(x)=\sum_{i}\sigma_{i}(x)^{p}$ on $x\in\setK_{\text{SDP}}$, where $p\in(0,1)$ and $\sigma_{i}(x)$ is the $i$-th singular value of the matrix $x$. The resulting optimization problem is a matrix-version of the non-convex regularized problem \eqref{eq:composite}. The use of the non-convex Schatten regularizer has received quite some attention because its favorable properties to promote sparse solutions. In particular, \cite{JiSoSzeYe13} used this approach to solve large-scale network localization problems with a potential reduction method based on a trust-region approach. Another application fitting into the above framework is the task to recover a low rank matrix $X\in\bar{\setK}_{\text{SDP}}$ from measurements $\scrP(x)=d\in\R^{m}$. To solve this problem an attractive formulation is to minimize $f(x)=\norm{\scrP(x)-d}^{2}+r(x)$, with $r(x)$ a $p$-Schatten norm for $p\in(0,1)$. See \cite{SurveyNonconvex} for a recent survey.

\subsection{Challenges and contribution.}
One of the challenges to approach problem \eqref{eq:Opt} algorithmically is to deal with the feasible set $\setL\cap\bar{\setK}$. A projection-based approach faces the computational bottleneck to project onto the intersection of a cone with an affine set, which makes the majority of the existing first-order \cite{GhaLan16,nesterov2020primal-dual,guminov2021combination,carmon2017convex,agarwal2017finding,cartis2019optimality} and second-order \cite{NesPol06,conn2000trust,cartis2012complexity,CurRobSam17,CarGouToi12,birgin2017regularization,CarGouToi18,cartis2019optimality,doikov2021minimizing} methods practically less attractive, as they either are designed for unconstrained problems or use proximal steps in the updates. 
When primal feasibility is not a major concern, augmented Lagrangian algorithms \cite{birgin2017complexity,grapiglia2020complexity,Andreani:2019uf} are an alternative, though they do not always come with complexity guarantees.
%
%do not produce feasible iterates, are not applicable to our problem \eqref{eq:Opt} with general conic constraints, and not always have complexity guarantees. \MS{I don't agree here. Our setting can be handled by augmented Lagrangian ideas. I guess primal feasibility and convergence to a second-order stationary point is the main motivation for the method here. }
%
These observations motivate us to focus on primal barrier-penalty methods that allow to decompose the feasible set and treat $\bar{\setK}$ and $\setL$ separately.
Barrier methods are classical and powerful for convex optimization in the form of interior-point methods, but the results in the non-convex setting are in a sense fragmentary, with many different algorithms existing for different particular instantiations of \eqref{eq:Opt}. In particular, the main focus of barrier methods for non-convex optimization has been on particular cases, such as non-negativity constraints \cite{Ye92,TseBomSch11,HBA-linear,BiaCheYe15,HaeLiuYe18,NeiWr20} and quadratic programming \cite{Ye92,FayLu06,LuYua07}. 
In this paper we develop a flexible and unifying algorithmic framework that is able to provide first- and second-order interior-point algorithms for \eqref{eq:Opt} with potentially non-convex objective functions, potentially non-differentiable on the boundary, and general conic constraints. To the best of our knowledge, our method is the first one providing complexity results for first- and second-order algorithms to reach approximate first- and second-order KKT points, respectively, under such weak assumptions. 
\paragraph{Our approach.}
At the core of our approach is the assumption that the cone $\bar{\setK}$ admits a \emph{logarithmically homogeneous self-concordant barrier} (LHSCB) $h(x)$ (\cite{NesNem94}, cf. Definition \ref{def:LHSCB}), for which we can retrieve information about the function value $h(x)$, the gradient $\nabla h(x)$ and the Hessian $H(x)=\nabla^{2}h(x)$ with relative ease. This is not a very restrictive assumption, since all standard conic restrictions in optimization (i.e. $\bar{\setK}_{\text{NN}},\bar{\setK}_{\text{SOC}}$ and $\bar{\setK}_{\text{SDP}}$) have this property. Using this barrier, our algorithms are designed to reduce the \emph{potential function}
\begin{equation}\label{eq:potential}
F_{\mu}(x)=f(x)+\mu h(x),
\end{equation}
where $\mu>0$ is a (typically) small penalty parameter. By definition, the domain of the potential function $F_{\mu}$ is the interior of the cone $\bar{\setK}$. Therefore, any algorithm designed to reduce the potential will automatically respect the conic constraints, and the satisfaction of the linear constraints $\setL$ can be ensured by choosing search directions from the nullspace of the linear operator $\bA$. Our target is to identify points satisfying approximate necessary first- and second-order optimality conditions for problem \eqref{eq:Opt} expressed in terms of $\eps$-KKT and $(\eps_1,\eps_2)$-2KKT points respectively (cf. Section \ref{sec:Optimality} for a precise definition).

\paragraph{Approaching first-order stationary points.}
To produce a first-order stationary point, we construct a novel gradient-based method, which we call the \emph{adaptive Hessian barrier algorithm} ($\AHBA$, Algorithm \ref{alg:AHBA}). The main computational steps involved in $\AHBA$ is the identification of a search direction and a step size policy, guaranteeing feasibility and sufficient decrease in the potential function value. To find a step direction, we employ a linear model for $F_{\mu}$ regularized by the squared local norm induced by the Hessian of $h$ which is then minimized over the tangent space of the affine set $\setL$. The step-size is adaptively chosen to ensure feasibility and sufficient decrease in the objective function value $f$. For a judiciously chosen value of $\mu$, we prove that this gradient-based method enjoys the upper iteration complexity bound $O(\eps^{-2})$ for reaching an $\eps$-KKT point when a ``descent Lemma'' holds relative to the local norm induced by the Hessian of $h$ (cf. Assumption \ref{ass:gradLip} and Theorem \ref{Th:AHBA_conv} in Section \ref{sec:firstorder}). We then embed $\AHBA$ into a path-following scheme that iteratively reduces the value of $\mu$ making the algorithm parameter-free and any-time convergent with the $O(\eps^{-2})$ complexity.

\paragraph{Approaching second-order stationary points.}
We next move on to derive a second-order method called the \emph{second-order adaptive Hessian barrier algorithm} ($\SAHBA$, Algorithm \ref{alg:SOAHBA}). Under this approach the step direction is determined by a minimization subproblem over the same tangent space. 
But, in this case, the minimized model is composed of the linear model for $F_{\mu}$ augmented by second-order term for $f$ and regularized by the cube of the local norm induced by the Hessian of $h$. 
The regularization parameter is chosen adaptively to allow for potentially larger steps in the areas of small curvature.
For a judiciously chosen value of $\mu$, we establish (see Theorem \ref{Th:SOAHBA_conv}) the worst-case bound $O(\max\{\eps_1^{-3/2},\eps_2^{-3/2}\})$ on the number of iterations for reaching an $(\eps_1,\eps_2)$-2KKT point, under a weaker assumption that the Hessian of $f$ is Lipschitz relative to the local norm induced by the Hessian of $h$ (see Assumption \ref{ass:2ndorder} in Section \ref{sec:secondorder} for a precise definition). We then propose a path-following version of $\SAHBA$ that iteratively reduces the value of $\mu$ making the algorithm parameter-free and any-time convergent with $O(\max\{\eps_1^{-3/2},\eps_2^{-3/2}\})$ complexity.

\subsection{Related work}
\label{sec:related}
To the best of our knowledge, $\AHBA$ and $\SAHBA$ are the first interior-point algorithms that achieve such complexity bounds universally for the general non-convex problem template \eqref{eq:Opt}. Our closest algorithmic and complexity-theoretic competitors are \cite{HaeLiuYe18,NeiWr20}. Both papers focus on the special case of non-negativity constraints as in Example \ref{ex:NN} and fix $\mu$ before the start of the algorithm based on the desired accuracy $\eps$, which may require some hyperparameter tuning in practice and may not work if the desired accuracy is not yet known. Interestingly, for the special case $\bar{\setK}=\bar{\setK}_{\text{NN}}$, our general algorithms provide stronger results under weaker assumptions, compared to first- and second-order methods in \cite{HaeLiuYe18} and first-order implementation of the second-order method in \cite{NeiWr20} (cf. Sections \ref{sec:FO_discussion} and \ref{sec:SO_discussion}).% As said above, augmented Lagrangian algorithms do not guarantee feasibility of the obtained approximate solutions, are designed  for problems either with non-negativity constraints \cite{birgin2017complexity,grapiglia2020complexity,xie2019complexity} or with second-order/other symmetric cone constraints (yet, without complexity analysis) \cite{Andreani:2019uf}, unlike our setting of general cones, including non-symmetric. 
%%%%%%%%%%%%%%%%%%%%%%%%
\paragraph{First-order methods.} In the unconstrained setting, when the gradient is Lipschitz continuous, the standard gradient descent achieves the lower iteration complexity bound $O(\eps^{-2})$ to find a first-order $\eps$-stationary point $\hat{x}$ such that $\norm{\nabla f(\hat{x})}_2 \leqslant\eps$  \cite{Nes18,CarDucHinSid19b,CarDucHinSid19}. 
Notably, despite problem \eqref{eq:Opt} has non-trivial constraints, our bound for $\AHBA$ matches this bound.
%Since constrained optimization is potentially more difficult than unconstrained and when $\setX$ is bounded our main assumption subsumes the standard  Lipschitz-gradient assumption (see Remark \ref{rem:bounded1}), we can consider our first-order complexity results as in some sense optimal. 
The original motivation for our work comes from the paper \cite{HBA-linear} on Hessian Barrier Algorithms, which in turn was strongly influenced by the continuous-time techniques of \cite{ABB04,BolTeb03}. Our results include second-order method and general conic constraints and hold far beyond the realm of \cite{HBA-linear}, where the complexity result is proved only for first-order method in the setting of non-negativity constraints and quadratic objective.
%%%%%%%%%%%%%%%%%%%%%%%%%%%%%%%%%%%%%%%%%%%%%
\paragraph{Second-order methods.}
In unconstrained optimization with Lipschitz continuous Hessian, cubic-regularized Newton methods \cite{Gri81,NesPol06} and second-order trust region algorithms \cite{conn2000trust,CarGouToi12,CurRobSam17} achieve the lower iteration complexity bound $O(\max\{\eps_1^{-3/2},\eps_2^{-3/2}\})$ \cite{CarDucHinSid19b,CarDucHinSid19} to find  a second-order $(\eps_1,\eps_2)$-stationary point; I.e. a point $\hat{x}$ satisfying $\norm{\nabla f(\hat{x})}\leq \eps_1$ and $\lambda_{\min} \left(\nabla^2 f(\hat{x})\right) \geq - \sqrt{\eps_2}$, where $\lambda_{\min}(\cdot)$ denotes the minimal eigenvalue of a matrix\footnote{A number of works, e.g. \cite{CarGouToi12,NeiWr20}, consider an $(\eps_1,\eps_2)$-stationary point defined as $\hat{x}$ such that $\|\nabla f(\hat{x}) \|_2 \leq \eps_1$ and $\lambda_{\min} \left(\nabla^2 f(\hat{x})\right) \geq - \eps_2$ and the corresponding complexity $O(\max\{\eps_1^{-3/2},\eps_2^{-3}\})$. Our definition and complexity bound are the same up to redefinition of $\eps_2$.}. 
Notably, despite problem \eqref{eq:Opt} has non-trivial constraints, our bound for $\SAHBA$ matches this bound.
%Since constrained optimization is potentially more difficult than unconstrained and when $\setX$ is bounded our main assumption subsumes the standard Lipschitz-Hessian assumption (see Remark \ref{Rm:Hess_Lip}), we can consider our second-order complexity results as in some sense optimal. 
The existing literature on non-convex problems with non-linear constraints either consider only equality constraints \cite{curtis2018complexity}, or only inequality constraints \cite{hinder2018worst-case}, or both, but require projection \cite{cartis2019optimality}. Moreover, they do not consider general conic constraints as in \eqref{eq:Opt}.
%%%%%%%%%%%%%%%%%%%%%%%%%%%%%%%%%%%%%%%%%%%%%
\paragraph{Approximate optimality conditions.}
\cite{BiaCheYe15} consider box-constrained minimization of the same objective as in \eqref{eq:composite} and propose a notion of $\eps$ scaled KKT points. 
Their definition is tailored to the geometry of the optimization problem, mimicking the complementarity slackness condition of the classical KKT theorem for the non-negative orthant. 
In particular, their first-order condition consist of feasibility of $x$ along with a scaled gradient condition.
\cite{HaeLiuYe18,NeiWr20} point out that without additional assumptions on $f$, points that satisfy the scaled gradient condition may not approach KKT points as $\eps$ decreases. 
Thus, \cite{HaeLiuYe18,NeiWr20}, provide alternative notions of approximate first- and second-order KKT conditions for the setting of non-negativity constraints. 
Inspired by \cite{HaeLiuYe18}, we define the corresponding notions for general cones. Our first-order conditions turn out to be stronger than that of  \cite{HaeLiuYe18,NeiWr20} and the second-order condition is equivalent to theirs in the particular case of non-negativity constraints (see Sections \ref{sec:FO_discussion} and \ref{sec:SO_discussion}). The proof that our algorithms are guaranteed to find such approximate KKT points requires some fine analysis exploiting the structural properties of logarithmically homogeneous barriers attached to the cone $\setK$, which, to the best of our knowledge, appear to be novel. 
 
%\subsection{Outline}
%The rest of the paper is organized as follows. Section \ref{sec:prelims} is technical and collects the relevant definitions, facts, and estimates provided by the theory of \ac{SC}  barriers. Section \ref{sec:Optimality} describes our notion of approximate first- and second-order KKT points that capture general cones in \eqref{eq:Opt}. Sections \ref{sec:firstorder} and \ref{sec:secondorder} introduce and analyze our new first-order and second-order interior-point methods together with their path-following versions. A detailed assessment of their iteration complexity, relative to the approximate KKT points, is provided, as well as discussions on relation to previous works.   

\subsection{Notation}
 In what follows $\setE$ denotes a finite-dimensional real vector space, and $\setE^{\ast}$ the dual space, which is formed by all linear functions on $\setE$. The value of $s\in\setE^{\ast}$ at $x\in\setE$ is denoted by $\inner{s,x}$. In the particular case where $\setE=\Rn$, we have $\setE=\setE^{\ast}$. Important elements of the dual space are gradients of differentiable functions $f:\setE\to\R$, denoted as $\nabla f(x)\in\setE^{\ast}$. For an operator $\bH:\setE\to\setE^{\ast}$, denote by $\bH^{\ast}$ is adjoint operator, defined by the identity 
 \[
(\forall u,v\in\setE): \qquad \inner{\bH u,v}:=\inner{u,\bH^{\ast}v}.
 \]
Thus, $\bH^{\ast}:\setE\to\setE^{\ast}$. It is called self-adjoint if $\bH=\bH^{\ast}$. We use $\lambda_{\max}(\bH)/\lambda_{\min}(\bH)$ to denote the maximum/minimum eigenvalue of such operators. Important examples of such self-adjoint operators are Hessians of twice differentiable functions $f:\setE\to\R$: 
\[
(\forall u,v\in\setE):\qquad \inner{\nabla^{2}f(x)u,v}=\inner{u,\nabla^{2}f(x)v}.
\]
Operator $\bH:\setE\to\setE^{\ast}$ is positive semi-definite if $\inner{\bH u,u}\geq 0$ for all $ u\in\setE$. If the inequality is always strict for non-zero $u$, then $\bH$ is called positive definite. These attributes are denoted as $\bH\succeq 0$ and $\bH\succ 0$, respectively.
By fixing a positive definite self-adjoint operator $\bH:\setE\to\setE^{\ast}$, we can define the following Euclidean norms
\[
\norm{u}=\inner{\bH u,u}^{1/2},\quad \norm{s}^{\ast}=\inner{s,\bH^{-1}s}^{1/2}\quad u\in\setE,s\in\setE^{\ast}.
\]
If $\setE=\Rn$, then $\bH$ is usually taken as the identity matrix $\bH=\bI$. 
The directional derivative of function $f$ is defined in the usual way:
\[
Df(x)[v]:=\lim_{\eps\to 0+} \frac{1}{\eps}[f(x+\eps v)-f(x)].
\]
More generally, for $v_{1},\ldots,v_{p}\in\setE$, we define $D^{p}f(x)[v_{1},\ldots,v_{p}]$ the $p$-th directional derivative at $x$ along directions $v_{i}\in\setE$. In that way we define $\nabla f(x)\in\setE^{\ast}$ by $Df(x)[u]=\inner{\nabla f(x),u}$ and the Hessian $\nabla^{2}f(x):\setE\to\setE^{\ast}$ by $\inner{\nabla^{2}f(x)u,v}=D^{2}f(x)[u,v]$. 
We denote $\setL_{0}=\{ v \in \setE \vert \bA v = 0\}$ the tangent space associated with the linear subspace $\setL\subset\setE$.

%%%%%%%%%%%%%%%%%%%%%%%%%%%%%%%%%%%%%%%%%
%%%%%%%% PRELIMS%%%%%%%%%%%%%%%%%%%%%%%%%%%
%%%%%%%%%%%%%%%%%%%%%%%%%%%
\section{Prelminiaries}
\label{sec:prelims}
%----------------------------------------------------------------------
%%% Prelims
%----------------------------------------------------------------------
% !TEX root = ./HBAConicMain.tex
%%%%%%%%%%%
%The purpose of this section and Appendix \ref{app:barrier} is to introduce all the concepts and notations we need in order to define and analyze our algorithms. Readers who are more interested in the algorithmic details can skip this section for the moment and continue directly with section \ref{sec:firstorder}, while referring back to this section whenever needed.
%
\subsection{Cones and their self-concordant barriers}
Let $\bar{\setK}\subset\setE$ be a \emph{regular cone}: $\bar{\setK}$ is closed convex, solid and pointed (i.e. contains no lines). 
We assume that $\setK:=\Int(\bar{\setK}) \neq \emptyset$, where $\Int(\bar{\setK})$ is the interior of $\bar{\setK}$. Any such cone admits a self-concordant logarithmically homogeneous barrier $h(x)$ with finite parameter value $\nu$ \cite{NesNem94}.

\begin{definition}
\label{def:LHSCB}
A function $h:\bar{\setK}\to(-\infty,\infty]$ with $\dom h=\setK$ is called a $\nu$-\emph{logarithmically homogeneous self-concordant barrier} ($\nu$-LHSCB) for the cone $\bar{\setK}$ if:
\begin{itemize}
\item[(a)] $h$ is a $\nu$-\emph{self-concordant barrier} for $\bar{\setK}$, i.e., for all $x \in \setK$ and $u\in \setE$
\begin{align}
\label{eq:SC}
&\abs{D^{3}h(x)[u,u,u]}\leq 2D^{2}h(x)[u,u]^{3/2},\text{ and } \\
&\sup_{u\in\setE}\abs{2 Dh(x)[u]-D^{2}h(x)[u,u]}\leq \nu.
\end{align}
\item[(b)] $h$ is \emph{logarithmically homogeneous:} 
\[
h(tx)=h(x)-\nu\ln(t)\qquad \forall x\in \setK,t>0.
\]
\end{itemize}
We denote the set of $\nu$-logarithmically homogeneous barriers by $\scrH_{\nu}(\setK)$.
\end{definition}
%Condition \eqref{eq:SC} is saying that $h$ is a standard \ac{SC} barrier on $\bar{\setK}$. 
Given $h\in\scrH_{\nu}(\setK)$, from \cite[Thm 5.1.3]{Nes18} we know that for any $\bar{x}\in\bd(\bar{\setK})$, any sequence $(x_{k})_{k \geq 0}$ with $x_{k}\in\ \setK$ and $\lim_{k\to\infty}x_{k}=\bar{x}$ satisfies $\lim_{k\to\infty}h(x_{k})=+\infty$. For a pointed cone $\bar{\setK}$, we have $\nu\geq 1$ and the Hessian $H(x)\eqdef\nabla^{2}h(x):\setE\to\setE^{\ast}$ is a positive definite linear operator defined by $\inner{H(x)u,v}\eqdef D^{2}h(x)[u,v]$ for all $u,v\in\setE$, see  \cite[Thm. 5.1.6]{Nes18}. The Hessian gives rise to a \emph{local norm} 
\begin{equation}\label{eq:localnorm}
(\forall x\in \setK)(\forall u\in\setE):\quad \norm{u}_{x}=\inner{H(x)u,u}^{1/2}.
\end{equation}
We also define a \emph{dual norm} on $\setE^{\ast}$ as 
\begin{equation}\label{eq:dualnorm}
(\forall x\in\setK)(\forall s\in\setE^{\ast}):\quad \norm{s}_{x}^{\ast}= \inner{[H(x)]^{-1}s,s}^{1/2}.
\end{equation}
%Note that 
%\begin{equation}\label{eq:norm_represent}
%\norm{u}_{x} = \norm{[H(x)]^{1/2} u}, \text{and }\norm{s}_{x}^{\ast}= \norm{[H(x)]^{-1/2} s}\quad \forall x\in \setK, u\in\setE, s\in\setE^{\ast}. 
%\end{equation}
The \emph{Dikin ellipsoid} is defined as the open set
$\scrW(x;r)\eqdef\{u\in\setE\vert\;\norm{u-x}_{x}<r\}, r>0.$ The usage of the local norm adapts the unit ball to the local geometry of the set $\setK$. Indeed, the following classical result is key to the development of our methods. 
\begin{lemma}[Theorem 5.1.5 \cite{Nes18}]
\label{lem:Dikin}
For all $x\in\setK$ we have $\scrW(x;1)\subseteq\setK$.
\end{lemma}
\begin{proposition}[{Theorem 5.1.9 \cite{Nes18}}]
\label{prop:SCF_upper_bound}
Let $h\in\scrH_{\nu}(\setK)$, $x\in\dom h$, and a fixed direction $d \in \setE$. For all $t \in [0,\frac{1}{\norm{d}_x})$, with the convention that $\frac{1}{\norm{d}_{x}}=+\infty$ if $\norm{d}_x=0$, we have:
\begin{equation}
\label{eq:SCF_upper_bound} 
h(x + t d) \leq h(x) + t\inner{\nabla h(x),d} + t^2 \norm{d}_{x}^2 \omega(t \norm{d}_{x}),
\end{equation}
where $\omega(t)=\frac{-t-\ln(1-t)}{t^2}$.
\end{proposition}
\noindent
We will also use the following inequality for the function $\omega(t)$ \cite[Lemma 5.1.5]{Nes18}:
\begin{equation}\label{eq:omega_upper_bound}
\omega(t) \leq \frac{1}{2(1-t)}, \; t \in [0,1).
\end{equation}
We close this section with important examples of conic domains to which our method can be applied. 
\begin{example}[The exponential cone]
Consider the exponential cone studied by \cite{Chares:2009wb} defined as 
\begin{equation}
\setK_{\exp}=\{x\in\R^{3}\vert x_{1}\geq x_{2}e^{x_{3}/x_{2}},x_{2}>0\}
\end{equation}
with closure $\bar{\setK}_{\exp}=\cl(\setK_{\exp})$. This set admits a $3$-LHSB 
\[
h(x_{1},x_{2},x_{3})\eqdef -\ln(x_{2}\ln(x_{1}/x_{2})-x_{3})-\ln(x_{1})-\ln(x_{2})\in\scrH_{3}(\setK_{\exp}).
\]
We remark that this cone is not self-dual (cf. Definition \ref{def:SymCone}), but 
\[
G\bar{\setK}_{\exp}=\bar{\setK}^{\ast}_{\exp}=\cl\left(\{y\in\R^{3}\vert y_{1}\geq -y_{3}e^{y_{2}/y_{3}-1},y_{1}>0,y_{3}<0\}\right),
\]
under the linear transformation 
$
G=\left[\begin{array}{ccc} 
1/e & 0 & 0 \\
0 & 0 & -1 \\
0 & -1 & 0 
\end{array}\right].
$
There are many convex sets that can be represented using the exponential cone; We list some example below, but refer to the PhD thesis \cite{Chares:2009wb} for further details. 
\begin{itemize}
\item Exponential: $\{(t,u)\vert t\geq e^{u}\}\iff (t,1,u)\in\bar{\setK}_{\exp}$;
\item Logarithm: $\{(t,u)\vert t\leq \ln(u)\}\iff (u,1,t)\in\bar{\setK}_{\exp}$;
\item Entropy: $t\leq -u\ln(u)\iff t\leq u\ln(1/u)\iff (1,u,t)\in\bar{\setK}_{\exp}$;
\item Relative Entropy: $t\geq u\log(u/w)\iff (w,u,t)\in\bar{\setK}_{\exp}$;
\item Softplus function: $t\geq\ln(1+e^{u})\iff a+b\leq 1,(a,1,u-t)\in\bar{\setK}_{\exp},(b,1,-t)\in\bar{\setK}_{\exp}$.  
\end{itemize}
\close
\end{example}
%The next three examples show that our method can deal with the most important conic constraints in optimization. 
\begin{example}[Non-negativity constraints]
\label{Ex:NLP}
For $\setE=\Rn$ and $\bar{\setK}=\bar{\setK}_{\text{NN}}$, we define the log-barrier $h(x)=-\sum_{i=1}^{n}\ln(x_{i})$ for all $x\in\setK_{\text{NN}}=\Rn_{++}$. It is readily seen that $h\in\scrH_{n}(\setK)$. \close
\end{example}
%%%%%%%%%%%% SOC %%%%%%%%%%%%%%%%%55
\begin{example}[SOC constraints]
Let $\setE=\R^{n+1}$ and $\bar{\setK}=\bar{\setK}_{\text{SOC}}$ defined in Example \ref{ex:SOC}. For $x=(x_{0},\underline{x})\in\bar{\setK}_{\text{SOC}}$, we define the barrier $h(x)=-\ln(x_{0}^{2}-\underline{x}^{\top}\underline{x})$. It is well known that $h\in\scrH_{2}(\setK_{\text{SOC}})$ \cite{NesNem94}. 
\close
\end{example}
%%%%%%%%%%% SDP %%%%%%%%%%%%%%
\begin{example}[SDP constraints]
Let $\setE=\symm^{n}$ and $\bar{\setK}=\bar{\setK}_{\text{SDP}}$, defined in Example \ref{ex:SDP}. Consider the barrier $h(x)=-\ln\det(x)$. It is well known that $h\in\scrH_{n}(\setK_{\text{SDP}})$.
\close
\end{example}
%%%%%%%%%%%%%%%%%%%%%%%%%

\subsection{Exploiting the structure of Symmetric Cones}
%In the classical literature on interior-point methods, 
Nesterov and Todd \cite{NesTod97} introduced \emph{self-scaled barriers}, which later have been realized as LHSCB's for \emph{symmetric cones}. Such barriers are nowadays key to define primal-dual interior point methods for convex problems with potentially larger step sizes. Our method can also exploit the additional properties of self-scaled barriers, leading to potentially larger step sizes and faster convergence in our non-convex setting as well. For a given closed convex nonempty cone $\bar{\setK}$, its \emph{dual cone} is the closed convex and nonempty conce $\bar{\setK}^{\ast}$ defined as $\bar{\setK}^{\ast}\eqdef\{s\in\setE^{\ast}\vert\inner{s,x}\geq 0\;\forall x\in\bar{\setK}\}$. If $h\in\scrH_{\nu}(\setK)$, then the \emph{dual barrier} is defined $h_{\ast}(s)\eqdef\sup_{x\in\setK}\{\inner{-s,x}-h(x)\}$ for $s\in\bar{\setK}^{\ast}$.
 %%%%%%%
\begin{definition}\label{def:SymCone}
An open convex cone is said to be \emph{self-dual} if $\setK^{\ast}=\setK$. $\setK$ is homogeneous if for all $x,y\in\setK$ there exists a linear bijection $G:\setE\to\setE$ such that $Gx=y$ and $G\setK=\setK$. An open convex cone $\setK$ is called \emph{symmetric} if it is self-dual and homogeneous. 
\end{definition}
%%%%%%%%%%%%
The class of symmetric cones can be characterized within the language of Euclidean Jordan algebras \cite{FayLu06,Fay08,FarKor94,Schmieta2003}. For optimization, the three symmetric cones of most relevance are $\bar{\setK}_{\text{NN}},\bar{\setK}_{\text{SOC}}$ and $\bar{\setK}_{\text{SDP}}$. %On symmetric cones, \cite{NesTod97} defined a long-step primal-dual method which builds upon a subset of LHSCB's, defined as follows.
%%%%%SSB%%%%%%
\begin{definition}[\cite{NesTod97}]
$h\in\scrH_{\nu}(\setK)$ is a $\nu$-\emph{self-scaled barrier} ($\nu$-SSB) if for all $x,w\in\setK$ we have $H(w)x\in\setK$ and $h_{\ast}(H(w)x)=h(x)-2h(w)-\nu$. Let $\scrB_{\nu}(\setK)$ denote the class of $\nu$-SSBs.  
\end{definition}
We emphasize that $\scrB_{\nu}(\setK)\subset\scrH_{\nu}(\setK)$. \cite{HauGul02} showed that every symmetric cone admits a $\nu$-SSB for some $\nu\geq 1$, while a characterization of the barrier parameter $\nu$ has been obtained in \cite{GulTun98}. The main advantage of working with SSB's instead of LHSCB's is that we can make potentially longer steps in the interior of the cone $\setK$ towards the direction of its boundary. Let $x \in\setK$ and $d \in \setE$. Denote
\begin{equation}
\label{eq:sigma_def}
\sigma_x(d):= (\sup \{ t: x -t d \in \setK\})^{-1}
\end{equation}
Since $\scrW(x;1) \subseteq\setK$ for all $x\in\setK$, we have that $\sigma_x(d) \leq \norm{d}_{x}$ and $\sigma_{x}(-d)\leq\norm{d}_{x}$ for all $d\in\setE$. Therefore $[0,\frac{1}{\norm{d}_{x}})\subseteq [0,\frac{1}{\sigma_{x}(d)})$. Hence, if the scalar quantity $\sigma_{x}(d)$ can be computed efficiently, it would allow us to make a larger step without violating feasibility.%The next examples illustrate how $\sigma_{x}(d)$ looks like in practically relevant situations.
%%%%%%
\begin{example}
For $\bar{\setK}=\bar{\setK}_{\text{NN}}$, to guarantee $x-td\in\setK_{\text{NN}}$, we need $x_{i}-t d_{i}>0$ for all $i\in\{1,\ldots,n\}$. Hence, if $d_{i}\leq 0$, this is satisfied for all $t\geq 0$. If $d_{i}>0$, we obtain the restriction $t\leq \frac{x_{i}}{d_{i}}$. Hence, it follows that $\sigma_{x}(-d)=\max\{\frac{d_{i}}{x_{i}}:d_{i}>0\}$. 
\close
\end{example}
%%%%%
\begin{example}
For $\bar{\setK}=\bar{\setK}_{\text{SDP}}$, we see that $x-td\succ 0$ if and only if $\Id\succ t x^{-1/2}dx^{-1/2}$, where $\Id$ is the identity matrix. Hence, if $\lambda_{\max}(x^{-1/2}dx^{-1/2})>0$, then $t<\frac{1}{\lambda_{\max}(x^{-1/2}dx^{-1/2})}$. Thus, $\sigma_{x}(d)=\max\{\lambda_{\max}(x^{-1/2}dx^{-1/2}),0\}$.
\close
\end{example}
\noindent
We will also need the analogous result to Proposition \ref{prop:SCF_upper_bound} for barriers $h\in\scrB_{\nu}(\setK)$:
\begin{proposition}[{Theorem 4.2 \cite{NesTod97}}]
\label{prop:SSB_upper_bound}
Let $h\in\scrB_{\nu}(\setK)$ and $x \in\setK$. Let $d \in \setE$ be such that $\sigma_x(-d)>0$. Then, for all $t \in [0,\frac{1}{\sigma_x(-d)})$, we have:
\begin{equation}
\label{eq:SSB_upper_bound} 
h(x + t d) \leq h(x) + t \inner{\nabla h(x),d} + t^2 \norm{d}_x^2 \omega(t \sigma_x(-d)).
\end{equation}
\end{proposition}

\subsection{Unified Notation}
Our algorithms work on any conic domain on which we can efficiently evaluate a $\nu$-LHSCB. We formalize this in the following assumption
\begin{assumption}\label{ass:barrier}
$\bar{\setK}$ is a regular cone admitting an efficient barrier setup $h\in\scrH_{\nu}(\setK)$. By this we mean that at a given query point $x\in\setK$, we can construct an oracle that returns to us information about the values $h(x),\nabla h(x)$ and $H(x)=\nabla^{2}h(x)$, with low computational efforts. 
\end{assumption}
Given the potential advantages when working on symmetric cones, it is useful to develop a unified notation handling both cases at the same time. Note that when $h \in \scrB_{\nu}(\setK)$, we have the flexibility to treat $h$ either as $h \in\scrH_{\nu}(\setK)$ or as $h \in\scrB_{\nu}(\setK)$. To unify the presentation, we define
\begin{equation}\label{eq:zeta}
(\forall (x,d)\in\setX\times \setE):\;\zeta(x,d)=\left\{\begin{array}{ll} 
\norm{d}_{x} & \text{if }h\in\scrH_{\nu}(\setK)\setminus\scrB_{\nu}(\setK),\\
\sigma_{x}(-d) & \text{if }h\in\scrB_{\nu}(\setK).
\end{array}
\right.  
\end{equation}
Note that
\begin{align}
&(\forall (x,d)\in\setX\times \setE):\; \zeta(x,d)\leq\norm{d}_{x}, \label{eq:boundzeta}\\
&(\forall (x,d)\in\setX\times \setE)(\forall t \in [0,\frac{1}{\zeta(x,d)})):\; x + td \in\setK .\label{eq:step_length_zeta}
\end{align}
Finally, for the Bregman divergence $D_{h}(u,x):= h(u)-h(x)-\inner{\nabla h(x),u-x}$ defined for $x,u\in\setK$, Proposition \ref{prop:SCF_upper_bound}, Proposition \ref{prop:SSB_upper_bound} together with eq. \eqref{eq:zeta}, give us the one-and-for-all Bregman bound
\begin{equation}\label{eq:Dbound}
D_{h}(x+t d,x)=h(x+t d)-h(x)-\inner{\nabla h(x),td}\leq  t^2 \norm{d}_{x}^2 \omega(t \zeta(x,d))
\end{equation}
valid for all $(x,d)\in\feas\times \setE$ and $t \in [0,\frac{1}{\zeta(x,d)})$.

\section{Approximate optimality conditions}
\label{sec:Optimality}
The next definition specifies our notion of an approximate first-order KKT point for problem \eqref{eq:Opt}.
%%%%%%%%%%%
\begin{definition}\label{def:eps_KKT}
%Let Assumptions \ref{ass:1} and \ref{ass:full_rank} hold. 
Given $\eps \geq 0$, we call a triple $(\bar{x},\bar{y},\bar{s})\in\setE\times\R^{m}\times\setE^{\ast}$ an $\eps$-KKT point for problem \eqref{eq:Opt} if
\begin{align}
& \bA\bar{x}=b, \bar{x}\in\setK, \bar{s} \in\setK^{\ast}, \label{eq:eps_optim_equality_cones} \\
&\|\nabla f(\bar{x}) - \bA^{\ast}\bar{y}-\bar{s} \|\leq \eps,  \label{eq:eps_optim_grad} \\
& \inner{\bar{s},\bar{x}} \leq \eps. \label{eq:eps_optim_complem} 
\end{align}
\end{definition}
%%%%%%%%%%%%%%
To justify this definition, let $x^{\ast}$ be a local solution of problem \eqref{eq:Opt}. Then, for $\delta>0$ sufficiently small, the point $x^{\ast}$ is the unique global solution to the perturbed optimization problem with ball restriction $\overline{\ball(x^{\ast};\delta)}\eqdef\{x\in\setE\vert\; \norm{x-x^{\ast}}\leq\delta\}$:
\begin{equation}
\label{eq:limit_optimality_proof_0}
	\min_{x\in\setX\cap\overline{\ball(x^{\ast};\delta)}} f(x) + \frac{1}{4}\norm{x-x^{*}}^4.
\end{equation}
Next, using the barrier $h\in\scrH_{\nu}(\setK)$, we absorb the constraint $x\in\setK$ in the penalty $\mu_k h(x)$, where $\mu_k> 0$, $\mu_{k}\downarrow 0$ is a given sequence. This leads to the barrier formulation
\begin{equation}\label{eq:limit_optimality_proof_1}
\min_{x\in\setX\cap\overline{\ball(x^{\ast};\delta)}}\varphi_{k}(x)\eqdef F_{\mu_{k}}(x)+\frac{1}{4}\norm{x-x^{\ast}}^{4},\quad F_{\mu_{k}}(x)= f(x)+\mu_{k}h(x).
\end{equation}
From the classical theory of interior penalty methods \cite{FiacMcCo68}, it is known that a global solution $x^k$ exists for this problem for all $k$ and that cluster points of $x^k$ are global solutions of \eqref{eq:limit_optimality_proof_0}. Clearly, $x^{k}\in \setX\cap\overline{\ball(x^{\ast},\delta)}$ for all $k$ and $x^{k}\to x^{\ast}$. Setting $s^{k}= -\mu_{k}\nabla h(x^{k})$, which belongs to $\setK^{\ast}$ by eq. \eqref{eq:relations_1}, and exploiting the properties of the barrier function $h\in\scrH_{\nu}(\setK)$, we see that 
$
\inner{s^{k},x^{k}}=-\mu_{k}\inner{\nabla h(x^{k}),x^{k}} \stackrel{\eqref{eq:log_hom_scb_hess_prop}}{=}\mu_{k}\nu.
%=-\mu_{k}\inner{[H(x^{k})]^{-1/2}\nabla h(x^{k}),[H(x^{k})]^{1/2}x^{k}}\\
%&\overset{\eqref{eq:log_hom_scb_hess_prop}}{=}\mu_{k}\inner{[H(x^{k})]^{1/2}x^{k},[H(x^{k})]^{1/2}x^{k}}\\
%&=\mu_{k}\norm{x^{k}}^{2}_{x^{k}}=\mu_{k}\nu.
$
Consequently, $\lim_{k\to\infty}\inner{s^{k},x^{k}}=0$. Since $x^{k}\to x^{\ast}$, the restriction $x^{k}\in\overline{\ball(x^{\ast};\delta)}$ will automatically hold for $k$ sufficiently large. By the full-rank assumption, the first-order optimality conditions of problem \eqref{eq:limit_optimality_proof_1} reads as 
\[
\nabla f(x^{k})-\bA^{\ast}y^{k}-s^{k}-\norm{x^{k}-x^{\ast}}^{2}\cdot(x^{k}-x^{\ast})=0,
\]
for all $k$ large enough. Hence, setting $\delta\leq\eps^{1/3}$, $\mu_k \leq \eps/\nu$, and $\bar{x}=x^{k},\bar{s}=s^{k},\bar{y}=y^{k}$, we obtain a triple satisfying conditions \eqref{eq:eps_optim_equality_cones}-\eqref{eq:eps_optim_complem}.
%%%%%%%%%%%%%%%%%%%%%
%\begin{remark}
%For $\eps=0$ the conditions of Definition \ref{def:eps_KKT} reduce to the exact first-order optimality conditions of \cite{FayLu06}.
%\end{remark}
%%%%%%%%%%%%%%%%%%%%%%%%%
\\

Assuming twice continuous differentiability of $f$ on $\setX$, our notion of an approximate second-order KKT point for problem \eqref{eq:Opt} is defined as follows.
\begin{definition}\label{def:eps_SOKKT}
%Let Assumptions \ref{ass:1}, and \ref{ass:full_rank} hold. 
Given $\eps_1,\eps_2 \geq 0$, we call a triple $(\bar{x},\bar{y},\bar{s})\in\setE\times\R^{m}\times\setE^{\ast}$ an $(\eps_1,\eps_2)$-2KKT point for problem \eqref{eq:Opt} if
\begin{align}
& \bA\bar{x}=b, \bar{x}\in\setK, \bar{s} \in\setK^{\ast}, \label{eq:eps_SO_optim_equality_cones} \\
&\|\nabla f(\bar{x}) - \bA^{\ast}\bar{y}-\bar{s} \|\leq \eps_1,  \label{eq:eps_SO_optim_grad} \\
&\inner{\bar{s},\bar{x}} \leq \eps_1, \label{eq:eps_SO_optim_complem} \\
& \nabla^2f(\bar{x}) + \sqrt{\eps_2} H(\bar{x}) \succeq 0 \;\; \text{on} \;\; \setL_0.   \label{eq:eps_SO_optim_SO}
\end{align}
\end{definition}
The first three conditions are the same as for the $\eps$-KKT point. 
The last one can be justified as follows. Using the full-rank condition, the second-order optimality condition for problem \eqref{eq:limit_optimality_proof_1} says that $x^{k}$ satisfies 
\[
\inner{(\nabla^{2}f(x^{k})+\mu_{k}H(x^{k}))d,d} \geq -2\inner{x^{k}-x^{\ast},d}^{2} - \norm{x^{k}-x^{\ast}}^2\norm{d}^2\geq -3\delta^{2}\norm{d}_{2}^{2}\qquad\forall d\in\setL_{0}.
\]
%says that $x^{k}$ is an isolated local solution to problem \eqref{eq:limit_optimality_proof_1} for $k$ sufficiently large if $\nabla^{2}\varphi_{k}(x^{k})\succ 0$ on $\setL_{0}\eqdef\setL-\setL$. This gives directly
%\[
%\inner{(\nabla^{2}f(x^{k})+\mu_{k}H(x^{k}))d,d}> -2\inner{x^{k}-x^{\ast},d}^{2}> -\delta^{2}\norm{d}_{2}^{2}\qquad\forall d\in\setL_{0}.
%\]
Setting $\mu_k \leq \sqrt{\eps_2}$ and $\delta \leq (\eps_2/9)^{1/4}$, we see that $x^{k}$ satisfies  $\inner{(\nabla^{2}f(x^{k})+\sqrt{\eps_2}(H(x^{k})+\bI))d,d}\geq 0,\forall d\in\setL_{0}$, which is clearly implied by \eqref{eq:eps_SO_optim_SO}. 
%Letting $\delta\to 0$, then for a given $\eps_{2}>0$ we see that condition \eqref{eq:eps_SO_optim_SO} applies for all $k$ sufficiently large. 
%\begin{remark}
%\label{rem:foopt}
%We note that both proposed definitions are stronger than the ones in \cite{HaeLiuYe18}, who use the weaker infinity norm in conditions \eqref{eq:eps_optim_grad} and \eqref{eq:eps_SO_optim_grad}, respectively. Moreover, they also use the weaker infinity norm in conditions \eqref{eq:eps_optim_complem} and \eqref{eq:eps_SO_optim_complem}. 
%Since in our case $\bar{x}\in\setK, \bar{s} \in\setK^{\ast}$, in our case we have for each $i=1,...,\dim(\setE)$, $ 0 \leq \bar{s}_i\bar{x}_i \leq \inner{\bar{s},\bar{x}} \leq \eps$. We discuss this in more details in Section \ref{sec:FO_discussion}, where we compare the complexity bounds for our first-order method and the first-order method in \cite{HaeLiuYe18}, as well as in Section \ref{sec:SO_discussion}, where we compare the complexity bounds for our second-order method and the second-order method in \cite{HaeLiuYe18}. \MS{We talk about this in laters sections anyhow so I think we do not need to put it here.}
%\end{remark}
%%%%%
\begin{remark}\label{rem:soopt}
To compare our second-order condition with the ones previously formulated in the literature, we consider the particular case $\bar{\setK}=\bar{\setK}_{NN}$ as in \cite{HaeLiuYe18,NeiWr20} with the log-barrier setup giving $H(x)=\diag[x_{1}^{-2},\ldots,x_{n}^{-2}]\eqdef\XX^{-2}$. Within this setup, our second-order condition \eqref{eq:eps_SO_optim_SO} becomes, after multiplication by  $[H(x)]^{-1/2}=\XX$ from left and right,
\[
\XX\nabla^{2}f(x)\XX+\sqrt{\eps_{2}}\bI\succeq 0\qquad \text{on the set }\{d\in\setE\vert \bA\XX d=0\}.%\ker(\bA\XX).
\] 
This is equivalent to Proposition 2(c) in \cite{HaeLiuYe18}, modulo our use of $\sqrt{\eps_{2}}$ instead of $\eps$ in \cite{HaeLiuYe18}, as well as equation (1.6d) in \cite{NeiWr20}, modulo our use of $\sqrt{\eps_{2}}$ instead of $\eps_H$ in \cite{NeiWr20}.
%On the other hand, our condition \eqref{eq:eps_SO_optim_SO} does not rely on the coupling $[H(\bar{x})]^{-1/2}=\XX$ which may not hold for general cones and is formulated in terms %of the Hessian $H(x)$ without the use of $[H(\bar{x})]^{-1/2}$ which may be tricky to define.
%
%
%This representation of the second-order optimality condition makes our set of KKT conditions easier to compare with the ones previously formulated in the literature, e.g. in \cite{HaeLiuYe18}. They consider the case $\bar{\setK}=\bar{\setK}_{NN}$ with the log-barrier setup. This gives $H(x)=\diag[x_{1}^{-2},\ldots,x_{n}^{-2}]\eqdef\XX^{-2}$. Within this setup, the second-order KKT condition becomes 
%\[
%\XX\nabla^{2}f(x)\XX+\sqrt{\eps_{2}}\bI\succeq 0\qquad \text{on }\ker(\bA\XX).
%\] 
%This corresponds to Proposition 2(c) in \cite{HaeLiuYe18}, modulo our use of $\sqrt{\eps_{2}}$ instead of $\eps$ in \cite{HaeLiuYe18}, as well as equarion (1.6d) in \cite{NeiWr20}, modulo our use of $\sqrt{\eps_{2}}$ instead of $\eps_H$ in \cite{NeiWr20}.
%
%Condition \eqref{eq:eps_SO_optim_SO} is equivalent to 
%\[
%[H(\bar{x})]^{-1/2}\nabla^2f(\bar{x})[H(\bar{x})]^{-1/2} + \sqrt{\eps_2}\bI \succeq 0 \quad \text{on }\ker(\bA[H(\bar{x})]^{-1/2}).
%\]
\close
\end{remark}
\begin{remark}
If $\bar{\setK}$ is a symmetric cone, the complementarity conditions \eqref{eq:eps_optim_complem} and \eqref{eq:eps_SO_optim_complem} are equivalent to complementarity notions formulated in terms of the multiplication $\circ$ under which $\setK$ becomes an Euclidean Jordan algebra. \citep[][Prop. 2.1]{LouFukMas18} shows that $x\circ y=0$ if and only $\inner{x,y}=0$, where $\inner{\cdot,\cdot}$ is the inner product of the ambient space $\setE$. Moreover, if $\bar{\setK}$ is a primitive symmetric cone, then by \citep[][Prop. III.4.1]{FarKor94}, there exists a constant $a>0$ such that $a\tr(x\circ y)=\inner{x,y}$ for all $x,y\in\setK$. In view of this relation, our complementarity notions could be specialized to the condition $\bar{s}\circ \bar{x}\leq \eps$. Hence, our approximate KKT conditions reduce to the ones reported in \cite{AndFukHaeSanSec21}. In particular, for $\bar{\setK}=\bar{\setK}_{\text{NN}}$ we recover the standard complementary slackness condition $s^{k}_{i}x^{k}_{i}\to 0$ as $k\to\infty$ for all $i$, as in this case the Jordan product $\circ$ gives rise to the Hadamard product. See \cite{Andreani:2019uf} for more details.
\close
\end{remark}

\subsection{On the relation to scaled critical points}
In absence of differentiability at the boundary, a popular formulation of necessary optimality conditions involves the definition of scaled-critical points. Indeed, at a local minimizer $x^{\ast}$, the scaled first-order optimality condition $x_{i}^{\ast}[\nabla f(x^{\ast})]_{i}=0,1\leq i\leq n$ holds, where the product is taken to be $0$ when the derivative does not exist. Based on this characterization, one may call a point $x\in\setK_{\text{NN}}$ with $\abs{x_{i}[\nabla f(x)]_{i}}\leq\eps$ for all $i=1,\ldots,n$ and $\eps$-scaled first-order point. Algorithms designed to produce $\eps$-scaled first-order points, with some small $\eps>0$, have been introduced in \cite{BiaCheYe15} and \cite{BiaChe15}. As reported in \cite{HaeLiuYe18}, there are several problems associated with this weak definition of a critical point. First, when derivatives are available on $\bar{\setK}_{\text{NN}}$, the standard definition of a critical point would entail the inclusion $\inner{\nabla f(x),x'-x}\geq 0$ for all $x'\in \bar{\setK}_{\text{NN}}.$ Hence, $[\nabla f(x)]_{i}=0$ for $x_{i}>0$ and $[\nabla f(x)]_{i}\geq 0$ for $x_{i}=0$. It follows, $\nabla f(x)\in\bar{\setK}_{\text{NN}}$, a condition that is absent in the definition of a scaled critical point. Second, scaled critical points come with no measure of strength, as they holds trivially when $x=0$, regardless of the objective function. Third, there is a general gap between local minimizers and limits of $\eps$-scaled first-order points, when $\eps\to 0^{+}$ (see \cite{HaeLiuYe18}). Similar remarks apply to the scaled second-order condition, considered in \cite{BiaChe15}. Our definition of approximate KKT points overcome these issues. In fact, our definitions of approximate first- and second-order KKT points is continuous in $\eps$, and therefore in the limit our approximate KKT points coincide with the classical first- and second-order KKT conditions for a local minimizer. This is achieved without assuming global differentiability of the objective function or performing an additional smoothing of the problem data as in \cite{Bian:2013vd,BiaChe15}.

%%%%%%%%%%%%%%%%%%%%%%%%%%%%%%%%%%%%%%%%%%%%%%
%%%%%%ALGORITHM%%%%%%%%%%%%%

\section{A first-order Hessian-Barrier Algorithm}
\label{sec:firstorder}
%----------------------------------------------------------------------
%%% First-Order ALGORITHM
%----------------------------------------------------------------------
% !TEX root = ./HBAConicMain.tex

In this section we introduce a first-order potential reduction method for solving \eqref{eq:Opt} that uses a barrier $h \in\scrH_{\nu}(\setK)$ and potential function \eqref{eq:potential}. %Our first-order method employs a quadratic regularization strategy for the linearization of $F_{\mu}(x)$ using the local norm at the current position $x$. 
%We are thus following classical numerical optimization ideas, related to the Levenberg–Morrison–Marquardt techniques \cite{NocWri00}, and also more recently used in \cite{CarGouToi11}. \PD{I'm afraid that this will give an impression that our results are simple and classical.}
We assume that we are able to compute an approximate analytic center at low computational cost. Specifically, our algorithm relies on the availability of a $\nu$-analytic center, i.e. a point $x^{0}\in\setX$ such that 
\begin{equation}\label{eq:analytic_center}
h(x)\geq h(x^{0})-\nu\qquad\forall x\in\setX. 
\end{equation}
To obtain such a point $x^{0}$, one can apply interior point methods to the convex programming problem $\min_{x\in \feas}h(x)$. Moreover, since $\nu \geq 1$ we do not need to solve it with high precision, making the application of computationally cheap first-order method, such as \cite{Dvurechensky:2022tu}, an appealing choice for this preprocessing step. 

\subsection{Local properties}
Given $x\in\setX$, define the set of \emph{feasible directions} as $\scrF_{x}=\{v\in\setE\vert x+v\in\feas\}.$ Lemma \ref{lem:Dikin} implies that 
\begin{equation}\label{eq:Dikinv}
\scrT_{x}=\{v\in\setE\vert \bA v=0,\norm{v}_{x}<1\}\subseteq\scrF_{x}.
\end{equation}
Upon defining $d=[H(x)]^{1/2}v$ for $v\in\scrT_{x}$, we obtain a point  $d \in \R^{{\rm dim}(\setE)}$ satisfying $\bA[H(x)]^{-1/2}d=0$ and $\norm{d}=\norm{v}_{x}$. Hence, for $x\in\setK$, we can equivalently characterize the set $\scrT_{x}$ as $\scrT_{x}=\{[H(x)]^{-1/2}d\vert \bA[H(x)]^{-1/2}d=0,\norm{d}<1\}$.\\ 
Our complexity analysis relies on the ability to control the behavior of the objective function along the set of feasible directions and with respect to the local norm. 
%%%%%%%%%%%%%%%%%%%%%%%%
\begin{assumption}[Local smoothness]
\label{ass:gradLip}
$f:\setE\to\R\cup\{+\infty\}$ is continuously differentiable on $\feas$ and there exists a constant $M>0$ such that for all $x\in\feas$ and $v\in\scrT_{x}$ we have 
\begin{equation}\label{eq:gradLip}
f(x+v) - f(x) - \inner{\nabla f(x),v} \leq \frac{M}{2}\norm{v}_x^2.
\end{equation}
\end{assumption}
\noindent
\begin{remark}
\label{rem:bounded1}
If the set $\bar{\setX}$ is bounded, we have $\lambda_{\min}(H(x)) \geq \sigma$ for some $\sigma >0$. In this case, assuming $f$ has an $M$-Lipschitz continuous gradient, the classical descent lemma \cite{Nes18} implies Assumption \ref{ass:gradLip}. Indeed,
\[
f(x+v) - f(x) - \inner{\nabla f(x),v} \leq \frac{M}{2}\norm{v}^2 \leq \frac{M}{2\sigma}\norm{v}_x^2.
\]
\close
\end{remark}
\begin{remark}
We emphasize that the local Lipschitz smoothness condition \eqref{eq:gradLip} does not require global differentiability. Consider the composite non-smooth and non-convex model \eqref{eq:composite} on $\bar{\setK}_{\text{NN}}$, with $\varphi(s)=s$ for $s \geq 0$. This means $\sum_{i=1}^{n}\varphi(x_{i}^{p})=\norm{x}_{p}^{p}$ for $p\in(0,1)$ and $x\in\bar{\setK}_{\text{NN}}$. As a concrete example for the smooth part of the problem let us consider the $L_{2}$-loss $\ell(x)=\frac{1}{2}\norm{\bN x-\bp}^{2}$. This gives rise to the $L_{2}-L_{p}$ minimization problem, an important optimization formulation arising in phase retrieval, mathematical statistics, signal processing and image recovery \cite{Fou09, GeJiaYe11,Chen:2014wx,LiLiuYaoYe17}. For $x\in\setK_{\text{NN}}$, set $M=\lambda_{\max}(\bN^{\ast}\bN)$, so that 
\[
\ell(x^{+})\leq \ell(x)+\inner{\nabla\ell(x),x^{+}-x}+\frac{M}{2}\norm{x^{+}-x}^{2},
\]
Since $t\mapsto t^{p}$ is concave for $t>0$ and $p\in(0,1)$, we have 
\[
(x_{i}^{+})^{p}\leq  x^{p}_{i}+px_{i}^{p-1}(x^{+}_{i}-x_{i})\qquad i=1,\ldots,n.
\]
Adding all these inequalities together, we immediately arrive at condition \eqref{eq:gradLip} in terms of the Euclidean norm. Over a bounded feasible set $\bar{\setX}$, Remark \ref{rem:bounded1} makes it clear that this implies Assumption \ref{ass:gradLip}. At the same time, $f$ is not differentiable at zero. \close
\end{remark}
We  emphasize that in Assumption \ref{ass:gradLip} the constant $M$ is in general either unknown or is a very conservative upper bound. Therefore, adaptive techniques should be used to estimate it and are likely to improve the practical performance of the method. 

Considering $x\in\setX,v\in\scrT_{x}$ and combining eq. \eqref{eq:gradLip} with eq. \eqref{eq:Dbound} (with $d=v$ and $t=1< \frac{1}{\norm{v}_x} \stackrel{\eqref{eq:boundzeta}}{\leq} \frac{1}{\zeta(x,v)}$) reveals a suitable quadratic model, to be used in the design of our first-order algorithm.
\begin{lemma}[Quadratic Overestimation]
For all $x\in\setX,v\in\scrT_{x}$ and $L\geq M$, we have 
\begin{equation}\label{eq:descentFOM}
F_{\mu}(x+v)\leq F_{\mu}(x)+\inner{\nabla F_{\mu}(x),v}+\frac{L}{2}\norm{v}^{2}_{x}+\mu\norm{v}^{2}_{x}\omega(\zeta(x,v)).
\end{equation}
\end{lemma}

\subsection{Algorithm description and its complexity}
\label{S:FO_descr}
Let $x \in \feas$ be given. Our first-order method employs a quadratic model $ Q^{(1)}_{\mu}(x,v)$ to compute a search direction $v_{\mu}(x)$, given by 
\begin{equation}\label{eq:search}
v_{\mu}(x) \eqdef \argmin_{v\in\setE:\bA v=0} \left\{  Q^{(1)}_{\mu}(x,v) \eqdef F_{\mu}(x) + \inner{\nabla F_{\mu}(x),v}+\frac{1}{2}\norm{v}_{x}^{2} \right\}.
\end{equation}
%\begin{equation}\label{eq:model1}
%Q^{(1)}_{\mu}(x,v) \eqdef F_{\mu}(x) + \inner{\nabla F_{\mu}(x),v}+\frac{1}{2}\norm{v}_{x}^{2}.
%\end{equation}
%We construct a search direction $v_{\mu}(x)$ as the solution of the strongly convex subproblem 
For the above problem, we have the following system of optimality conditions involving the dual variable $y_{\mu}(x)\in\R^{m}$:
\begin{align}
\nabla F_{\mu}(x) + H(x)v_{\mu}(x) - \bA^{\ast} y_{\mu}(x) &= 0, \label{eq:finder_1} \\
\bA  v_{\mu}(x) &=0. \label{eq:finder_2}
\end{align}
Since $H(x)\succ 0$ for $x\in\feas$, any standard solution method \citep{NocWri00} can be applied for the above linear system.
Moreover, this system can be solved explicitly.
Indeed, since $H(x)\succ 0$ for $x\in\feas$, and $\bA$ has full column rank, the linear operator $\bA[H(x)]^{-1}\bA^{\ast}$ is invertible. Hence, $v_{\mu}(x)$  is given explicitly as
\begin{equation*}%\label{eq:v_explicit}
v_{\mu}(x)= - ([H(x)]^{-1}\bA^{\ast}(\bA[H(x)]^{-1}\bA^{\ast})^{-1}\bA[H(x)]^{-1} - [H(x)]^{-1} ) \nabla F_{\mu}(x) \eqdef-\bS_{x}\nabla F_{\mu}(x).
\end{equation*}
To give some intuition behind this expression, observe that we can give an alternative representation of $\bS_{x}$ as $\bS_{x}v = [H(x)]^{-1/2}\Pi_{x}[H(x)]^{-1/2}v$, where
\[
\Pi_{x}v\eqdef v-[H(x)]^{-1/2}\bA^{\ast}(\bA[H(x)]^{-1}\bA^{\ast})^{-1}\bA[H(x)]^{-1/2}v.
\]
This shows that $\bS_{x}$ is just the $\norm{\cdot}_{x}$-orthogonal projection operator onto $\ker(\bA[H(x)]^{-1/2})$. Hence, we can always find a scalar $t>0$ such that $t v_{\mu}(x)\in\setL_{0}$ and $\norm{t v_{\mu}(x)}_{x}<1$. Any such scalar will be a suitable candidate for a step size. To determine an acceptable step-size, consider a point $x \in\setX$, the  search direction $v_{\mu}(x)$ gives rise to a family of parameterized arcs $x^{+}(t)\eqdef x+tv_{\mu}(x)$, where $t\geq 0$. Our aim is to choose this step-size to ensure feasibility of the iterates and decrease of the potential. By \eqref{eq:step_length_zeta} and \eqref{eq:finder_2}, we know that $x^{+}(t)\in\feas$ for all $t\in I_{x,\mu} \eqdef [0,\frac{1}{\zeta(x,v_{\mu}(x))})$. Multiplying \eqref{eq:finder_1} by $v_{\mu}(x)$ and using \eqref{eq:finder_2}, we obtain 
$\inner{\nabla F_{\mu}(x),v_{\mu}(x)}=-\norm{v_{\mu}(x)}_{x}^{2}$. Choosing $t \in  I_{x,\mu}$, we bound
\[
t^{2}\norm{v_{\mu}(x)}_{x}^{2}\omega(t\zeta(x,v_{\mu}(x))) \stackrel{\eqref{eq:omega_upper_bound}}{\leq}  \frac{t^{2}\norm{v_{\mu}(x)}_{x}^{2}}{2(1-t\zeta(x,v_{\mu}(x)))}. 
\]
Therefore, if $t\zeta(x,v_{\mu}(x))\leq 1/2$, we readily see from \eqref{eq:descentFOM} that
%\footnote{Note that since $t\zeta(x,v_{\mu}(x))=\zeta(x,tv_{\mu}(x))$ and we assumed that $t\zeta(x,v_{\mu}(x))\leq 1/2$, it is sufficient to make Assumption \ref{ass:gradLip} on a potentially smaller set $\widetilde{\scrT}_{x}\eqdef \{v\in\setE\vert \bA v=0,\zeta(x,v)\leq 1/2\}$ instead of $\scrT_x$ defined in \eqref{eq:Dikinv}.} 
\begin{align}
F_{\mu}(x^{+}(t))-F_{\mu}(x)&\leq -t\norm{v_{\mu}(x)}_{x}^{2}+\frac{t^{2}M}{2}\norm{v_{\mu}(x)}_{x}^{2}+\mu t^{2}\norm{v_{\mu}(x)}_{x}^{2} \nonumber\\
&= -t \norm{v_{\mu}(x)}_{x}^{2}\left(1-\frac{M+2\mu}{2}t\right) \eqdef -\eta_{x}(t).\label{eq:success}
\end{align}
The function $t \mapsto \eta_{x}(t)$ is strictly concave with the unique maximum at $ \frac{1}{M+2\mu}$, and two real roots at $t\in\left\{0,\frac{2}{M+2\mu}\right\}$. 
Thus, maximizing the per-iteration decrease $\eta_{x}(t)$  under the restriction $0\leq t\leq\frac{1}{2\zeta(x,v_{\mu}(x))}$, we choose the step-size
\begin{equation*}
%\label{eq:}
\ct_{\mu,M}(x)\eqdef \min \left\{\frac{1}{M+2\mu},\frac{1}{2\zeta(x,v_{\mu}(x))}\right\}.
\end{equation*}
This step-size rule, however, requires knowledge of the parameter $M$. To boost numerical performance, we employ a backtracking scheme in the spirit of \cite{NesPol06} to estimate the constant $M$ at each iteration. This procedure generates a sequence of positive numbers $(L_{k})_{k\geq 0}$ for which the local Lipschitz smoothness condition \eqref{eq:gradLip} holds. More specifically, suppose that $x^{k}$ is the current position of the algorithm with the corresponding initial local Lipschitz estimate $L_{k}$ and $v^{k}=v_{\mu}(x^{k})$ is the corresponding search direction. To determine the next iterate $x^{k+1}$, we iteratively try step-sizes $\alpha_k$ of the form $\ct_{\mu,2^{i_k}L_k}(x^{k})$ for $i_k\geq 0$ until the local smoothness condition \eqref{eq:gradLip} holds with $x=x^{k}$, $v= \alpha_k v^{k}$ and local Lipschitz estimate $M=2^{i_k}L_k$, see \eqref{eq:LS}. This process must terminate in finitely many steps, since when $2^{i_k}L_k \geq M$, inequality \eqref{eq:gradLip} with $M$ changed to $2^{i_k}L_k$, i.e., \eqref{eq:LS}, follows from Assumption \ref{ass:gradLip}. Combining the search direction finding problem \eqref{eq:search} with the just outlined backtracking strategy, yields an \underline{A}daptive first-order \underline{H}essian-\underline{B}arrier \underline{A}lgorithm ($\AHBA$, Algorithm \ref{alg:AHBA}).
%%%%%%%%%%%%%%%%%%%%%%%%%%%%%%%%
\begin{algorithm}[t]
\caption{ \underline{A}daptive first-order \underline{H}essian-\underline{B}arrier \underline{A}lgorithm  - $\AHBA(\mu,\eps,L_{0},x^{0})$}
\label{alg:AHBA}
\SetAlgoLined
\KwData{ $h \in\scrH_{\nu}(\setK)$, $\mu>0,\eps>0,L_0>0,x^{0}\in\setX$.
}
\KwResult{$(x^{k},y^{k},s^{k},L_{k})\in\setX\times\R^{m}\times\setK^{\ast}\times\R_{+}$, where $s^{k}=\nabla f(x^{k}) -\bA^{\ast}y^{k}$, and $L_{k}$ is the last estimate of the Lipschitz constant.}
%Set $L_0 > 0$ -- initial guess for $M$, 
Set $k=0$\;
\Repeat{
		$\norm{v^k}_{x^k} < \tfrac{\eps}{\nu}$ 
	}{	
		Set $i_k=0$. Find $v^k\eqdef v_{\mu}(x^k)$ and  the corresponding dual variable $y^k\eqdef y_{\mu}(x^k)$ as the solution to
		\begin{equation}\label{eq:finder}
		\min_{v\in\setE:\bA v=0}\{F_{\mu}(x^k) + \inner{\nabla F_{\mu}(x^k),v}+\frac{1}{2}\norm{v}_{x^{k}}^{2}\}. 
		\end{equation}
		\Repeat{
			\begin{equation}				
				f(z^{k}) \leq f(x^{k}) + \inner{\nabla f(x^{k}),z^{k}-x^{k}}+2^{i_{k}-1}L_{k}\norm{z^{k}-x^{k}}^{2}_{x^{k}}.
				\label{eq:LS}
			\end{equation}
		}
		{
			\begin{equation}\label{eq:alpha_k}
				\alpha_k \eqdef  \min \left\{\frac{1}{2^{i_k}L_{k} + 2 \mu},\frac{1}{2\zeta(x^k,v^k)} \right\},  	\text{where $\zeta(\cdot,\cdot)$ as in \eqref{eq:zeta}}	
			\end{equation}
			%%%%%%%%%%%%%%%
			Set $z^{k}=x^{k} + \alpha_k v^k$, $i_k=i_k+1$\;%, $i_k=i_k+1$\;
		}
		Set $L_{k+1} = 2^{i_k-1}L_{k}$, $x^{k+1}=z^{k}$, $k=k+1$\;%i_k-2 is since when we find a suitable $i_k$ we still set $i_k=i_k+1$
	}
\end{algorithm}
%%%%%%%%%%%%%%%%%%%%%%%%%%%%%%%
Our main result on the iteration complexity of Algorithm \ref{alg:AHBA} is the following Theorem, whose proof is given in Section \ref{sec:ProofFOM}. 
\begin{theorem}
\label{Th:AHBA_conv}
Let Assumptions \ref{ass:1}-\ref{ass:gradLip} hold. Fix the error tolerance $\eps>0$, the regularization parameter $\mu=\frac{\eps}{\nu}$, and some initial guess $L_0>0$ for the Lipschitz constant. Let $(x^{k})_{k\geq 0}$ be the trajectory generated by $\AHBA(\mu,\eps,L_{0},x^{0})$, where $x^{0}$ is a $\nu$-analytic center satisfying \eqref{eq:analytic_center}. Then the algorithm stops in no more than 
\begin{equation}
\label{eq:FO_main_Th_compl}
\K_{I}(\eps,x^{0})= \ceil[\bigg]{4(f(x^{0}) - f_{\min}(\setX)+ \eps) \frac{\nu^{2}(\max\{M,L_0\}+\eps/\nu)}{\eps^{2}}}
\end{equation}
outer iterations, and the number of inner iterations is no more than $2(\K_{I}(\eps,x^{0})+1)+\max\{\log_{2}(M/L_{0}),0\}$. Moreover, the last iterate obtained from $\AHBA(\mu,\eps,L_{0},x^{0})$ constitute a $2\eps$-KKT point for problem \eqref{eq:Opt} in the sense of Definition \ref{def:eps_KKT}.
%\begin{align}
%&\|\nabla f(x^{k}) - \bA^{\ast}y^{k} - s^{k} \| = 0 \leq 2\eps, \label{eq:FO_main_Th_eps_KKT_1} \\
%& |\inner{s^{k},x^{k}}| \leq 2\eps, \label{eq:FO_main_Th_eps_KKT_2}  \\
%& \bA x^{k}=b, s^{k} \in \setK^{\ast}, x^{k}\in\setK.   \label{eq:FO_main_Th_eps_KKT_3}
%\end{align}  
\end{theorem}
 %%%%%%%%
\begin{remark}
The line-search process of finding the appropriate $i_k$ is simple since only recalculating $z^k$ is needed, and repeatedly solving problem \eqref{eq:finder} is not required. Furthermore, the sequence of constants $L_k$ is allowed to decrease along subsequent iterations, which is achieved by the division by the constant factor 2 in the final updating step of each iteration. This potentially leads to longer steps and faster decrease of the potential.
\close
\end{remark}
%%%%%%%%%%%%%%%%%
\begin{remark}
\label{rem:FO_complexity_simplified}
Since $\nu \geq 1$, $f(x^{0}) - f_{\min}(\setX)$ is expected to be larger than $\eps$, and the constant $M$ is potentially large, we see that the main term in the complexity bound \eqref{eq:FO_main_Th_compl} is $O\left(\frac{M\nu^2(f(x^{0}) - f_{\min}(\setX))}{\eps^2}\right)=O(\frac{\nu^{2}}{\eps^{2}})$, i.e. has the same dependence on $\eps$ as the standard complexity bounds  \cite{CarDucHinSid19b,CarDucHinSid19,lan2020first} of first-order methods for non-convex problems under the standard Lipschitz-gradient assumption, which on bounded sets is subsumed by our Assumption \ref{ass:gradLip}. Further, if the function $f$ is quadratic, Assumption \ref{ass:gradLip} holds with $M=0$ and we can take $L_0=0$. In this case, the complexity bound \eqref{eq:FO_main_Th_compl} improves to $O\left(\frac{\nu(f(x^{0}) - f_{\min}(\setX))}{\eps}\right)$. 

Just like classical interior-point methods, the iteration complexity of $\AHBA$ depends on the barrier parameter $\nu\geq 1$. For conic domains, the characterization of this barrier parameter has thus been an active research line. \cite{GulTun98} demonstrated that for symmetric cones, the barrier parameter is equivalent to algebraic properties of the cone and identified it with the rank of the cone (see \cite{FarKor94} for a definition of the rank of a symmetric cone). This deep analysis gives an exact characterization of the optimal barrier parameter for the most important conic domains in optimization. For $\setK_{\text{NN}}$ and $\setK_{\text{SDP}}$, it is known that $\nu=n$ is optimal, whereas for $\setK_{\text{SOC}}$ the optimal barrier parameter is $\nu=2$ (and therefore independent of the ambient dimension $n$). 
\close
\end{remark}

\paragraph{Connection with interior point flows on polytopes.}
Consider $\bar{\setK}=\bar{\setK}_{\text{NN}}$, and $\setX=\setK_{\text{NN}}\cap\setL$. We are given a function $f:\bar{\setX}\to\R$ which is the restriction of a smooth function $f:\Rn\to\R$. %Let $\nabla f(x)=[\partial_{1}f(x),\ldots,\partial_{n}f(x)]^{\top}$ denote the gradient in $\Rn$ at $x\in\setX$. 
The canonical barrier for this setting is $h(x)=-\sum_{i=1}^{n}\ln(x_{i})$, so that $H(x)=\diag[x_{1}^{-2},\ldots,x_{n}^{-2}]=\XX^{-2}$ for $x\in\setX$. Applying our first-order method on this domain gives the search direction 
$v_{\mu}(x)=-\bS_{x}\nabla F_{\mu}(x)=-\XX(\bI-\XX\bA^{\top}(\bA\XX^{2}\bA^{\top})^{-1}\bA\XX)\XX\nabla F_{\mu}(x)$. This explicit formula yields various interesting connections between our approach and classical methods. For $\bA=\1_{n}^{\top}$, the feasible set $\setX$ reduces to the relative interior of the $(n-1)$-dimensional unit simplex. In this case, the vector field $v_{\mu}(\cdot)$ simplifies further to 
\[
[v_{\mu}(x)]_{i}=[\XX^{2}\nabla F_{\mu}(x)]_{i}-\frac{x_{i}^{2}}{\sum_{j=1}^{n}x_{j}^{2}}\sum_{j}[\XX^{2}\nabla F_{\mu}(x)]_{j} \quad 1\leq i\leq n,
\]
Observe that $v_{\mu}(x)\in (\1_{n})^{\bot}=\ker(\1_{n}^{\top})$. For $f(x)=c^{\top}x$ and $\mu=0$, we further obtain from this formula the search direction employed in \emph{affine scaling} methods for linear programming \cite{BayLag89,BayLag89II,AdlMont91,TseBomSch11}. 
%\[
%\bS_{x}c=\XX^{2} c-\XX^{2}\bA^{\top}(\bA\XX^{2}\bA^{\top})^{-1}\bA\XX^{2}c.
%\]
%This map is the search direction in \emph{affine scaling} methods for linear programming \cite{BayLag89,BayLag89II,AdlMont91}, a fundamental class of interior-point methods. Extensions to quadratic programming have been studied in various papers, notably by \cite{Ye92,Tse04,TsuMon96,Ye98}.
%
\cite{HBA-linear} partly motivated their algorithm as a discretization of the Hessian-Riemannian gradient flows introduced in \cite{ABB04} and \cite{BolTeb03}. Heuristically, we can therefore interpret $\AHBA$ as an Euler discretization (with non-monotone adaptive step-size policies) of the gradient-like flow $\dot{x}(t)=-\bS_{x(t)}\nabla F_{\mu}(x(t))$, which resembles very much the class of dynamical systems introduced in \cite{BolTeb03}. This gives an immediate connection to a large class of interior point flows on polytopes, heavily studied in control theory \cite{HelMoo96}.
%%Proof%%%
\subsection{Proof of Theorem \ref{Th:AHBA_conv}}
\label{sec:ProofFOM}
Our proof proceeds in several steps. First, we show that procedure $\AHBA(\mu,\eps,L_{0},x^{0})$ produces points in $\setX$, and, thus, is indeed an interior-point method. Next, we show that the line-search process of finding appropriate $L_k$ in each iteration is finite, and estimate the total number of trials in this process. Then we enter the core of our analysis where we prove that if the stopping criterion does not hold at iteration $k$, i.e. $\norm{v^{k}}_{x^{k}} \geq \tfrac{\eps}{\nu}$, then the objective $f$ is decreased by a quantity $O(\eps^{2})$, and, since the objective is globally lower bounded, we conclude that the method stops in at most $O(\eps^{-2})$ iterations. Finally, we show that when the stopping criterion holds, the method has generated an $\eps$-KKT point. 

\subsubsection{Interior-point property of the iterates}
\label{S:FO_correct}
By construction $x^{0}\in\setX$. Proceeding inductively, let $x^{k}\in\setX$ be the $k$-th iterate of the algorithm, delivering the search direction $v^{k}=v_{\mu}(x^{k})$. 
By eq. \eqref{eq:alpha_k}, the step-size $\alpha_k$ satisfies $\alpha_{k}\leq \frac{1}{2\zeta(x^{k},v^{k})}$, and, hence, $\alpha_{k}\zeta(x^{k},v^{k})\leq 1/2$ for all $k\geq 0$. 
Thus, by \eqref{eq:step_length_zeta} $x^{k+1}=x^{k}+\alpha_{k}v^{k} \in\setK$. Since, by \eqref{eq:finder}, $\bA v^{k} =0$, we have that $x^{k+1} \in \setL$. Thus, $x^{k+1}\in\setK \cap \setL=\setX$. By induction, we conclude that $(x^{k})_{k\geq 0}\subset\setX$. 

\subsubsection{Bounding the number of backtracking steps}
\label{sec:backtrack1}
Let us fix iteration $k$. Since the sequence $2^{i_k} L_k $ is increasing as $i_k$ is increasing, and Assumption \ref{ass:gradLip} holds, we know that when $2^{i_k} L_k \geq \max\{M,L_k\}$, the line-search process for sure stops since inequality \eqref{eq:LS} holds. 
%Hence, if $L_0 \leq M$, we have that $2^{i_k} L_k \leq 2M$. Otherwise, if $L_0 > M$, we have 
Hence, $2^{i_k} L_k \leq 2\max\{M,L_k\}$ must be the case, and, consequently, $L_{k+1} = 2^{i_k-1} L_k \leq \max\{M,L_k\}$, which, by induction, gives $L_{k+1} \leq \bar{M}\eqdef\max\{M,L_0\}$. 
%By construction we have $L_{k+1}=\frac{1}{2}2^{i_{k}}L_{k}$. 
At the same time, $\log_{2}\left(\frac{L_{k+1}}{L_{k}}\right)= i_{k}-1$, $\forall k\geq 0$. Let $N(k)$ denote the number of inner line-search iterations up to the $k-$th iteration of $\AHBA(\mu,\eps,L_{0},x^{0})$. Then, using that $L_{k+1} \leq \bar{M}=\max\{M,L_0\}$, 
\begin{align*}
N(k)&=\sum_{j=0}^{k}(i_{j}+1)=\sum_{j=0}^{k} (\log_{2}(L_{j+1}/L_{j})+2 ) \leq 2(k+1)+\max\{\log_{2}(M/L_{0}),0\}.
\end{align*}
This shows that on average the inner loop ends after two trials. 

\subsubsection{Per-iteration analysis and a bound for the number of iterations}
%Given $\eps>0$, we choose $\mu=\eps/\nu$ and $x^{0}$ a $\nu$-analytic center. 
Let us fix iteration counter $k$. Since $L_{k+1} = 2^{i_k-1}L_k$, the step-size \eqref{eq:alpha_k} reads as $\alpha_{k}=\min \left\{\frac{1}{2L_{k+1} + 2 \mu},\frac{1}{2\zeta(x^k,v^k)} \right\}$. Hence, $\alpha_{k}\zeta(x^{k},v^{k})\leq 1/2$, and \eqref{eq:success} with the identification $t=\alpha_k = \ct_{\mu,2L_{k+1}}(x^{k})$, $M=2L_{k+1}$, $x=x^{k}$, $v_{\mu}(x^{k}) \eqdef v^k$ gives:  
\begin{equation}
\label{eq:FO_per_iter_proof_2}
F_{\mu}(x^{k+1})-F_{\mu}(x^{k})\leq -\alpha_k \norm{v^k}_{x^k}^{2}\left(1-(L_{k+1}+\mu)\alpha_k\right) \leq -\frac{\alpha_k \norm{v^k}_{x^k}^{2}}{2},
\end{equation}
where we used that $\alpha_k \leq \frac{1}{2(L_{k+1}+\mu)}$ in the last inequality. 
Substituting into \eqref{eq:FO_per_iter_proof_2} the two possible values of the step-size $\alpha_k$ in \eqref{eq:alpha_k} gives
%Next, we consider two possible cases of the value of the step-size $\alpha_k$ and substitute it into \eqref{eq:FO_per_iter_proof_2}.
\begin{equation}
\label{eq:per_iter_decr_0}
F_{\mu}(x^{k+1})-F_{\mu}(x^{k})\leq 
\left\{
\begin{array}{ll}
- \frac{\norm{v^{k}}_{x^{k}}^{2} }{4(L_{k+1}+\mu)} & \text{if }  \alpha_k=\frac{1}{2(L_{k+1}+\mu)}\\

- \frac{\norm{v^{k}}_{x^{k}}^{2} }{4\zeta(x^k,v^k)}  \stackrel{\eqref{eq:boundzeta}}{\leq} - \frac{\norm{v^{k}}_{x^{k}}}{4} & \text{if }  \alpha_k=\frac{1}{2\zeta(x^k,v^k)}.
\end{array}\right.
\end{equation}
Recalling $L_{k+1} \leq \bar{M}$ (see section \ref{sec:backtrack1}), we obtain that 
\begin{equation}
\label{eq:per_iter_decr}
F_{\mu}(x^{k+1}) - F_{\mu}(x^{k}) \leq  -\frac{\norm{v^{k}}_{x^{k}}}{4} \min\left\{1,  \frac{\norm{v^{k}}_{x^{k}}}{\bar{M}+\mu}\right\}=-\delta_{k}.
\end{equation}
Rearranging and summing these inequalities for $k$ from $0$ to $K-1$ gives
\begin{align}
&K\min_{k=0\ldots,K-1} \delta_{k} \leq  \sum_{k=0}^{K-1}\delta_{k}\leq F_{\mu}(x^{0})-F_{\mu}(x^{K}) \notag \\
& \quad\stackrel{\eqref{eq:potential}}{=} f(x^{0}) - f(x^{K}) + \mu (h(x^{0}) - h(x^{K})) \leq f(x^{0}) - f_{\min}(\setX) + \eps, \label{eq:FO_per_iter_proof_6} 
\end{align}
where we used that, by the assumptions of Theorem \ref{Th:AHBA_conv}, $x^{0}$ is a $\nu$-analytic center defined in \eqref{eq:analytic_center} and $\mu = \eps/\nu$, implying that $h(x^{0}) - h(x^{K}) \leq \nu = \eps/\mu$.
Thus, up to passing to a subsequence, $\delta_{k}\to 0$, and consequently $\norm{v^{k}}_{x^{k}} \to 0$ as $k \to \infty$. This shows that the stopping criterion in Algorithm \ref{alg:AHBA} is achievable.

Assume now that the stopping criterion $\norm{v^k}_{x^k} < \frac{\eps}{ \nu}$ does not hold for $K$ iterations of $\AHBA$. Then, for all $k=0,\ldots,K-1,$ it holds that 
$\delta_{k}\geq \min\left\{\frac{\eps}{4 \nu},\frac{\eps^{2}}{4 \nu^{2}(\bar{M}+\mu)}\right\}$. 
Together with the parameter coupling $\mu=\frac{\eps}{\nu}$, it follows from \eqref{eq:FO_per_iter_proof_6} that
\[
K\frac{\eps^{2}}{4 \nu^{2}(\bar{M}+\eps/\nu)} = K \min\left\{\frac{\eps}{4 \nu},\frac{\eps^{2}}{4 \nu^{2}(\bar{M}+\eps/\nu)}\right\}\leq f(x^{0})-f_{\min}(\setX)+\eps.
\]
Hence, recalling that $\bar{M}=\max\{M,L_0\}$,
\[
K \leq 4(f(x^{0}) - f_{\min}(\setX)+ \eps) \cdot \frac{\nu^{2}(\max\{M,L_0\}+\eps/\nu)}{\eps^{2}},%\max\left\{\frac{\nu}{\eps},\frac{\nu^{2}(M+\eps/\nu)}{\eps^{2}}\right\},
\] 
i.e., the algorithm stops for sure after no more than this number of iterations. This, combined with the bound for the number of inner steps in Section \ref{sec:backtrack1}, proves the first statement of Theorem \ref{Th:AHBA_conv}.

\subsubsection{Generating $\eps$-KKT point}
To finish the proof of Theorem \ref{Th:AHBA_conv}, we now show that when Algorithm \ref{alg:AHBA} stops for the first time, it returns a $2\eps$-KKT point of \eqref{eq:Opt} according to Definition \ref{def:eps_KKT}.

Let the stopping criterion hold at iteration $k$. By the optimality condition \eqref{eq:finder_1} and the definition of the potential \eqref{eq:potential}, we have
\begin{equation}
\label{eq:FO_KKT_proof_0}
\nabla f(x^{k})-\bA^{\ast}y^{k}+\mu \nabla h(x^{k}) =-H(x^{k})v^{k}\iff [H(x^{k})]^{-1}\left(\nabla f(x^{k})-\bA^{\ast}y^{k}+\mu \nabla h(x^{k}) \right)=-v^{k}.
\end{equation}
Denoting $g^{k}\eqdef-\mu\nabla h(x^{k})$, multiplying both equations, and using the stopping criterion $\norm{v^{k}}_{x^{k}} < \frac{\eps}{\nu}$, we conclude 
\begin{equation}
\label{eq:FO_KKT_proof_00}
\norm{\nabla f(x^{k})-\bA^{\ast}y^{k}-g^{k}}^{\ast}_{x^{k}}=\norm{v^{k}}_{x^{k}}<\frac{\eps}{\nu}.
\end{equation}
Whence, setting $s^{k}\eqdef\nabla f(x^{k})-\bA^{\ast}y^{k}\in\setE^{\ast}$, we get, by the definition of the dual norm, 
\begin{align}
\frac{\eps}{\nu} > \norm{v^{k}}_{x^{k}}&=\norm{s^{k}-g^{k}}^{\ast}_{x^{k}} \label{eq:FO_KKT_proof_1}\\
&=\norm{s^{k}-g^{k}}_{[H(x^{k})]^{-1}}\stackrel{\eqref{eq:relations}}{=}\norm{s^{k}-g^{k}}_{\nabla^{2}h_{\ast}(-\nabla h(x^{k}))} \\
&=\norm{s^{k}-g^{k}}_{\nabla^{2}h_{\ast}(\frac{1}{\mu}g^{k})} = \mu\norm{s^{k}-g^{k}}_{\nabla^{2}h_{\ast}(g^{k})}, \notag
\end{align}
where in the last equality we used that since $h_{\ast}\in\scrH_{\nu}(\setK^{\ast})$, by \eqref{eq:log_hom_scb_hess_homog_prop}, 
%$\frac{1}{t^{2}}\nabla^{2}h_{\ast}(s)=\nabla^{2}h_{\ast}(ts)$. Therefore, 
$\nabla^{2}h_{\ast}(\frac{1}{\mu}g^{k})=\mu^{2}\nabla^{2}h_{\ast}(g^{k})$.
%, and we conclude $\norm{v^{k}}_{x^{k}}=\mu\norm{s^{k}-g^{k}}_{\nabla^{2}h_{\ast}(g^{k})}.$
%Whence, since the stopping criterion $\norm{v^{k}}_{x^{k}} < \frac{\eps}{\nu}$ holds at iteration $k$,
Thus, we arrive at
\begin{align}
\norm{s^{k}-g^{k}}_{\nabla^{2}h_{\ast}(g^{k})} = \frac{\norm{v^{k}}_{x^{k}}}{\mu} < \frac{\eps}{\mu \nu} = 1, \label{eq:FO_KKT_proof_2}
\end{align}
where in the last equality we used that, by the assumptions of Theorem \ref{Th:AHBA_conv}, $\mu=\frac{\eps}{\nu}$.
Thus, since, by \eqref{eq:relations_1}, $g^{k}=-\mu\nabla h(x^{k})\in \setK^{\ast}$, we get that $s^{k}\in\setK^{\ast}$. By construction, $x^{k}\in \setK$ and $\bA x^{k} = b$. Thus, \eqref{eq:eps_optim_equality_cones} holds.
Furthermore, $\|\nabla f(x^{k})-\bA^{\ast}y^{k} - s^{k}\|=0\leq 2 \eps$, meaning that \eqref{eq:eps_optim_grad} holds. 
Finally, since $(x^{k},s^{k})\in\setK\times\setK^{\ast}$, we see 
\begin{align}
0 \leq \inner{s^{k},x^{k}}&=\inner{s^{k}-g^{k},x^{k}}+\inner{g^{k},x^{k}} \notag\\
&\leq \norm{s^{k}-g^{k}}_{x^{k}}^{\ast}\cdot\norm{x^{k}}_{x^{k}}-\mu\inner{\nabla h(x^{k}),x^{k}} \notag \\
&\stackrel{\eqref{eq:FO_KKT_proof_1},\eqref{eq:log_hom_scb_norm_prop},\eqref{eq:log_hom_scb_hess_prop}}{=}\norm{v^{k}}_{x^{k}}\sqrt{\nu}+\mu\nu \notag \\
&<\sqrt{\nu}\frac{\eps}{\nu}+\eps\leq 2\eps, \label{eq:FO_KKT_proof_3}
\end{align}
where the last inequality uses $\nu\geq 1$. Hence, the complementarity condition \eqref{eq:eps_optim_complem} holds as well. This finishes the proof of Theorem \ref{Th:AHBA_conv}.

\subsection{Discussion}
\label{sec:FO_discussion}
\paragraph{Strengthened KKT condition.}
%The main technical challenge in our analysis is to guarantee the approximate complementarity condition \eqref{eq:eps_optim_complem}. For this, we need to show that the dual variable $s^{k}$ belongs to the dual cone $\setK^{\ast}$, and second, we need to show that \eqref{eq:eps_optim_complem} holds. Another challenge is to translate the condition \eqref{eq:FO_KKT_proof_00} stated in terms of the local norm to the condition \eqref{eq:eps_optim_grad} formulated in terms of the standard Euclidean norm. In our setting of general, potentially non-symmetric, cones all this is more difficult than for the prominent (symmetric) cone case $\bar{\setK}=\bar{\setK}_{\text{NN}}$, as studied in \cite{HaeLiuYe18}. % where a first-order algorithm is also proposed. 
%Let us compare the bounds reported in that references with ours. 
For $\bar{\setK}=\bar{\setK}_{\text{NN}}$, \cite{HaeLiuYe18} consider a first-order potential reduction method employing the standard log-barrier $h(x)=-\sum_{i=1}^n \ln(x_i)$ using a trust-region subproblem for obtaining the search direction. For $x\in\setK_{\text{NN}}$, we have $\nabla h(x)=[-x_{1}^{-1},\ldots,-x_{n}^{-1}]^{\top}$, $H(x)=\diag[x_{1}^{-2},\ldots,x_{n}^{-2}]=\XX^{-2}$. Combining \eqref{eq:FO_KKT_proof_0}, the information $[H(x^{k})]^{-1/2}\nabla h(x^{k})= - \1_{n},\;\nu=n$, and the stopping criterion of Algorithm \ref{alg:AHBA} at iteration $k$, saying that $\norm{v^{k}}_{x^{k}}<\frac{\eps}{\nu}$, we see 
%\begin{equation}
%\label{eq:FO_remarks_1}
\begin{align*}
\norm{H(x^{k})^{-\frac{1}{2}} (\nabla f(x^{k})-\bA^{\ast}y^{k}) - \mu \1_{n}}_{\infty}& \leq\norm{H(x^{k})^{-\frac{1}{2}} (\nabla f(x^{k})-\bA^{\ast}y^{k}) - \mu \1_{n}  }\\
& = \norm{-H(x^{k})^{\frac{1}{2}} v^{k}} < \frac{\eps}{n}.
\end{align*}
Therefore, since $\mu=\eps/n$ and $s^{k}=\nabla f(x^{k})-\bA^{\ast}y^{k}\in\setK^{\ast}_{\text{NN}}=\Rn_{++}$, we obtain from the triangle inequality
\[
0<\norm{\XX^{k}s^{k}}_{\infty}\leq \norm{H(x^{k})^{-1/2} s^{k}-\mu\1_{n}}_{\infty}+\mu\leq\frac{2\eps}{n}.
\]
By Remark \ref{rem:FO_complexity_simplified}, these inequalities are achieved after $O\left(\frac{M n^2(f(x^{0}) - f_{\min}(\setX))}{\eps^2}\right)$ iterations of $\AHBA$, and they are seen to be by the factor $\frac{1}{n}$ sharper than the complementarity measure employed in \cite{HaeLiuYe18}. Conversely, in order to attain an approximate KKT point with the same strength as in \cite{HaeLiuYe18}, the above calculations suggest that we can weaken our tolerance from $\eps$ to $\eps\cdot n$, which results in an overall iteration complexity of   
$O\left(\frac{M  (f(x^{0}) - f_{\min}(\setX))}{\eps^2}\right)$, and a complementarity measure $\norm{\XX^{k}s^{k}}_{\infty}\leq 2\eps$. %
Thus, in the particular case of non-negativity constraints our general algorithm is able to obtain results similar to \cite{HaeLiuYe18}, but under weaker assumptions. At the same time, our algorithm ensures a stronger measure of complementarity. Indeed, our algorithm guarantees that $x^{k}\in\setK_{\text{NN}},s^{k}=\nabla f(x^k) -\bA^{\ast}y^{k}\in\setK_{\text{NN}}$, i.e., $x^{k},s^{k}\geq 0$, and approximate complementary $0\leq \sum_{i=1}^{n}\abs{x_{i}^{k}s_{i}^{k}}=\sum_{i=1}^{n}x_{i}^{k}s_{i}^{k}\leq 2\eps$ after $O\left(\frac{M n^2(f(x^{0}) - f_{\min}(\setX))}{\eps^2}\right)$ iterations, which is stronger than $\max_{1\leq i\leq n}\abs{x_{i}^{k}s_{i}^{k}} \leq 2\eps$ guaranteed by \cite{HaeLiuYe18}. Indeed, $\max_{1\leq i\leq n}\abs{x_{i}^{k}s_{i}^{k}} \leq \sum_{i=1}^{n}\abs{x_{i}^{k}s_{i}^{k}} \leq n \max_{1\leq i\leq n}\abs{x_{i}^{k}s_{i}^{k}}$, and both equalities are achievable. Moreover,  to match our stronger guarantee, one has to change $\eps \to\eps/n$ in the complexity bound of \cite{HaeLiuYe18}, which leads to the same $O\left(\frac{M n^2(f(x^{0}) - f_{\min}(\setX))}{\eps^2}\right)$ complexity bound. Besides this important insights, our algorithm is designed for general cones, rather than only for $\bar{\setK}_{\text{NN}}$. Therefore, we provide a unified approach for essentially all conic domains of relevance in optimization. Finally, our method does not rely on the trust-region techniques as in \cite{HaeLiuYe18} that may slow down the convergence in practice since the radius of the trust region is no grater than $O(\eps)$ leading to short steps.

\paragraph{Exploiting problem structure.}
In \eqref{eq:per_iter_decr_0} we can clearly observe the benefit of the use of $\nu$-SSB in our algorithm, whenever $\setK$ is a symmetric cone. Indeed, when $\alpha_k=\frac{1}{2\zeta(x^k,v^k)}$, the per-iteration decrease of the potential is $\frac{\norm{v^{k}}_{x^{k}}^{2} }{4\zeta(x^k,v^k)} \geq \frac{ \eps \norm{v^{k}}_{x^{k}}}{4\nu\zeta(x^k,v^k)} $ which may be large if $\zeta(x^k,v^k)=\sigma_{x^k}(-v^k) \ll \norm{v^{k}}_{x^{k}}$.

\paragraph{The role of the potential function.}
Next, we discuss more explicitly, how the algorithm and complexity bounds depend on the parameter $\mu$. The first observation is that from \eqref{eq:FO_KKT_proof_2}, to guarantee that $s^{k} \in \setK^{\ast}$, we need the stopping criterion to be $\norm{v^{k}}_{x^k} < \mu$, which by \eqref{eq:FO_KKT_proof_3} leads to the error $2 \mu \nu$ in the complementarity conditions. From the analysis following equation \eqref{eq:FO_per_iter_proof_6}, we have that  
\[
K\frac{\mu^{2}}{4  (\bar{M}+\mu)} = K \min\left\{\frac{\mu}{4},\frac{\mu^{2}}{4 \ (\bar{M}+\mu)}\right\}\leq f(x^{0})-f_{\min}(\setX)+\mu \nu.
\]
Whence, recalling that $\bar{M}=\max\{M,L_0\}$,
\[
K \leq 4(f(x^{0}) - f_{\min}(\setX)+ \mu \nu) \cdot \frac{\max\{M,L_0\}+\mu}{\mu^{2}}.%\max\left\{\frac{\nu}{\eps},\frac{\nu^{2}(M+\eps/\nu)}{\eps^{2}}\right\},
\]
Thus, we see that after $O(\mu^{-2})$ iterations the algorithm finds a $(2 \mu \nu)$-KKT point, and  if $\mu \to 0$, we have convergence to a KKT point, but the complexity bound tends to infinity and becomes non-informative. At the same time, as it is seen from \eqref{eq:finder}, when $\mu \to 0$, the algorithm itself converges to a preconditioned gradient method since $F_{\mu}(x)=f(x) + \mu h(x) \to  f(x)$. 
We also see from the above explicit expressions in terms of $\mu$ that the design of the algorithm requires careful balance between the desired accuracy of the approximate KKT point expressed mainly by the complementarity condition, stopping criterion, and complexity. Moreover, the step-size should be also taken carefully to ensure the feasibility of the iterates, and the standard for first-order methods step-size $1/M$ may not work.

\subsection{Anytime convergence via restarting $\AHBA$}
\label{sec:path-following}
The analysis of Algorithm \ref{alg:AHBA} is based on the a-priori fixed tolerance $\eps>0$ and the parameter coupling $\mu=\eps/\nu$. This coupling allows us to embed Algorithm \ref{alg:AHBA} within a restarting scheme featuring a decreasing sequence $\{\mu_{i}\}_{i\geq 0}$, followed by restarts of $\AHBA$. This restarting strategy frees Algorithm \ref{alg:AHBA} from hard-coded parameters and connects it well to traditional barrier methods.% In this way, we obtain, to the best of our knowledge, the first path-following method with complexity guarantees for non-convex problems with linear and conic constraints. \MS{Not sure if this can not be interpreted as overselling.}

To describe this double-loop algorithm, we fix $\eps_{0}>0$ and select the starting point $x_0^{0}$ as a $\nu$-analytic centre of $\setX$ with respect to $h\in\scrH_{\nu}(\setK)$. We let $i \geq 0$ denote the counter for the restarting epochs at the start of which the value $\mu_{i}$ is decreased. In epoch $i$, we generate a sequence $\{x^{k}_{i}\}_{k=0}^{K_{i}}$ by calling $\AHBA(\mu_{i},\eps_{i},L_0^{(i)},x^{0}_{i})$ until the stopping condition is reached. This will take at most $\K_{I}(\eps_{i},x^{0}_{i})$ iterations, specified in eq. \eqref{eq:FO_main_Th_compl}. We store the last iterate $\hat{x}_{i}=x^{K_{i}}_{i}$ and the last estimate of the Lipschitz modulus $\hat{M}_{i}=L_{K_i}^{(i)}$ obtained from procedure $\AHBA(\mu_{i},\eps_{i},L_0^{(i)},x^{0}_{i})$ and then restart the algorithm using the ``warm starts'' $x^{0}_{i+1}=\hat{x}_{i}$, $L_{0}^{(i+1)}=\hat{M}_{i}/2$, $\eps_{i+1}=\eps_{i}/2$, $\mu_{i+1}=\eps_{i+1}/\nu$. If $\eps \in (0,\eps_0)$ is the target accuracy of the final solution, it suffices to perform $ \lceil \log_{2}(\eps_{0}/\eps)\rceil+1$ restarts since, by construction, $\eps_{i} = \eps_0 \cdot 2^{-i}$. 
%%%%%%%%%%%%%%%%%%%%%%%%%%%%
\begin{algorithm}[h]
\caption{Restarting $\AHBA$}
\label{alg:RestartHBA}
\SetAlgoLined
\KwData{ $h \in\scrH_{\nu}(\setK)$, $\eps_{0}>0$, $x_0^{0}\in\setX$ satisfying \eqref{eq:analytic_center}, $L_0^{(0)}>0$.}
\KwResult{Point $\hat{x}_i$, dual variables $\hat{y}_i$, $\hat{s}_i = \nabla f(\hat{x}_i) -\bA^{\ast}\hat{y}_i$.}
%Set $I=\lceil\log_{2}(\eps_{0}/\eps)\rceil+1$\;
\For{$i=0,1,\ldots$} 
{ Set $\eps_{i}=2^{-i}\eps_{0}$, $\mu_{i}=\frac{\eps_i}{\nu}$\; 
 Obtain $(\hat{x}_{i},\hat{y}_{i},\hat{s}_{i},\hat{M}_{i})$ from $\AHBA(\mu_{i},\eps_i,L_0^{(i)},x^{0}_{i})$\;
 %
% $(x_{i}^{K_{i}},L_{i}^{K_{i}})$ from $\AHBA(\mu_{i},\eps_i,x^{0}_{i})$\; 
 Set $x_{i+1}^{0}=\hat{x}_{i}$ and $L_0^{(i+1)}=\hat{M}_{i}/2$.
	}
\end{algorithm}
%%%%%%%%%%%%%%%%%	
\begin{theorem}\label{th:ComplexityPathfollowing}
Let Assumptions \ref{ass:1}-\ref{ass:gradLip} hold. 
Then, for any $\eps \in (0,\eps_0)$, Algorithm \ref{alg:RestartHBA} finds a $2\eps$-KKT point for problem \eqref{eq:Opt} in the sense of Definition \ref{def:eps_KKT} after
no more than $I(\eps):=\lceil \log_{2}(\eps_{0}/\eps)\rceil+1$ restarts and at most 
$
\left\lceil \frac{64}{3\eps^2}(f(x^{0})-f_{\min}(\setX)+\eps_0)\nu^2(\max\{M,L_0^{(0)}\}+\eps_0/\nu)\right\rceil
$
 iterations of $\AHBA$.
\end{theorem}
\begin{proof}
Let us consider a restart $i \geq 0$ and repeat the proof of Theorem \ref{Th:AHBA_conv} with the change $\eps \to \eps_i$,  $\mu \to \mu_i = \eps_i/\nu$, $L_0 \to L_0^{(i)}=\hat{M}_{i-1}/2$, $\bar{M}=\max\{M,L_0\} \to \bar{M}_i=\max\{M,L_0^{(i)}\}$, $x^0 \to x^{0}_{i}=\hat{x}_{i-1}$.
Let $K_i$ be the last iteration of $\AHBA(\mu_{i},\eps_i,L_0^{(i)},x^{0}_{i})$ meaning that $\norm{v^{K_i}}_{x^{K_i}} < \frac{\eps_i}{\nu}$ and $\norm{v^{K_i-1}}_{x^{K_i-1}} \geq \frac{\eps_i}{\nu}$. From the analysis following equation \eqref{eq:FO_per_iter_proof_6}, we have that
\begin{align}
&K_i \frac{\eps_i^{2}}{4 \nu^{2}(\bar{M}_i+\eps_i/\nu)} \leq K_i \min_{k=0\ldots,K_i-1} \delta_{k}^i \leq  \sum_{k=0}^{K_i-1}\delta_{k}^i\leq F_{\mu_i}(x^{0}_i)-F_{\mu_i}(x^{K_i}_i). \label{eq:PF_proof_1} 
\end{align}
Further, using the fact that $\mu_i$ is a decreasing sequence and \eqref{eq:analytic_center}, it is easy to deduce
\begin{align}
F_{\mu_{i+1}}(x^{0}_{i+1})&=F_{\mu_{i+1}}(x^{K_i}_{i})\stackrel{\eqref{eq:potential}}{=}f(x^{K_i}_{i}) + \mu_{i+1} h(x^{K_i}_{i}) \stackrel{\eqref{eq:potential}}{=}F_{\mu_{i}}(x^{K_i}_{i}) + (\mu_{i+1} - \mu_{i}) h(x^{K_i}_{i}) \notag \\
&\stackrel{\eqref{eq:analytic_center}}{\leq} F_{\mu_{i}}(x^{K_i}_{i}) + (\mu_{i+1} - \mu_{i}) (h(x_0^0)-\nu)\notag \\
& \stackrel{\eqref{eq:PF_proof_1}}{\leq}F_{\mu_i}(x^{0}_i)-K_i \frac{\eps_i^{2}}{4 \nu^{2}(\bar{M}_i+\eps_i/\nu)} + (\mu_{i+1} - \mu_{i}) (h(x_0^0)-\nu).
\label{eq:PF_proof_2} 
\end{align}
Letting $I\equiv I(\eps):=\left\lceil \log_2 (\frac{\eps_0}{\eps}) \right\rceil+1$, by Theorem \ref{Th:AHBA_conv} applied to the restart $I-1$, we see that $\AHBA(\mu_{I-1},\eps_{I-1},L_0^{(I-1)},x^{0}_{I-1})$ outputs a $2\eps$-KKT point for problem \eqref{eq:Opt} in the sense of Definition \ref{def:eps_KKT}.
%Clearly, to achieve any desired accuracy $\eps$, i.e., find $2\eps$-KKT point for problem \eqref{eq:Opt}, it is sufficient to make $I=I_I(\eps)=\left\lceil \log_2 \frac{\eps_0}{\eps} \right\rceil+1$ restarts $i=0,...,I-1$. 
Summing inequalities \eqref{eq:PF_proof_2} for all the performed restarts $i=0,...,I-1$ and rearranging the terms, we obtain
\begin{align}
\sum_{i=0}^{I-1} K_i \frac{\eps_i^{2}}{4 \nu^{2}(\bar{M}_i+\eps_i/\nu)} & \leq F_{\mu_0}(x^{0}_0) - F_{\mu_{I}}(x^{0}_{I}) + (\mu_{I} - \mu_{0}) (h(x_0^0)-\nu) \notag \\
& \stackrel{\eqref{eq:potential}}{=} f(x^{0}_0) + \mu_0 h(x^{0}_0)- f(x^{0}_{I}) - \mu_{I}h(x^{0}_{I}) + (\mu_{I} - \mu_{0}) (h(x_0^0)-\nu) \notag \\
& \stackrel{\eqref{eq:analytic_center}}{\leq} f(x^{0}_0) - f_{\min}(\setX) + \mu_0 h(x^{0}_0) - \mu_{I}h(x^{0}_{0}) + \mu_{I} \nu + (\mu_{I} - \mu_{0}) (h(x_0^0)-\nu) \notag \\ 
& \leq f(x^{0}_0) - f_{\min}(\setX) + \mu_0 \nu = f(x^{0}_0) - f_{\min}(\setX) + \eps_0.
\label{eq:PF_proof_3} 
\end{align}
Moreover, based on our updating choice $L_0^{(i+1)}=\hat{M}_{i}/2$, it holds that 
\begin{align*}
\bar{M}_i &= \max\{M,L_0^{(i)}\} = \max\{M,\hat{M}_{i-1}/2\}\\
& = \max\{M,L_{K_{i-1}}^{(i-1)}/2\} \leq \max\{M,\bar{M}_{i-1}\} \leq ... \leq \max\{M, \bar{M}_{0}\} \leq  \max\{M,L_0^{(0)}\}.
\end{align*}
Hence, 
%$\bar{M}_i = \max\{M,L_0^{(i)}\} = \max\{M,\hat{M}_{i-1}/2\} = \max\{M,L_{K_i-1}^{i-1}/2\} \leq \max\{M,\bar{M}_{i-1}\} \leq ... \leq \max\{M, \bar{M}_{0}\}=  \max\{M,L_0\}$
%$\bar{M}_i = \max\{M,\hat{M}_i\}\leq \max\{M,L_0\}=\bar{M}$, we obtain for each $i=0,...,I-1$ that
\begin{align}
K_i \leq  4(f(x^{0}) - f_{\min}(\setX)+  \eps_0) \cdot \frac{\nu^{2}(\bar{M}_i+\eps_i/\nu)}{\eps_i^{2}} \leq \frac{C}{\eps_i^2},
\label{eq:PF_proof_4} 
\end{align}
where $C\equiv 4(f(x^{0})-f_{\min}(\setX)+\eps_0)\nu^2(\max\{M,L_0^{(0)}\}+\eps_0/\nu)$. Finally, we obtain that the total number of iterations of procedures $\AHBA(\mu_{i},\eps_{i},L_0^{(i)},x_{i}^{0}),0\leq i \leq I-1$, to reach accuracy $\eps$ is at most
\begin{align*}
\sum_{i=0}^{I-1}K_{i}&\leq \sum_{i=0}^{I-1} \frac{C}{\eps_i^2} \leq \frac{C}{\eps_0^2} \sum_{i=0}^{I-1} (2^i)^2 
\leq \frac{C}{3\eps_0^2} \cdot (4^{2+\log_2(\frac{\eps_0}{\eps})}) = \frac{16C}{3\eps^2}\\
&=\frac{64(f(x^{0})-f_{\min}(\setX)+\eps_0)\nu^2(\max\{M,L_0^{(0)}\}+\eps_0/\nu)}{3\eps^{2}}.
\end{align*}
\end{proof}

\section{A second-order Hessian-Barrier Algorithm}
\label{sec:secondorder}
%----------------------------------------------------------------------
%%% 2nd-order
%----------------------------------------------------------------------
% !TEX root = ./HBAConicMain.tex
%

In this section we introduce a second-order potential reduction method for problem \eqref{eq:Opt} under the assumption that the second-order Taylor expansion of $f$ on the set of feasible directions $\scrT_{x}$ defined in \eqref{eq:Dikinv} is sufficiently accurate in the geometry induced by $h \in\scrH_{\nu}(\setK)$. 
%%%%%%%%%%%%
\begin{assumption}[Local second-order smoothness]
\label{ass:2ndorder}
$f:\setE\to\R\cup\{+\infty\}$ is twice continuously differentiable on $\feas$ and there exists a constant $M>0$ such that, for all $x\in\feas$ and $v\in\scrT_{x}$, we have
\begin{equation}\label{eq:SO_Lipschitz_Gradient}
\norm{\nabla f(x+v)-\nabla f(x)-\nabla^{2}f(x)v}^{\ast}_{x}\leq\frac{M}{2}\norm{v}^{2}_{x}.
\end{equation}
\end{assumption}
A sufficient condition for \eqref{eq:SO_Lipschitz_Gradient} is the following local counterpart of the global Lipschitz condition on
the Hessian of $f$:
\begin{equation}\label{eq:LipHess}
(\forall x\in\setX)(\forall u,v\in\scrF_{x}):\;\norm{\nabla^{2}f(x+u)-\nabla^{2}f(x+v)}_{\text{op},x}\leq M\norm{u-v}_{x},
\end{equation}
where 
$\norm{\BB}_{\text{op},x}\eqdef\sup_{u:\norm{u}_{x}\leq 1}\left\{\frac{\norm{\BB u}_{x}^{\ast}}{\norm{u}_{x}}\right\}$ is the induced operator norm for a linear operator $\BB:\setE\to\setE^{\ast}$. Indeed, this condition implies \eqref{eq:SO_Lipschitz_Gradient}:
\begin{align*}
&\norm{\nabla f(x+v)-\nabla f(x)-\nabla^{2}f(x)v}^{\ast}_{x}=\norm{\int_{0}^{1}(\nabla^{2}f(x+tv)-\nabla^{2}f(x))v\dif t }^{\ast}_{x}\\
&\leq \int_{0}^{1}\norm{\nabla^{2}f(x+tv)-\nabla^{2}f(x)}_{\text{op},x}\cdot \norm{v}_{x}\dif t   \leq \frac{M}{2}\norm{v}^{2}_{x}.
\end{align*}
%\begin{align*}
%&\norm{\nabla f(x+v)-\nabla f(x)-\nabla^{2}f(x)v}^{\ast}_{x}=\norm{\int_{0}^{1}(\nabla^{2}f(x+tv)-\nabla^{2}f(x))v\dif t }^{\ast}_{x}\\
%&\quad \leq\int_{0}^{1}\norm{(\nabla^{2}f(x+tv)-\nabla^{2}f(x))v}_{x}^{\ast}\dif t \leq \int_{0}^{1}\norm{\nabla^{2}f(x+tv)-\nabla^{2}f(x)}_{\text{op},x}\cdot \norm{v}_{x}\dif t  \\
%&\quad \quad \leq \frac{M}{2}\norm{v}^{2}_{x}.
%\end{align*}
Further, \eqref{eq:SO_Lipschitz_Gradient} in turn implies another important estimate 
\begin{equation}\label{eq:cubicestimate}
f(x+v)-\left[f(x)+\inner{\nabla f(x),v}+\frac{1}{2}\inner{\nabla^{2}f(x)v,v}\right]\leq\frac{M}{6}\norm{v}^{3}_{x}.
\end{equation}
Indeed, for all $x\in\setX$ and $v\in\scrT_{x}$,
\begin{align*}
&\abs{f(x+v)-f(x)-\inner{\nabla f(x),v}-\frac{1}{2}\inner{\nabla^{2}f(x)v,v}}=\abs{\int_{0}^{1}\inner{\nabla f(x+tv)-\nabla f(x)-\frac{1}{2}\nabla^{2}f(x)v,v}\dif t}\\
&\quad\leq \int_{0}^{1}\norm{\nabla f(x+tv)-\nabla f(x)-\frac{1}{2}\nabla^{2}f(x)v}_{x}^{*}\dif t\cdot\norm{v}_{x} 
\leq \frac{M}{6}\norm{v}^{3}_{x}.
\end{align*}
%%%%%%%%%%%%%%%%%%%%%%%%%%%
\begin{remark}
\label{Rm:Hess_Lip}
Assumption \ref{ass:2ndorder} subsumes, when $\bar{\setX}$ is bounded, the standard Lipschitz-Hessian setting since 
if the Hessian of $f$ is Lipschitz with modulus $M$ with respect to the Euclidean norm, we have by \cite[Eq. (2.2)]{NesPol06}.
\[
\norm{\nabla f(x+v)-\nabla f(x)-\nabla^{2}f(x)v}\leq\frac{M}{2}\norm{v}^{2}.
\]
Since $\bar{\setX}$ is bounded, one can observe that $\lambda_{\max}([H(x)]^{-1})^{-1}=\lambda_{\min}(H(x)) \geq \sigma$ for some $\sigma >0$, and \eqref{eq:SO_Lipschitz_Gradient} holds. Indeed, denoting $g=\nabla f(x+v)-\nabla f(x)-\nabla^{2}f(x)v$, we obtain
\[
(\norm{g}_x^*)^2 \leq \lambda_{\max}([H(x)]^{-1})\norm{g}^{2} \leq \frac{M^2}{4\lambda_{\min}(H(x))}\norm{v}^{4} \leq \frac{M^2}{4\sigma^{3}}\norm{v}_x^4.
\]
\close
\end{remark}

\begin{remark}
The cubic overestimation of the objective function in \eqref{eq:cubicestimate} does not rely on global second order differentiability assumptions. To illustrate this we invoke again the structured composite optimization problem \eqref{eq:composite}, assuming that the data fidelity function $\ell$ is twice continuously differentiable on an open neighborhood containing $\setX$, with Lipschitz continuous Hessian $\nabla^{2}\ell$ with modulus $\gamma$ w.r.t. the Euclidean norm. On the domain $\setK_{\text{NN}}$ we employ the canonical barrier $h(x)=-\sum_{i=1}^{n}\ln(x_{i})$, with $H(x)=\diag\{x_{1}^{-2},\ldots,x_{n}^{-2}\}=\XX^{-2}$. This means, for all $x,x^{+}\in\setX$, we have 
\[
\ell(x^{+})\leq\ell(x)+\inner{\nabla\ell(x),x^{+}-x}+\frac{1}{2}\inner{\nabla^{2}\ell(x)(x^{+}-x),x^{+}-x}+\frac{\gamma}{6}\norm{x^{+}-x}^{3}.
\]
As penalty function, we again consider the $L_{p}$ regularizer with $p\in(0,1)$. For any $t,s>0$, one has 
\[
t^{p}\leq s^{p}+ps^{p-1}(t-s)+\frac{p(p-1)}{2} s^{p-2}(t-s)^{2}+\frac{p(p-1)(p-2)}{6}s^{p-3}(t-s)^{3}.
\]
Since $v\in\scrT_{x}$ if and only if $v=[H(x)]^{-1/2}d=\XX d$ for some $d\in\R^{\dim(\setE)}$ satisfying $\bA\XX d=0$ and $\norm{d}<1$. Since $p(1-p)\leq 1/4$, it follows $p(1-p)(2-p)\leq 1/2$. Thus, using $x^{+}=x+v=x+\XX d$, we get 
\begin{align*}
f(x^{+}) &- \left((f(x)+\inner{\nabla f(x),\XX d}+\frac{1}{2}\inner{\nabla^{2}f(x)\XX d,\XX d}\right) \leq \frac{\gamma}{6}\norm{\XX d}^{3}+\frac{1}{12}\sum_{i=1}^{n}x^{p}_{i}d_{i}^{3}\\
&\leq \frac{\gamma}{6}\norm{\XX d}^{3}+\frac{1}{12}\norm{x}^{p}_{\infty}\sum_{i=1}^{n}d_{i}^{3} \leq \frac{1}{6}\left(\gamma\norm{x}_{\infty}^{3}+\frac{1}{2}\norm{x}^{p}_{\infty}\right)\norm{d}^{3}.
\end{align*}
%\begin{align*}
%f(x^{+})&\leq f(x)+\inner{\nabla f(x),\XX d}+\frac{1}{2}\inner{\nabla^{2}f(x)\XX d,\XX d}+\frac{\gamma}{6}\norm{\XX d}^{3}+\frac{1}{12}\sum_{i=1}^{n}x^{p}_{i}d_{i}^{3}\\
%&\leq f(x)+\inner{\nabla f(x),\XX d}+\frac{1}{2}\inner{\nabla^{2}f(x)\XX d,\XX d}+\frac{\gamma}{6}\norm{\XX d}^{3}+\frac{1}{12}\norm{x}^{p}_{\infty}\sum_{i=1}^{n}d_{i}^{3}\\
%&\leq f(x)+\inner{\nabla f(x),\XX d}+\frac{1}{2}\inner{\nabla^{2}f(x)\XX d,\XX d}+\frac{1}{6}\left(\gamma\norm{x}_{\infty}^{3}+\frac{1}{2}\norm{x}^{p}_{\infty}\right)\norm{d}^{3}
%\end{align*}
Assuming that $\bar{\setX}$ is bounded, there exists a universal constant $M>0$ such that $\gamma\norm{x}_{\infty}^{2}+\frac{1}{2}\norm{x}^{p}_{\infty}\leq M$. Combining this with Remark \ref{Rm:Hess_Lip}, we obtain a cubic overestimation as in eq. \eqref{eq:cubicestimate}. Importantly, $f(x)$ is not differentiable for $x \in \{x_i=0,\text{ for some } i\}$.  
\close
\end{remark}

We emphasize that in Assumption \ref{ass:2ndorder} the constant $M$ is in general unknown or may be a conservative upper bound. Therefore, adaptive techniques should be used to estimate it and are likely to improve the practical performance of the method. Assumption \ref{ass:2ndorder} also implies, by \eqref{eq:cubicestimate} and \eqref{eq:Dbound} (with $d=v$ and $t=1< \frac{1}{\norm{v}_x} \stackrel{\eqref{eq:boundzeta}}{\leq} \frac{1}{\zeta(x,v)}$), the following upper bound for the potential function $F_{\mu}$. 
\begin{lemma}[Cubic Overestimation]
\label{lem:cubic}
For all $x\in\setX,v\in\scrT_{x}$ and $L\geq M$, we have 
\begin{equation}\label{eq:cubicdecrease}
F_{\mu}(x+v)\leq F_{\mu}(x)+\inner{\nabla F_{\mu}(x),v}+\frac{1}{2}\inner{\nabla^{2}f(x)v,v}+\frac{L}{6}\norm{v}^{3}_{x}+\mu\norm{v}^{2}_{x}\omega(\zeta(x,v)).
\end{equation}
\end{lemma}

\subsection{Algorithm description and its complexity theorem}
\label{S:SO_alg_descr}
Let $x\in\setX$ be given. In order to find a search direction, we choose a parameter $L>0$, construct a cubic-regularized model of the potential $F_{\mu}$ \eqref{eq:potential}, and minimize it on the linear subspace $\setL_{0}$:
\begin{equation}\label{eq:cubicproblem}
v_{\mu,L}(x)\in\Argmin_{v\in\setE:\bA v=0}\left\{ Q^{(2)}_{\mu,L}(x,v)\eqdef F_{\mu}(x)+\inner{\nabla F_{\mu}(x),v}+\frac{1}{2}\inner{\nabla^{2}f(x)v,v}+\frac{L}{6}\norm{v}_{x}^{3} \right\},
\end{equation}
where by $\Argmin$ we denote the set of global minimizers. The model consists of three parts: linear approximation of $h$, quadratic approximation of $f$, and a cubic regularizer with penalty parameter $L>0$. Since this model and our algorithm use the second derivative of $f$, we call it a second-order method.
Our further derivations rely on the first-order optimality conditions for the problem \eqref{eq:cubicproblem}, which say that there exists $y_{\mu,L}(x)\in\R^{m}$ such that $v_{\mu,L}(x)$ satisfies
\begin{align}
\nabla F_{\mu}(x)+\nabla^{2}f(x)v_{\mu,L}(x)+\frac{L}{2}\norm{v_{\mu,L}(x)}_{x}H(x)v_{\mu,L}(x) - \bA^{\ast} y_{\mu,L}(x) &= 0,\label{eq:opt1}\\
 - \bA v_{\mu,L}(x)&=0. \label{eq:opt2}
\end{align}
We also use the following extension of \cite[Prop. 1]{NesPol06} to our setting with the local norm induced by $H(x)$.
\begin{proposition}
For all $x\in\feas$ it holds 
\begin{equation}\label{eq:PD}
\nabla^{2}f(x)+\frac{L}{2}\norm{v_{\mu,L}(x)}_{x}H(x)\succeq 0\qquad\text{ on }\;\;\setL_{0}.
\end{equation}
\end{proposition}
\begin{proof}
The proof follows the same strategy as Lemma 3.2 in \cite{CarGouToi11a}. Let $\{z_{1},\ldots,z_{p}\}$ be an orthonormal basis of $\setL_{0}$ and the linear operator $\bZ:\R^{p}\to\setL_{0}$ be defined by $\bZ w=\sum_{i=1}^{p}z_{i}w^{i}$ for all $w=[w^{1};\ldots;w^{p}]^{\top}\in\R^{p}$. With the help of this linear map, we can absorb the null-space restriction, and formulate the search-direction finding problem \eqref{eq:cubicproblem} using the projected data
\begin{equation}\label{eq:dataKernel}
\gbold\eqdef\bZ^{\ast}\nabla F_{\mu}(x),\; \bJ\eqdef\bZ^{\ast}\nabla^{2}f(x)\bZ,\;\bH\eqdef\bZ^{\ast}H(x)\bZ \succ 0.
\end{equation}
We then arrive at the cubic-regularized subproblem to find $u_{L}\in\R^{p}$ s.t.
\begin{equation}\label{eq:cubicauxiliary}
u_{L}\in  \Argmin_{u\in\R^{p}}\{\inner{\gbold,u}+\frac{1}{2}\inner{\bJ u,u}+\frac{L}{6}\norm{u}^{3}_{\bH}\},
\end{equation}
where  $\norm{\cdot}_{\bH}$ is the norm induced by the operator $\bH$. From \cite[Thm. 10]{NesPol06} we deduce  
\[
\bJ+\frac{L\norm{u_{L}}_{\bH}}{2}\bH\succeq 0.
\]
%where $u_{L}\in\R^{p}$ represents one global solution of problem \eqref{eq:cubicauxiliary}. 
Denoting $v_{\mu,L}(x) = \bZ u_{L}$, we see
\begin{align*}
\norm{u_{L}}_{\bH}=\inner{\bZ^{\ast}H(x)\bZ u_{L},u_{L}}^{1/2}=\inner{H(x)(\bZ u_{L}),\bZ u_{L}}^{1/2}&=\norm{v_{\mu,L}(x)}_{x}, \text{  and} \\
\bZ^{\ast}\left(\nabla^{2}f(x)+\frac{L}{2}\norm{v_{\mu,L}(x)}_{x}H(x)\right)\bZ&\succeq 0,
\end{align*}
which implies $\nabla^{2}f(x)+\frac{L}{2}\norm{v_{\mu,L}(x)}_{x}H(x)\succ 0$ over the null space $\setL_{0} = \{ v\in\setE:\bA v=0\}$. 
\end{proof}
\noindent
The above proposition gives some ideas on how one could numerically solve problem \eqref{eq:cubicproblem} in practice. In a preprocessing step, we once calculate matrix $\bZ$ and use it during the whole algorithm execution. At each iteration we calculate the new data using \eqref{eq:dataKernel}, leaving us with a standard \textit{unconstrained} cubic subproblem  \eqref{eq:cubicauxiliary}. \cite{NesPol06} show how such problems can be transformed to a \emph{convex} problem to which fast convex programming methods could in principle be applied. However, we can also solve it via recent efficient methods based on Lanczos' method \cite{CarGouToi11a,Jia22}. Whatever numerical tool is employed, we can recover our search direction by $v_{\mu,L}(x)$ by the matrix vector product $\bZ u_{L}$ in which $u_{L}$ denotes the solution obtained from this subroutine.
%%%%%%%%%%%%%%%%%%%%%
%\begin{remark}
%Our framework is formulated under the unrealistic assumption that the search direction can be computed exactly. An important extension of the current analysis will be to consider inexact computations. We leave this important question for future research.
%\close
%\end{remark}
%%%%%%%%%%%%%%

%step-size
Our next goal is to construct an admissible step-size policy,  given the search direction $v_{\mu,L}(x)$. 
Let $x\in\setX$ be the current position of the algorithm. Define the parameterized family of arcs $x^{+}(t)\eqdef x+t v_{\mu,L}(x)$, where $t\geq 0$ is a step-size. 
By \eqref{eq:step_length_zeta} and since $v_{\mu,L}(x)\in\setL_{0}$ by \eqref{eq:opt2}, we know that $x^{+}(t)$ is in $\setX$ provided that $t\in I_{x,\mu,L}\eqdef[0,\frac{1}{\zeta(x,v_{\mu,L}(x))})$. For all such $t$, Lemma \ref{lem:cubic} yields  
\begin{equation}\label{eq:SO_potent_upper_bound}
\begin{split}
F_{\mu}(x^{+}(t))\leq F_{\mu}(x)&+t\inner{\nabla F_{\mu}(x),v_{\mu,L}(x)} + \frac{t^2}{2}\inner{\nabla^{2}f(x)v_{\mu,L}(x),v_{\mu,L}(x)} \\
&+ \frac{Mt^3}{6}\norm{v_{\mu,L}(x)}^{3}_{x}  +\mu t^{2}\omega(t\zeta(x,v_{\mu,L}(x))).
\end{split}
\end{equation}
Since $v_{\mu,L}(x) \in \setL_0$, multiplying \eqref{eq:PD} with $v_{\mu,L}(x)$ from the left and the right, and multiplying \eqref{eq:opt1} by $v_{\mu,L}(x)$ and combining with \eqref{eq:opt2}, we obtain
\begin{align}
&\inner{\nabla^{2}f(x)v_{\mu,L}(x),v_{\mu,L}(x)}\geq-\frac{L}{2}\norm{v_{\mu,L}(x)}^{3}_{x},\label{eq:descent1}\\
&\inner{\nabla F_{\mu}(x),v_{\mu,L}(x)}+\inner{\nabla^{2}f(x)v_{\mu,L}(x),v_{\mu,L}(x)}+\frac{L}{2}\norm{v_{\mu,L}(x)}^{3}_{x}=0.\label{eq:normal}
\end{align}
Under the additional assumption that $t\leq 2$ and $L\geq M$, we obtain
\begin{align*}
&t\inner{\nabla F_{\mu}(x),v_{\mu,L}(x)} + \frac{t^2}{2}\inner{\nabla^{2}f(x)v_{\mu,L}(x),v_{\mu,L}(x)} + \frac{Mt^3}{6}\norm{v_{\mu,L}(x)}^{3}_{x}\\
&\stackrel{\eqref{eq:normal}}{=} - t \left(\inner{\nabla^{2}f(x)v_{\mu,L}(x),v_{\mu,L}(x)}+\frac{L}{2}\norm{v_{\mu,L}(x)}^{3}_{x}\right) &\\
&+ \frac{t^2}{2}\inner{\nabla^{2}f(x)v_{\mu,L}(x),v_{\mu,L}(x)} + \frac{Mt^3}{6}\norm{v_{\mu,L}(x)}^{3}_{x} \\
& = \left(\frac{t^2}{2} - t \right) \inner{\nabla^{2}f(x)v_{\mu,L}(x),v_{\mu,L}(x)} - \frac{Lt}{2}\norm{v_{\mu,L}(x)}^{3}_{x} + \frac{Mt^3}{6}\norm{v_{\mu,L}(x)}^{3}_{x} \\
& \stackrel{\eqref{eq:descent1},t\leq 2}{\leq} \left(\frac{t^2}{2} - t \right) \left(- \frac{L}{2}\norm{v_{\mu,L}(x)}^{3}_{x} \right) - \frac{Lt}{2}\norm{v_{\mu,L}(x)}^{3}_{x} + \frac{Mt^3}{6}\norm{v_{\mu,L}(x)}^{3}_{x} \\
& = - \norm{v_{\mu,L}(x)}^{3}_{x} \left(\frac{Lt^2}{4} - \frac{Mt^3}{6} \right) 
\stackrel{L \geq M}{\leq} - \norm{v_{\mu,L}(x)}^{3}_{x} \frac{Lt^2}{12} \left(3 - 2t \right).
\end{align*}
Substituting this into \eqref{eq:SO_potent_upper_bound}, we arrive at
\begin{align*}
F_{\mu}(x^{+}(t))&\leq F_{\mu}(x) - \norm{v_{\mu,L}(x)}^{3}_{x} \frac{Lt^2}{12} \left(3 - 2t \right)+\mu t^{2}\omega(t\zeta(x,v_{\mu,L}(x)))\\ 
&\stackrel{\eqref{eq:omega_upper_bound}}{\leq}  F_{\mu}(x) - \norm{v_{\mu,L}(x)}^{3}_{x} \frac{Lt^2}{12} \left(3 - 2t \right)+\mu \frac{t^{2}\norm{v_{\mu,L}(x)}_{x}^{2}}{2(1-t\zeta(x,v_{\mu,L}(x))}. 
\end{align*} 
for all  $t \in I_{x,\mu,L}$. Therefore, if $t\zeta(x,v_{\mu,L}(x))\leq 1/2$, we readily see%\footnote{
%Note that since $t\zeta(x,v_{\mu,L}(x))=\zeta(x,tv_{\mu,L}(x))$ and we assumed that $t\zeta(x,v_{\mu,L}(x))\leq 1/2$, it is sufficient to make Assumption \ref{ass:2ndorder} on a potentially smaller set $\widetilde{\scrT}_{x}\eqdef \{v\in\setE\vert \bA v=0,\zeta(x,v)\leq 1/2\}$ instead of $\scrT_x$ defined in \eqref{eq:Dikinv}.
%}  
\begin{align}
F_{\mu}(x^{+}(t))-F_{\mu}(x)&\leq - \frac{Lt^2\norm{v_{\mu,L}(x)}^{3}_{x} }{12} \left(3 - 2t \right)+\mu t^{2}\norm{v_{\mu,L}(x)}_{x}^{2} \nonumber \\
&= - \norm{v_{\mu,L}(x)}^{3}_{x} \frac{Lt^2}{12}\left(3 - 2t - \frac{12 \mu}{L\norm{v_{\mu,L}(x)}_{x}} \right) \eqdef -\eta_{x}(t).\label{eq:progressSOM}
\end{align}
Maximizing the above function $\eta_{x}(t)$ and finding a lower bound for its optimal value is technically quite challenging. Instead, we adopt the following step-size rule
\begin{equation}
\label{eq:SO_t_opt_def}
\ct_{\mu,L}(x) \eqdef \frac{1}{\max\{1,2\zeta(x,v_{\mu,L}(x))\}} = \min \left\{1, \frac{1}{2\zeta(x,v_{\mu,L}(x))}  \right\}.
\end{equation}  
Note that $\ct_{\mu,L}(x) \leq 1$ and $\ct_{\mu,L}(x)\zeta(x,v_{\mu,L}(x))\leq 1/2$. Thus, this choice of the step-size is feasible to derive \eqref{eq:progressSOM}. 
%%%%%%%%%%%%

Just like Algorithm \ref{alg:AHBA}, our second-order method employs a line-search procedure to estimate the Lipschitz constant $M$ in \eqref{eq:SO_Lipschitz_Gradient}, \eqref{eq:cubicestimate} in the spirit of \cite{NesPol06,CarGouToi12}. More specifically, suppose that $x^{k}\in\setX$ is the current position of the algorithm with the corresponding initial local Lipschitz estimate $M_{k}$. To determine the next iterate $x^{k+1}$, we solve problem \eqref{eq:cubicproblem} with $L= L_k = 2^{i_k}M_{k}$ starting with $ i_k =0$, find the corresponding search direction $v^{k}=v_{\mu,L_{k}}(x^{k})$ and the new point $x^{k+1} = x^{k} + \ct_{\mu, L_{k}}(x^{k})v^{k}$. 
Then, we check whether the inequalities \eqref{eq:SO_Lipschitz_Gradient} and \eqref{eq:cubicestimate} hold with  $M=L_{k}$, $x=x^{k}$, $v = \ct_{\mu, L_{k}}(x^{k})v^{k}$, see \eqref{eq:SO_LS_2} and \eqref{eq:SO_LS_1}. 
If they hold, we make a step to $x^{k+1}$. 
Otherwise, we increase $i_k$ by 1 and repeat the procedure. Obviously, when $L_{k} = 2^{i_k}M_k \geq M$, both inequalities \eqref{eq:SO_Lipschitz_Gradient} and \eqref{eq:cubicestimate} with $M$ changed to $L_k$, i.e., \eqref{eq:SO_LS_2} and \eqref{eq:SO_LS_1}, are satisfied and the line-search procedure ends. For the next iteration we set $M_{k+1} = \max\{2^{i_k-1}M_{k},\underline{L}\}=\max\{L_{k}/2,\underline{L}\}$, so that the estimate for the local Lipschitz constant on the one hand can decrease allowing larger step-sizes, and on the other hand is bounded from below. 
The resulting procedure gives rise to a \underline{S}econd-order \underline{A}daptive \underline{H}essian-\underline{B}arrier \underline{A}lgorithm ($\SAHBA$, Algorithm \ref{alg:SOAHBA}).
%%%%%%%%%%%%%%%%%%%%%%%%%%%%%%%%
\begin{algorithm}[t!]
\caption{\underline{S}econd-order \underline{A}daptive \underline{H}essian-\underline{B}arrier \underline{A}lgorithm - $\SAHBA(\mu,\eps,M_{0},x^{0})$}
\label{alg:SOAHBA}
\SetAlgoLined
\KwData{ $h \in\scrH_{\nu}(\setK)$, $\mu>0,\eps>0,M_0\geq 144  \eps,x^{0}\in\setX$.}
\KwResult{$(x^{k},y^{k-1},s^{k},M_{k})\in\setX\times\R^{m}\times\setK^{\ast}\times\R_{+}$, where $s^{k}=\nabla f(x^{k}) -\bA^{\ast}y^{k-1}$, and $M_{k}$ is the last estimate of the Lipschitz constant.}
Set $\underline{L} \eqdef 144  \eps$, $k=0$\;
%Set $144  \eps \eqdef \underline{L} < M_0$ -- initial guess for $M$\;% $\mu = \frac{\eps}{4\nu}$, $k=0$, $x^{0}\in\setX$ -- $4\nu$-analytic center (see \eqref{eq:analytic_center})
\Repeat{
		$\norm{v^{k-1}}_{x^{k-1}} < \Delta_{k-1}\eqdef\sqrt{\frac{\eps}{4L_{k-1}\nu}}$ \textnormal{ and }$\|v^{k}\|_{x^{k}} < \Delta_{k}\eqdef\sqrt{\frac{\eps}{4L_{k}\nu}}$
	}{	
		Set $i_k=0$.
		
		\Repeat{
		\begin{align}
			& f(z^{k}) \leq f(x^{k}) + \inner{\nabla f(x^{k}),z^{k}-x^{k}}+\frac{1}{2}\inner{\nabla^{2}f(x^{k})(z^{k}-x^{k}),z^{k}-x^{k}} \notag \\
		&\qquad\qquad\qquad\qquad\qquad+\frac{L_{k}}{6}\norm{z^{k}-x^{k}}^{3}_{x^{k}}, \quad \textnormal{and } \label{eq:SO_LS_1}\\
		&  \norm{\nabla f(z^{k})-\nabla f(x^{k})-\nabla^{2}f(x^{k})(z^{k}-x^{k})}^{\ast}_{x^{k}}\leq\frac{L_{k}}{2}\norm{z^{k}-x^{k}}^{2}_{x^{k}}. \label{eq:SO_LS_2}
		\end{align}
		}
		{
			Set $L_k = 2^{i_k}M_k$. Find $v^k \eqdef v_{\mu,L_k}(x^{k})$ and $y^k \eqdef y_{\mu,L_k}(x^{k})$ as a global solution to
			\begin{align}
				&\hspace{-2em} \min_{v:\bA v=0} \left\{F_{\mu}(x^k)+\inner{\nabla F_{\mu}(x^k),v}+\frac{1}{2}\inner{\nabla^{2}f(x^k)v,v}+\frac{L_{k}}{6}\norm{v}_{x^k}^{3} \right\}. \label{eq:SO_finder} \\
				&\hspace{-1em} \text{Set }\;\; \alpha_k\eqdef\min \left\{1, \frac{1}{2\zeta(x^k,v^k)}  \right\},  	\text{where $\zeta(\cdot,\cdot)$ as in \eqref{eq:zeta}}. \label{eq:SO_alpha_k}
			\end{align}
					
			Set $z^{k}=x^{k} + \alpha_k v^k$, $i_k=i_k+1$\;
		}
		Set $M_{k+1} =\max\{\frac{L_{k}}{2},\underline{L}\}$, %\max\{2^{i_k-1}M_{k},\underline{L}\}$,  
		$x^{k+1}=z^{k}$, $k=k+1$%i_k-2 is since when we find a suitable $i_k$ we still set $i_k=i_k+1$
	}
\end{algorithm}
Our main result on the iteration complexity of Algorithm \ref{alg:SOAHBA} is the following Theorem, whose proof is given in Section \ref{sec:ProofSOM}. 
\begin{theorem}
\label{Th:SOAHBA_conv}
Let Assumptions \ref{ass:1}, \ref{ass:barrier}, and \ref{ass:2ndorder} hold. Fix the error tolerance $\eps>0$, the regularization parameter $\mu= \frac{\eps}{4\nu}$, and some initial guess $M_0>144\eps$ for the Lipschitz constant. Let $(x^{k})_{k\geq 0}$ be the trajectory generated by $\SAHBA(\mu,\eps,M_{0},x^{0})$, where $x^{0}$ is a $4\nu$-analytic center satisfying \eqref{eq:analytic_center}.
Then the algorithm stops in no more than 
\begin{equation}
\label{eq:SO_main_Th_compl}
\K_{II}(\eps,x^{0})= \ceil[\bigg]{\frac{192 \nu^{3/2} \sqrt{2\max\{M,M_0\}}(f(x^{0}) - f_{\min}(\setX)+ \eps)}{\eps^{3/2} }}
\end{equation}
outer iterations, and the number of inner iterations is no more than $2(\K_{II}(\eps,x^{0})+1)+2\max\{\log_{2}(2M/M_{0}),1\}$. Moreover, the output of $\SAHBA(\mu,\eps,M_{0},x^{0})$ constitute an $(\eps,\frac{\max\{M,M_0\}\eps}{8\nu})$-2KKT point for problem \eqref{eq:Opt} in the sense of Definition \ref{def:eps_SOKKT}.
%\begin{align}
%&\|\nabla f(x^{k}) - \bA^{\ast}y^{k-1} - s^{k} \| = 0 \leq \eps, \label{eq:SO_main_Th_eps_KKT_1} \\
%& |\inner{s^{k},x^{k}}| \leq \eps,  \label{eq:SO_main_Th_eps_KKT_2}  \\
%& \bA x^{k}=b, s^{k} \in \setK^{\ast}, x^{k}\in\setK \label{eq:SO_main_Th_eps_KKT_3}  \\
%& \nabla^2f(x^k)  + H(x^k) \sqrt{\frac{M\eps}{\nu}}  \succeq 0 \;\; \text{on} \;\; \setL_0. \label{eq:SO_main_Th_eps_KKT_4} 
%\end{align}  
\end{theorem}
\begin{remark}
\label{rem:SO_complexity_simplified}
Since $f(x^{0}) - f_{\min}(\setX)$ is expected to be larger than $\eps$, and the constant $M$ is potentially large, we see that the main term in the complexity bound \eqref{eq:SO_main_Th_compl} is $O\left(\frac{\nu^{3/2}\sqrt{M}(f(x^{0}) - f_{\min}(\setX))}{\eps^{3/2}}\right)=O((\frac{\nu}{\eps})^{3/2})$.
Note that the complexity result  $O(\max\{\eps_1^{-3/2},\eps_2^{-3/2}\})$ reported in \cite{CarDucHinSid19b,CarDucHinSid19} to find an $(\eps_1,\eps_2)$-2KKT point for arbitrary $ \eps_1,\eps_2 > 0$, is known to be optimal for unconstrained smooth non-convex optimization by second-order methods under the standard Lipschitz-Hessian assumption. It can be easily obtained from our theorem by setting $\eps=\max\{\eps_1^{-3/2},\eps_2^{-3/2}\}$. 
\close
\end{remark}
%\begin{remark}
%An interesting observation is that our algorithm can be interpreted as a damped version of a cubic-regularized Newton's method. When $h \in\scrH_{\nu}(\setK)$ and we consider it as an element of $\scrH_{\nu}(\setK)$, we have that $\zeta(x^k,v^k) = \norm{v^k}_{x^k}$ and the stepsize $\alpha_k$ satisfies $\alpha_k= \min \left\{1, \frac{1}{2\norm{v^k}_{x^k}}  \right\}$. At the initial phase, when the algorithm is far from an $(\eps_1,\eps_2)$-2KKT point, we have $\norm{v^k}_{x^k} > 1/2$ and $\alpha_k = \frac{1}{2\norm{v^k}_{x^k}}<1$. When the algorithm is getting closer to $(\eps_1,\eps_2)$-2KKT point, $\norm{v^k}_{x^k}$ becomes smaller and the algorithm automatically switches to full steps $\alpha_k=1$. 
%
%At the same time our algorithm is completely different from cubic-regularized Newton's method \cite{NesPol06} applied to minimize the potential $F_{\mu}$. Indeed, we regularize by the cube of the local norm, rather than cube of the standard Euclidean norm, and we do not form a second-order Taylor expansion of $F_{\mu}$. These adjustments are needed to align the search direction subproblem with the local geometry of the feasible set. Moreover, for our algorithm, the analysis of the cubic-regularized Newton's method is not directly applicable since it relies on the stepsize 1, which may lead to infeasible iterates in our case.
%\close
%\end{remark}

\subsection{Proof of Theorem \ref{Th:SOAHBA_conv}}
\label{sec:ProofSOM}
%In this subsection we analyze the convergence of $\SAHBA$ and prove Theorem \ref{Th:SOAHBA_conv}. 
The main steps of the proof are similar to the analysis of Algorithm \ref{alg:AHBA}. We start by showing the feasibility of the iterates and correctness of the line-search process. Next, we analyze the per-iteration decrease of $F_{\mu}$ and $f$ and show that if the stopping criterion does not hold at iteration $k$, then the objective function is decreased by the value $O(\eps^{3/2})$. From this, since the objective is globally lower bounded, we conclude that the algorithm stops in  $O(\eps^{-3/2})$ iterations. Finally, we show that when the stopping criterion holds, the primal-dual pair $(x^{k}$, $y^{k-1})$ resulting from solving the cubic subproblem \eqref{eq:SO_finder} yields a dual slack variable $s^{k}$ such that this triple constitutes an second-order KKT point. 

\subsubsection{Interior point property of the iterates}
By construction $x^{0}\in\setX$. Proceeding inductively, let $x^{k}\in\setX$ be the $k$-th iterate of the algorithm, with the search direction $v^{k}\equiv v_{\mu,L}(x^{k})$.  By eq. \eqref{eq:SO_alpha_k}, the step-size $\alpha_k$ satisfies $\alpha_{k}\leq \frac{1}{2\zeta(x^{k},v^{k})}$. Consequently, $\alpha_{k}\zeta(x^{k},v^{k})\leq 1/2$ for all $k\geq 0$, and using \eqref{eq:step_length_zeta} as well as $\bA v^{k} =0$, we have that $x^{k+1}=x^{k}+\alpha_{k}v^{k}\in\setX$. By induction, it follows that $x^{k}\in\setX$ for all $k\geq 0$.

\subsubsection{Bounding the number of backtracking steps}
\label{sec:backtrack2}
%To bound the number of cycles involved in the line-search process for finding appropriate constants $L_{k}$, we proceed as in Section \ref{sec:backtrack1}. Since the derivations are analogous to the ones performed there, we skip the details, and just report here that the number of line-search iterations up to iteration $k$ of $\SAHBA(\mu,\eps,M_0,x^{0})$ is bounded as  
%\begin{align*}
%N(k)&=\sum_{j=0}^{k}(i_{j}+1)\leq i_0+1 + \sum_{j=1}^{k}\left(\log_{2}\left(\frac{L_{j}}{L_{j-1}}\right)+2\right) 
%\leq 2(k+1) + 2 \log_2\left(\frac{2\bar{M}}{M_0}\right),
%\end{align*}
%since $L_{k} \leq 2\bar{M} \eqdef 2\max\{M_0,M\}$ in the last step. Thus, on average, the inner loop ends after two trials. 
To bound the number of cycles involved in the line-search process for finding appropriate constants $L_{k}$, we proceed as in Section \ref{sec:backtrack1}. 
Let us fix an iteration $k$. The sequence $L_k = 2^{i_k} M_k$ is increasing as $i_k$ is increasing, and Assumption \ref{ass:2ndorder} holds. 
This implies \eqref{eq:cubicestimate}, and thus when $L_k = 2^{i_k} M_k \geq \max\{M,M_k\}$, the line-search process for sure stops since inequalities \eqref{eq:SO_LS_1} and \eqref{eq:SO_LS_2} hold.	
Hence, $L_k=2^{i_k} M_k \leq 2\max\{M,M_k\}$ must be the case, and, consequently, $M_{k+1}=\max\{L_k/2, \underline{L}\} \leq  \max\{\max\{M,M_k\}, \underline{L}\} = \max\{M,M_k\}$, which, by induction, gives $M_{k} \leq \bar{M}\equiv \max\{M,M_0\}$ and $L_k  \leq 2\bar{M}$. 
%
%
%$L_{k+1} = 2^{i_k-1} L_k \leq \max\{M,L_k\}$, which, by induction, gives $L_{k+1} \leq \bar{M}\eqdef\max\{M,L_0\}$. 
%Moreover, from this observation it follows that $L_k = 2^{i_k} M_k \leq 2\bar{M}=2\max\{M_0,M\}$. 
%
%
At the same time, by construction, $M_{k+1}= \max\{2^{i_k-1}M_{k},\underline{L}\} = \max\{L_k/2,\underline{L}\} \geq L_k/2 $. Hence, $L_{k+1} = 2^{i_{k+1}} M_{k+1} \geq 2^{i_{k+1}-1} L_k$ and therefore $\log_{2}\left(\frac{L_{k+1}}{L_{k}}\right)\geq i_{k+1}-1$, $\forall k\geq 0$. At the same time, at iteration $0$ we have $L_0=2^{i_0} M_0 \leq 2\bar{M}$, whence, $i_0 \leq \log_2\left(\frac{2\bar{M}}{M_0}\right)$.
Let $N(k)$ denote the number inner line-search iterations up to iteration $k$ of $\SAHBA$. Then, 
\begin{align*}
N(k)&=\sum_{j=0}^{k}(i_{j}+1)\leq i_0+1 + \sum_{j=1}^{k}\left(\log_{2}\left(\frac{L_{j}}{L_{j-1}}\right)+2\right) 
\leq 2(k+1) + 2 \log_2\left(\frac{2\bar{M}}{M_0}\right),
\end{align*}
since $L_{k} \leq 2\bar{M}= 2\max\{M,M_0\}$ in the last step. Thus, on average, the inner loop ends after two trials.

%Next, we give a bound on the number of cycles involved in the line-search process for finding appropriate constants $L_k$. This allows us to connect the iteration complexity estimate to the number of function evaluations in the worst-case. Let us fix an iteration $k$. The sequence $L_k = 2^{i_k} M_k$ is increasing as $i_k$ is increasing and Assumption \ref{ass:2ndorder} holds. This implies \eqref{eq:cubicestimate}, and thus when $L_k = 2^{i_k} M_k \geq M$, the line-search process for sure stops since inequalities \eqref{eq:SO_LS_1} and \eqref{eq:SO_LS_2} hold.	
%Moreover, from this observation it follows that $L_k = 2^{i_k} M_k \leq 2\bar{M}=2\max\{M_0,M\}$. This also allows us to estimate the total number of backtracking steps after $k$ iterations. By construction, $M_{k+1}= \max\{2^{i_k-1}M_{k},\underline{L}\} = \max\{L_k/2,\underline{L}\} \geq L_k/2 $. Hence, $L_{k+1} = 2^{i_{k+1}} M_{k+1} \geq 2^{i_{k+1}-1} L_k$ and therefore $\log_{2}\left(\frac{L_{k+1}}{L_{k}}\right)\geq i_{k+1}-1$, $\forall k\geq 0$. At the same time, at iteration $0$ we have $L_0=2^{i_0} M_0 \leq 2\bar{M}$, whence, $i_0 \leq \log_2\left(\frac{2\bar{M}}{M_0}\right)$.
%Let $N(k)$ denote the number inner line-search iterations up to iteration $k$ of $\SAHBA$. Then, 
%\begin{align*}
%N(k)&=\sum_{j=0}^{k}(i_{j}+1)\leq i_0+1 + \sum_{j=1}^{k}\left(\log_{2}\left(\frac{L_{j}}{L_{j-1}}\right)+2\right) 
%\leq 2(k+1) + 2 \log_2\left(\frac{2\bar{M}}{M_0}\right),
%\end{align*}
%since $L_{k} \leq 2\bar{M}$ in the last step. Thus, on average, the inner loop ends after two trials. 

\subsubsection{Per-iteration analysis and a bound for the number of iterations}
Let us fix iteration counter $k$. The main assumption of this subsection is that the stopping criterion is not satisfied, i.e. either
$\|v^{k}\|_{x^{k}} \geq \Delta_{k}$ or $\|v^{k-1}\|_{x^{k-1}} \geq \Delta_{k-1}$. 
Without loss of generality, we assume that the first inequality holds, i.e., $\|v^{k}\|_{x^{k}} \geq \Delta_{k}$, and consider iteration $k$. Otherwise, if the second inequality holds, the same derivations can be made considering the iteration $k-1$ and using the second inequality $\|v^{k-1}\|_{x^{k-1}} \geq \Delta_{k-1}$. Thus, at the end of the $k$-th iteration 
\begin{equation}
\label{eq:SO_per_iter_proof_1}
\|v^{k}\|_{x^{k}} \geq \Delta_{k}=\sqrt{\frac{\eps}{4L_{k}\nu}}.
\end{equation}
Since the step-size $\alpha_k= \min\{1,\frac{1}{ 2\zeta(x^k,v^k)}\} = \ct_{\mu,L_k}(x^{k})$ in \eqref{eq:SO_alpha_k} satisfies $\alpha_k \leq 1$ and $\alpha_k \zeta(x^k,v^k) \leq 1/2$ (cf. \eqref{eq:SO_t_opt_def} and a remark after it), we can repeat the derivations of Section \ref{S:SO_alg_descr}, 
changing  \eqref{eq:cubicestimate} to  \eqref{eq:SO_LS_1}.
%changing \eqref{eq:SO_Lipschitz_Gradient} and \eqref{eq:cubicdecrease} to \eqref{eq:SO_LS_2} and \eqref{eq:SO_LS_1} respectively. 
In this way we obtain the following counterpart of \eqref{eq:progressSOM} with $t=\alpha_k$, $L=L_k$, $x=x^k$, $v_{\mu,L_k}(x^{k}) \eqdef v^k$: 
\begin{align}
F_{\mu}(x^{k+1})-F_{\mu}(x^{k})&\leq - \norm{v^{k}}^{3}_{x^{k}} \frac{L_k\alpha_k^2}{12}\left(3 - 2\alpha_k - \frac{12 \mu}{L_k\norm{v^{k}}_{x^{k}}} \right)\nonumber\\
& \leq - \norm{v^{k}}^{3}_{x^{k}} \frac{L_k\alpha_k^2}{12}\left(1 - \frac{12 \mu}{L_k\norm{v^{k}}_{x^{k}}} \right)\label{eq:SO_per_iter_proof_2},
\end{align}
where in the last inequality we used that $\alpha_k \leq 1$ by construction. 
Substituting $\mu = \frac{\eps}{4\nu}$, and using \eqref{eq:SO_per_iter_proof_1}, we obtain
\begin{align*}
1 - \frac{12 \mu}{L_k\norm{v^{k}}_{x^{k}}} &= 1 - \frac{12 \eps}{4\nu L_k\norm{v^{k}}_{x^{k}}}  %\stackrel{\eqref{eq:SO_per_iter_proof_1}}{\geq} 1 - \frac{3 \beta^2\eps}{\nu L_k\Delta_k} \\
\stackrel{\eqref{eq:SO_per_iter_proof_1}}{\geq} 1 - \frac{3 \eps}{\nu L_k\sqrt{\frac{\eps}{4L_{k}\nu}}} \\
&= 1 - \frac{6 \sqrt{\eps}}{ \sqrt{\nu L_k}} \geq 1 - \frac{6 \sqrt{\eps}}{ \sqrt{ 144 \nu \eps	}}  \geq \frac{1}{2},
\end{align*}
using that, by construction, $L_k =2^{i_k}M_k \geq \underline{L} = 144 \eps$ and that $\nu \geq 1$. 
Hence, from \eqref{eq:SO_per_iter_proof_2}, 
\begin{equation}
\label{eq:SO_per_iter_proof_3}
F_{\mu}(x^{k+1})-F_{\mu}(x^{k})\leq  - \norm{v^{k}}^{3}_{x^{k}} \frac{L_k\alpha_k^2}{24}.
\end{equation}
%Consider two possible cases of the value of the step-size $\alpha_k$ in \eqref{eq:SO_alpha_k} and substitute it into \eqref{eq:SO_per_iter_proof_3}:
Substituting into \eqref{eq:SO_per_iter_proof_3} the two possible values of the step-size $\alpha_k$ in \eqref{eq:SO_alpha_k} gives
\begin{equation}
\label{eq:SO_per_iter_proof_8}
F_{\mu}(x^{k+1})-F_{\mu}(x^{k})\leq 
\left\{
\begin{array}{ll}
- \norm{v^{k}}^{3}_{x^{k}} \frac{L_k}{24}, & \text{if }  \alpha_k=1,\\
 -  \frac{L_k\norm{v^{k}}^{3}_{x^{k}}}{96 (\zeta(x^k,v^k))^2} \stackrel{\eqref{eq:boundzeta}}{\leq} -  \frac{L_k\norm{v^{k}}_{x^{k}}}{96}& \text{if } \alpha_k=\frac{1}{2\zeta(x^k,v^k)}.
\end{array}\right.
\end{equation}
This implies
%From  \eqref{eq:SO_per_iter_proof_4} and \eqref{eq:SO_per_iter_proof_6}, we obtain that
\begin{equation}
\label{eq:SO_per_iter_proof_7}
F_{\mu}(x^{k+1})-F_{\mu}(x^{k}) \leq - \frac{L_k\norm{v^{k}}_{x^{k}}}{96} \min\left\{1, 4\norm{v^{k}}_{x^{k}}^2 \right\} \eqdef-\delta_{k}.
\end{equation}
Rearranging and summing these inequalities for $k$ from $0$ to $K-1$, and using that $L_k \geq \underline{L}$, we obtain
\begin{align}
%\label{eq:}
K\min_{k=0,...,K-1} &\frac{\underline{L}\norm{v^{k}}_{x^{k}}}{96} \min\left\{1, 4\norm{v^{k}}_{x^{k}}^2\right\} \leq  \sum_{k=0}^{K-1} \delta_k 
\leq F_{\mu}(x^{0})-F_{\mu}(x^{K}) \notag \\
&\stackrel{\eqref{eq:potential}}{=} f(x^{0}) - f(x^{K}) + \mu (h(x^{0}) - h(x^{K})) \leq f(x^{0}) - f_{\min}(\setX) + \eps, \label{eq:SO_per_iter_proof_9}
\end{align}
where we used that, by the assumptions of Theorem \ref{Th:SOAHBA_conv}, $x^{0}$ is a $4\nu$-analytic center defined in \eqref{eq:analytic_center} and $\mu = \frac{\eps}{4\nu}$, implying that $h(x^{0}) - h(x^{K}) \leq 4\nu = \eps/\mu$.
Thus, up to passing to a subsequence, we have $\norm{v^{k}}_{x^{k}} \to 0$ as $k \to \infty$, which makes the stopping criterion in Algorithm \ref{alg:SOAHBA} achievable.

Assume now that the stopping criterion does not hold for $K$ iterations of $\SAHBA$. 
Then, for all $k=0,\ldots,K-1,$ it holds that 
\begin{align}
\delta_k &= \frac{L_k}{96} \min\left\{\norm{v^{k}}_{x^{k}}, 4\norm{v^{k}}_{x^{k}}^3 \right\} 
\stackrel{\eqref{eq:SO_per_iter_proof_1}}{\geq} \frac{L_k}{96} \min\left\{  \sqrt{\frac{\eps}{4L_{k}\nu}} ,  \frac{4 \eps^{3/2}}{4^{3/2}L_{k}^{3/2}\nu^{3/2}}  \right\} \notag \\
&\stackrel{L_k \leq 2\bar{M}, \nu \geq 1}{\geq}  \frac{1}{96} \min\left\{  \frac{L_k \sqrt{\eps}}{\sqrt{8\bar{M}} \nu^{3/2}} ,  \frac{ \eps^{3/2}}{2 L_{k}^{1/2}\nu^{3/2}}  \right\}  \notag \\
&\stackrel{L_k \leq 2\bar{M},L_k \geq 144  \eps}{\geq} \frac{1}{96} \min\left\{ \frac{(144 \eps) \cdot \sqrt{\eps}}{\sqrt{8\bar{M}} \nu^{3/2}} ,  \frac{ \eps^{3/2}}{\sqrt{8\bar{M}}\nu^{3/2}}  \right\}  = \frac{\eps^{3/2}}{192 \nu^{3/2} \sqrt{2\bar{M}} },
\end{align}
Thus, from \eqref{eq:SO_per_iter_proof_9}.
\begin{align*}
K \frac{\eps^{3/2}}{192 \nu^{3/2} \sqrt{2\bar{M}} } \leq f(x^{0}) - f_{\min}(\setX) + \eps. 
\end{align*}
Hence, reacalling that $\bar{M}=\max\{M_0,M\}$, $K \leq \frac{192 \nu^{3/2} \sqrt{2\max\{M_0,M\}}(f(x^{0}) - f_{\min}(\setX)+ \eps)}{\eps^{3/2} }$,
i.e. the algorithm stops for sure after no more than this number of iterations. This, combined with the bound for the number of inner steps in Section \ref{sec:backtrack2}, proves the first statement of Theorem \ref{Th:SOAHBA_conv}.

\subsubsection{Generating a $(\eps_{1},\eps_{2})$-2KKT point}
In this section, to finish the proof of Theorem \ref{Th:SOAHBA_conv}, we show that if the stopping criterion in Algorithm \ref{alg:SOAHBA} holds, i.e. $\norm{v^{k-1}}_{x^{k-1}} < \Delta_{k-1}$ and $\norm{v^{k}}_{x^{k}} < \Delta_{k}$, then the algorithm has generated an $(\eps_{1},\eps_{2})$-2KKT point of \eqref{eq:Opt} according to Definition \ref{def:eps_SOKKT}, with $\eps_{1}=\eps$ and $\eps_{2}=\frac{\max\{M_{0},\bar{M}\}\eps}{8\nu}$.

Let the stopping criterion hold at iteration $k$. Using the first-order optimality condition \eqref{eq:opt1} for the subproblem \eqref{eq:SO_finder} solved at iteration $k-1$, there exists a dual variable $y^{k-1}\in\R^{m}$ such that \eqref{eq:opt1} holds. Now, expanding the definition of the potential \eqref{eq:potential} and adding $\nabla f(x^{k})$ to both sides, we obtain
\begin{align*}
 \nabla f(x^{k}) - & \bA^{\ast}y^{k-1} + \mu \nabla h(x^{k-1}) \\
&=\nabla f(x^{k}) - \nabla f(x^{k-1}) - \nabla^2 f(x^{k-1})v^{k-1} - \frac{L_{k-1}}{2}\norm{v^{k-1}}_{x^{k-1}}H(x^{k-1})v^{k-1}.
\end{align*}
Setting $s^{k}\eqdef \nabla f(x^{k})-\bA^{\ast}y^{k-1} \in \setE^{\ast}$ and $g^{k-1}\eqdef-\mu\nabla h(x^{k-1})$, after multiplication by $[H(x^{k-1})]^{-1}$, this is equivalent to 
\begin{align*}
[H(x^{k-1})]^{-1}\left(s^{k}-g^{k-1}\right)&=[H(x^{k-1})]^{-1}\left(\nabla f(x^{k}) - \nabla f(x^{k-1}) - \nabla^2 f(x^{k-1})v^{k-1}\right)\\
& - \frac{L_{k-1}}{2}\norm{v^{k-1}}_{x^{k-1}}v^{k-1}.
\end{align*}
Multiplying both of the above equalities, we arrive at 
\begin{align*}
\left(\norm{s^{k}-g^{k-1}}^{\ast}_{x^{k-1}}\right)^{2}
&=\left( \left\|\nabla f(x^{k}) - \nabla f(x^{k-1}) - \nabla^2 f(x^{k-1})v^{k-1}  - \frac{L_{k-1}}{2}\norm{v^{k-1}}_{x^{k-1}}H(x^{k-1})v^{k-1} \right\|_{x^{k-1}}^*\right)^{2}.
%&+\frac{L^{2}_{k-1}}{4}\norm{v^{k-1}}^{4}_{x^{k-1}}\\
%&-2\inner{\nabla f(x^{k}) - \nabla f(x^{k-1}) - \nabla^2 f(x^{k-1})v^{k-1},\frac{L_{k-1}}{2}\norm{v^{k-1}}_{x^{k-1}}v^{k-1}}.
\end{align*}
Taking the square root and applying the triangle inequality, we obtain
%Using the trivial inequality $-2\inner{a,b}\leq \norm{a}^{2}+\norm{b}^{2}$, it follows 
%\begin{align*}
%\left(\norm{s^{k}-g^{k-1}}^{\ast}_{x^{k-1}}\right)^{2}&=2\left(\norm{\nabla f(x^{k}) - \nabla f(x^{k-1}) - \nabla^2 f(x^{k-1})v^{k-1}}^{\ast}_{x^{k-1}}\right)^{2}+\frac{L^{2}_{k-1}}{2}\norm{v^{k-1}}^{4}_{x^{k-1}}.
%\end{align*}
%Whence, 
\begin{align*}
\norm{s^{k}-g^{k-1}}^{\ast}_{x^{k-1}}&\leq\norm{\nabla f(x^{k}) - \nabla f(x^{k-1}) - \nabla^2 f(x^{k-1})v^{k-1}}^{\ast}_{x^{k-1}}+\frac{L_{k-1}}{2}\norm{v^{k-1}}^{2}_{x^{k-1}}\\
&\stackrel{\eqref{eq:dualnorm},\eqref{eq:SO_LS_2},\eqref{eq:localnorm}}{\leq} \frac{L_{k-1}}{2}\norm{\alpha_{k-1}v^{k-1}}^{2}_{x_{k-1}}+\frac{L_{k-1}}{2}\norm{v^{k-1}}^{2}_{x^{k-1}}.
\end{align*}
Since the stopping criterion holds, at iteration $k-1$ we have
\begin{align}
\hspace{-1em}\zeta(x^{k-1},v^{k-1}) &\stackrel{\eqref{eq:boundzeta}}{\leq} \|v^{k-1}\|_{x^{k-1}} < \Delta_{k-1} = \sqrt{\frac{\eps}{4L_{k-1}\nu}} \leq \sqrt{\frac{\eps}{4 \cdot 144 \eps \nu}} < \frac{1}{2}\label{eq:SO_eps_KKT_proof_0},
\end{align}
where we used that, by construction, $L_{k-1} \geq \underline{L} = 144 \eps$ and that $\nu \geq 1$. Hence, by \eqref{eq:SO_alpha_k}, we have that $\alpha_{k-1}=1$ and $x^{k}=x^{k-1}+v^{k-1}$. This, in turn, implies that 
\begin{equation}
\label{eq:SO_eps_KKT_proof_1}
\norm{s^{k}-g^{k-1}}^{\ast}_{x^{k-1}}\leq  L_{k-1}\norm{v^{k-1}}^{2}_{x^{k-1}}.
\end{equation}
Now we follow the analysis of the first-order method by noting that $\norm{s^{k}-g^{k-1}}^{\ast}_{x^{k-1}}=\mu\norm{s^{k}-g^{k-1}}_{\nabla^{2}h_{\ast}(g^{k-1})}$ and $\mu=\frac{\eps}{4\nu}$, 
which implies
\begin{equation}
\label{eq:SO_eps_KKT_proof_2}
\norm{s^{k}-g^{k-1}}_{\nabla^{2}h_{\ast}(g^{k-1})}\leq\frac{L_{k-1}}{\mu}\norm{v^{k-1}}^{2}_{x^{k-1}}<\frac{L_{k-1}}{\mu}\Delta_{k-1}^{2} = \frac{L_{k-1}}{\frac{\eps}{4\nu}} \cdot \frac{\eps}{4L_{k-1}\nu} = 1.
\end{equation}
%where we used the stopping criterion and that, by the assumptions of Theorem \ref{Th:SOAHBA_conv}, $\mu=\frac{\eps}{4\nu}$.
Thus, since, by \eqref{eq:relations_1}, $g^{k-1}=-\mu\nabla h(x^{k-1})\in \setK^{\ast}$, we get that $s^{k}\in\setK^{\ast}$. By construction, $x^{k}\in \setK$ and $\bA x^{k} = b$. Thus, \eqref{eq:eps_SO_optim_equality_cones} holds. We also have that, by construction, $\|\nabla f(x^{k})-\bA^{\ast}y^{k-1} - s^{k}\|=0\leq \eps$, meaning that \eqref{eq:eps_SO_optim_grad} holds with $\eps_1=\eps$. To finish the analysis of the first-order condition, it remains to check the complementarity condition \eqref{eq:eps_SO_optim_complem}. 
We have
\begin{align*}
\inner{s^{k},x^{k}}=\inner{s^{k},x^{k-1}+v^{k-1}}=\inner{s^{k},x^{k-1}}+\inner{s^{k},v^{k-1}}.
\end{align*}
We estimate each of the two terms in the r.h.s. separately. First, 
\begin{align*}
0 \leq \inner{s^{k},x^{k-1}}&%=\inner{\nabla f(x^{k})-\bA^{\ast}y^{k-1},x^{k-1}}=
=  \inner{s^{k}-g^{k-1},x^{k-1}} + \inner{g^{k-1},x^{k-1}} \\
&\leq \norm{s^{k}-g^{k-1}}^{\ast}_{x^{k-1}}\cdot\norm{x^{k-1}}_{x^{k-1}}-\mu\inner{\nabla h(x^{k-1}),x^{k-1}} \\
&\stackrel{\eqref{eq:SO_eps_KKT_proof_1},\eqref{eq:log_hom_scb_norm_prop},\eqref{eq:log_hom_scb_hess_prop}}{\leq}  L_{k-1}\norm{v^{k-1}}^{2}_{x^{k-1}}\sqrt{\nu}+\mu\nu. %\Delta^{2}_{k-1}
\end{align*}
Second, 
\begin{align*}
 \inner{s^{k},v^{k-1}}&\leq \norm{s^{k}}^{\ast}_{x^{k-1}}\cdot\norm{v^{k-1}}_{x^{k-1}}\leq \left(\norm{s^{k}-g^{k-1}}^{\ast}_{x^{k-1}}+\norm{g^{k-1}}^{\ast}_{x^{k-1}}\right)\cdot \norm{v^{k-1}}_{x^{k-1}}\\
&\stackrel{\eqref{eq:SO_eps_KKT_proof_1},\eqref{eq:log_hom_scb_norm_prop},\eqref{eq:log_hom_scb_hess_prop}}{\leq}  \left( L_{k-1}\norm{v^{k-1}}^{2}_{x^{k-1}}+\mu\sqrt{\nu}\right)\Delta_{k-1}.
\end{align*}
Summing up, using the stopping criterion $\norm{v^{k-1}}_{x^{k-1}}<\Delta_{k-1}$ and that, by \eqref{eq:SO_eps_KKT_proof_0}, $\Delta_{k-1} \leq 1\leq\sqrt{\nu}$, we obtain
\begin{align}
 0 \leq \inner{s^k,x^k} = \inner{s^k,x^{k-1}+v^{k-1}} \leq  2L_{k-1}\Delta_{k-1}^2 \sqrt{\nu} + 2\mu \nu = 
2 L_{k-1} \frac{\eps}{4L_{k-1}\nu}  \sqrt{\nu} + 2\frac{\eps}{4\nu} \nu \leq  \eps, \label{eq:SO_eps_KKT_proof_3}
\end{align}
i.e., \eqref{eq:eps_SO_optim_complem} holds with $\eps_1=\eps$.\\
Finally, we show the second-order condition \eqref{eq:eps_SO_optim_SO}.
%It remains to show the second-order approximate stationarity \eqref{eq:SO_main_Th_eps_KKT_4}. 
By inequality \eqref{eq:PD} for subproblem \eqref{eq:SO_finder} solved at iteration $k$, we obtain on $\setL_0$
\begin{align}
\nabla^2 f(x^{k})  &\succeq - \frac{L_{k}\norm{v^{k}}_{x^{k}}}{2} H(x^{k}) \succeq -\frac{L_{k} \Delta_k}{2} H(x^{k}) \notag \\
& =  - \frac{L_{k}}{2}\sqrt{\frac{\eps}{4L_{k}\nu}} H(x^{k}) = - \frac{\sqrt{L_k \eps}}{4\nu^{1/2}}H(x^{k})  
\succeq - \frac{ \sqrt{2\bar{M}\eps}}{4\nu^{1/2}}H(x^{k}) \label{eq:SO_eps_KKT_proof_4}
\end{align}
where we used the second part of the stopping criterion, i.e. $\norm{v^{k}}_{x^{k}}< \Delta_k$ and that $L_k \leq 2\bar{M}=2\max\{M,M_0\}$ (see Section \ref{sec:backtrack2}). Thus, \eqref{eq:eps_SO_optim_SO} holds with $\eps_2=\frac{\max\{M,M_0\}\eps}{8\nu}$, which finishes the proof of Theorem \ref{Th:SOAHBA_conv}.
%Using the 	second-order optimality condition \eqref{eq:PD} for the subproblem \eqref{eq:cubicproblem} solved at iteration $k$, we obtain, on $\setL_0$
%\begin{equation}
%%\label{eq:}
%[H(x^{k})]^{-1/2} \nabla^2 f(x^{k}) [H(x^{k})]^{-1/2} \succeq - \frac{L_{k}\norm{v^{k}}_{x^{k}}}{2} I \succeq -\frac{L_{k} \Delta_k}{2} I =  - \frac{L_{k}}{2} \beta\sqrt{\frac{\eps}{4L_{k}\sqrt{\nu}}} I = - \frac{\beta \sqrt{L_k \eps}}{4\nu^{1/4}}I  \succeq - \frac{\beta \sqrt{2M\eps}}{4\nu^{1/4}}I
%\end{equation}
%where we used the second part of the stopping criterion, i.e. $\norm{v^{k}}_{x^{k}}< \Delta_k$ and that $L_k \leq 2M$. This finishes the proof of \eqref{eq:SO_main_Th_eps_KKT_4} and the proof of Theorem \ref{Th:SOAHBA_conv}.

\subsection{Discussion}
\label{sec:SO_discussion}
%The analysis of $\SAHBA$ requires to overcome similar technical challenges to that of the analysis of $\AHBA$. Namely, we need to guarantee that $s^{k} \in\setK^{\ast}$, show that \eqref{eq:eps_SO_optim_complem} holds, and obtain the condition \eqref{eq:eps_SO_optim_grad}  formulated in terms of the standard Euclidean norm based on inequality \eqref{eq:SO_eps_KKT_proof_1} formulated in terms of the local norm. Additionally, we need to guarantee that the second-order condition \eqref{eq:eps_SO_optim_SO} holds. And again, in our setting of general, potentially non-symmetric, cones all this is more difficult than for the particular case of the non-negativity constraints considered in \cite{HaeLiuYe18,NeiWr20}, where a second-order algorithm and a first-order implementation of a second-order algorithm are proposed respectively. 
\paragraph{Strengthened KKT condition.}
As in Section \ref{sec:FO_discussion}, our aim in this section is to compare our result with those available in the contemporary literature. We therefore onsider the special case $\bar{\setK}=\bar{\setK}_{\text{NN}}$, endowed with the standard log-barrier $h(x)=-\sum_{i=1}^n \ln(x_i)$. Recall that for this barrier setup we have $\nabla h(x)=[-x_{1}^{-1},\ldots,-x_{n}^{-1}]^{\top}$, $H(x)=\diag[x_{1}^{-2},\ldots,x_{n}^{-2}]=\XX^{-2}$. Assume that the stopping criterion applies at iteration $k$. Using the first-order optimality condition \eqref{eq:opt1} for the subproblem \eqref{eq:SO_finder} solved at iteration $k-1$ and expanding the definition of the potential \eqref{eq:potential}, there exists a dual variable $y^{k-1}\in\R^{m}$ such that \eqref{eq:opt1} holds, i.e., 
\begin{align*}
 \nabla f(x^{k-1}) + \mu \nabla h(x^{k-1})  + \nabla^2 f(x^{k-1})v^{k-1} - \bA^{\ast}y^{k-1} =  - \frac{L_{k-1}}{2}\norm{v^{k-1}}_{x^{k-1}}H(x^{k-1})v^{k-1}.
\end{align*}
Multiplying both sides by $H(x^{k-1})^{-1/2}$, using the stopping criterion $\norm{v^{k-1}}_{x^{k-1}}<\sqrt{\frac{\eps}{4\nu L_{k-1}}}$, since $H(x^{k-1})^{-1/2}\nabla h(x^{k-1})= - \1_{n}$ and $\nu=n$, we obtain
\begin{align}
&\norm{\XX^{k-1}(\nabla^2 f(x^{k-1})v^{k-1} + \nabla f(x^{k-1})-\bA^{\ast}y^{k-1}) - \mu \1_{n}}_{\infty} \notag \\
& \leq \norm{\XX^{k-1} (\nabla^2 f(x^{k-1})v^{k-1} + \nabla f(x^{k-1})-\bA^{\ast}y^{k-1}) - \mu \1_{n}  } = \frac{L_{k-1}}{2} \norm{-\XX^{k} v^{k-1}}^2\notag\\
& < \frac{\eps}{8n}. \label{eq:SO_remarks_1}
\end{align}
Whence, since $\mu=\frac{\eps}{4n}$, the above bound \eqref{eq:SO_remarks_1} combined with the triangle inequality yields
\begin{align}
& \norm{\XX^{k-1}(\nabla^2 f(x^{k-1})v^{k-1} +\nabla f(x^{k-1})-\bA^{\ast}y^{k-1})}_{\infty} \notag \\
&\leq \norm{\XX^{k-1} (\nabla^2 f(x^{k-1})v^{k-1} + \nabla f(x^{k-1})-\bA^{\ast}y^{k-1}) - \mu \1_{n}  } +\norm{\mu\1_{n}} \notag\\
&=\frac{L_{k-1}}{2} \norm{-\XX^{k} v^{k-1}}^2+\norm{\mu\1_{n}}<\frac{3\eps}{8n}. \label{eq:SO_remarks_2}
\end{align}
Let $\mathbf{V}^{k-1} = \diag(v^{k-1})$. Using the fact that $x^{k}=x^{k-1}+v^{k-1}$ shown after \eqref{eq:SO_eps_KKT_proof_0}, we obtain
\begin{align*}
& \norm{\XX^{k}(\nabla f(x^{k})-\bA^{\ast}y^{k-1})}_{\infty} \notag \\
& = \norm{(\XX^{k-1}+\mathbf{V}^{k-1})(\nabla^2 f(x^{k-1})v^{k-1} +\nabla f(x^{k-1})-\bA^{\ast}y^{k-1} + \nabla f(x^{k}) - \nabla f(x^{k-1}) - \nabla^2 f(x^{k-1})v^{k-1})}_{\infty} \notag  \\
& \leq \norm{\XX^{k-1}(\nabla^2 f(x^{k-1})v^{k-1} +\nabla f(x^{k-1})-\bA^{\ast}y^{k-1} )}_{\infty}\\ %\label{eq:SO_remarks_3} \\
& \hspace{1em} + \norm{\XX^{k-1} (\nabla f(x^{k}) - \nabla f(x^{k-1}) - \nabla^2 f(x^{k-1})v^{k-1})}_{\infty}\\ %\label{eq:SO_remarks_4} \\
& \hspace{1em} + \norm{\mathbf{V}^{k-1}(\nabla^2 f(x^{k-1})v^{k-1} +\nabla f(x^{k-1})-\bA^{\ast}y^{k-1} )}_{\infty}\\ %\label{eq:SO_remarks_5}\\
& \hspace{1em} + \norm{ \mathbf{V}^{k-1}(\nabla f(x^{k}) - \nabla f(x^{k-1}) - \nabla^2 f(x^{k-1})v^{k-1})}_{\infty}\\ %\label{eq:SO_remarks_6}\\
&=I+II+III+IV.
\end{align*}
Let us estimate each of the four terms $I-IV$, using two technical facts \eqref{eq:technical_1}, \eqref{eq:technical_2} proved in Appendix \ref{sec:Appendix}. We have:
\begin{align*}
 I& \stackrel{\eqref{eq:SO_remarks_2}}{<} \frac{3\eps}{8n}, \\
II& \leq \norm{\XX^{k-1} (\nabla f(x^{k}) - \nabla f(x^{k-1}) - \nabla^2 f(x^{k-1})v^{k-1})}\\
& = \norm{\nabla f(x^{k}) - \nabla f(x^{k-1}) - \nabla^2 f(x^{k-1})v^{k-1}}_{x^{k-1}}^* \stackrel{\eqref{eq:SO_LS_2}}{\leq} \frac{L_{k-1}}{2}\norm{v^{k-1}}_{x^{k-1}}^2<\frac{\eps}{8n}, \\
III&  \stackrel{\eqref{eq:technical_2}}{\leq} \norm{v^{k-1}}_{x^{k-1}} \cdot \norm{\XX^{k-1} (\nabla^2 f(x^{k-1})v^{k-1} +\nabla f(x^{k-1})-\bA^{\ast}y^{k-1} )}_{\infty}  \stackrel{\eqref{eq:SO_eps_KKT_proof_0},\eqref{eq:SO_remarks_2}}{<}\frac{3\eps}{8n},
\end{align*}
where we have used $x^{k}=z^{k-1}=x^{k-1}+v^{k-1}$ in bounding $II$, and the last bound for expression $III$ uses $\norm{v^{k-1}}_{x^{k-1}}<1$, which is implied by eq. \eqref{eq:SO_eps_KKT_proof_0}. Finally, we also obtain
\begin{align*}
IV& \stackrel{\eqref{eq:technical_1}}{\leq} \norm{v^{k-1}}_{x^{k-1}} \cdot \norm{\nabla f(x^{k}) - \nabla f(x^{k-1}) - \nabla^2 f(x^{k-1})v^{k-1}}_{x^{k-1}}^{*}\\
&\stackrel{\eqref{eq:SO_eps_KKT_proof_0},\eqref{eq:SO_LS_2}}{\leq} \frac{L_{k-1}}{2}\norm{v^{k-1}}_{x^{k-1}}^2 <\frac{\eps}{8n}.
\end{align*}
Summarizing, we arrive at
\begin{equation}
\label{eq:SO_remarks_7}
\norm{\XX^{k}(\nabla f(x^{k})-\bA^{\ast}y^{k-1})}_{\infty} \leq \frac{\eps}{n}. 
\end{equation}
Further, by Theorem \ref{Th:SOAHBA_conv}, we have that $\nabla f(x^{k})-\bA^{\ast}y^{k-1} = s^k \in \setK^{\ast}_{\text{NN}}=\Rn_{++}$, and
\[
\nabla^2f(x^k)  + H(x^k) \sqrt{\frac{M\eps}{n}}  \succeq 0 \;\; \text{on} \;\; \setL_0.
\] 
By Remark \ref{rem:SO_complexity_simplified}, these inequalities are achieved after $O\left(\frac{\sqrt{M} n^{3/2}(f(x^{0}) - f_{\min}(\setX))}{\eps^{3/2}}\right)$ iterations. Assuming that $M\geq 1$, if we change $\eps \to \tilde{\eps}=\min\{n\eps, n\eps/M\}$, we obtain from these inequalities that in 
$O\left(\frac{\sqrt{M} n^{3/2}(f(x^{0}) - f_{\min}(\setX))}{\tilde{\eps}^{3/2}}\right) = O\left(\frac{M^2 (f(x^{0}) - f_{\min}(\setX))}{\eps^{3/2}}\right)$ iterations $\SAHBA$ guarantees
\begin{align*}
& x^{k} > 0, \; \nabla f(x^{k})-\bA^{\ast}y^{k-1} >0 \\
&\norm{\XX^{k}(\nabla f(x^{k})-\bA^{\ast}y^{k-1})}_{\infty} \leq \frac{\tilde{\eps}}{n} \leq \eps, \\
& \nabla^2f(x^k)  + H(x^k) \sqrt{\eps}  \succeq \nabla^2f(x^k)  + H(x^k) \sqrt{\frac{M\tilde{\eps}}{n}}  \succeq 0 \;\; \text{on} \;\; \setL_0.
\end{align*}
In contrast, the second-order algorithm of \cite{HaeLiuYe18} requires an additional assumption that the level set of the objective $f$ is bounded in the $L_{\infty}$-norm, gives a slightly worse guarantee $\nabla f(x^{k})-\bA^{\ast}y^{k-1} > -\eps$, and requires a larger number of iterations $O\left(\frac{ \max\{M,R\}^{7/2}(f(x^{0}) - f_{\min}(\setX))}{\eps^{3/2}}\right)$ ($R$ denoting the $L_{\infty}$ upper bound of the level set corresponding to $x^{0}$).  We also can repeat the same remark as in Section \ref{sec:FO_discussion} that our measure of complementarity $0\leq \inner{s^{k},x^{k}}\leq \eps$ is stronger than $\max_{1\leq i\leq n}\abs{x_{i}^{k}s_{i}^{k}}$ used in \cite{HaeLiuYe18,NeiWr20}. Furthermore, our algorithm is applicable to general cones admitting an efficient barrier setup, rather than only for $\bar{\setK}_{\text{NN}}$. For more general cones we can not use the coupling $H(x)^{-\frac{1}{2}} = \XX$, which was seen to be very helpful in the derivations of the bound \eqref{eq:SO_remarks_7} above. Thus, to deal with general cones, we had to find and exploit suitable properties of the barrier class $\scrH_{\nu}(\setK)$ and develop a new analysis technique that works for general, potentially non-symmetric, cones. Finally, our method does not rely on the trust-region techniques as in \cite{HaeLiuYe18} that may slow down the convergence in practice since the radius of the trust region is no grater than $O(\sqrt{\eps})$ leading to short steps.

\paragraph{Exploiting problem structure.}
We note that in \eqref{eq:SO_per_iter_proof_8} we can clearly observe the benefit of the use of $\nu$-SSB in our algorithm. When $\alpha_k=\frac{1}{2\zeta(x^k,v^k)}$, the per-iteration decrease of the potential is  $\frac{L_k\norm{v^{k}}^{3}_{x^{k}}}{96 (\zeta(x^k,v^k))^2} \geq \frac{ \sqrt{\eps L_k} \norm{v^{k}}_{x^{k}}^2}{96 \sqrt{4\nu} (\zeta(x^k,v^k))^2} $ which may be large if $\zeta(x^k,v^k)=\sigma_{x^k}(-v^k) \ll \norm{v^{k}}_{x^{k}}$.

\paragraph{Dependence on parameters.}
Next, we discuss more explicitly, how the algorithm and complexity bounds depend on the parameter $\mu$. The first observation is that from \eqref{eq:SO_eps_KKT_proof_2}, to guarantee that $s^{k} \in \setK^{\ast}$, we need the stopping criterion to be $\norm{v^{k-1}}_{x^{k-1}} < \Delta_{k-1}= \sqrt{\mu/L_{k-1}}$, which by \eqref{eq:SO_eps_KKT_proof_3} leads to the error $4 \mu \nu$ in the complementarity conditions and by \eqref{eq:SO_eps_KKT_proof_4} leads to the error $\sqrt{\mu/\bar{M}}$ in the second-order condition. From the analysis following equation \eqref{eq:SO_per_iter_proof_8}, we have that  
\[
K\frac{\mu^{3/2}}{24  \sqrt{\bar{M}}} %= K \min\left\{\frac{\mu}{4},\frac{\mu^{2}}{4 \ (\bar{M}+\mu)}\right\}
\leq f(x^{0})-f_{\min}(\setX)+\mu \nu.
\]
Whence, recalling that $\bar{M}=\max\{M,M_0\}$,
\[
K \leq 24(f(x^{0}) - f_{\min}(\setX)+ \mu \nu) \cdot \frac{\sqrt{2\max\{M,M_0\}}}{\mu^{3/2}}.%\max\left\{\frac{\nu}{\eps},\frac{\nu^{2}(M+\eps/\nu)}{\eps^{2}}\right\},
\]
Thus, we see that after $O(\mu^{-3/2})$ iterations the algorithm finds a $(4 \mu \nu,\mu/\bar{M})$-KKT point, and  if $\mu \to 0$, we have convergence to a KKT point, but the complexity bound tends to infinity and becomes non-informative. At the same time, as it is seen from \eqref{eq:SO_finder}, when $\mu \to 0$, the algorithm resembles a cubic-regularized Newton method, but with the regularization with the cube of the local norm. We also see from the above explicit expressions in terms of $\mu$ that the design of the algorithm requires careful balance between the desired accuracy of the approximate KKT point expressed mainly by the complementarity conditions, stopping criterion, and complexity. Moreover, the step-size must be selected carefully to ensure the feasibility of the iterates.

\subsection{Anytime convergence via restarting $\SAHBA$}
Similarly to the restarted $\AHBA$ (Algorithm \ref{alg:RestartHBA}), we can obtain anytime convergence envoking a restarted method that uses $\SAHBA$ as an inner procedure. We fix $\eps_{0}>0$ and select the starting point $x_0^{0}$ as a $4\nu$-analytic center of $\setX$ in the sense of eq. \eqref{eq:analytic_center}. In epoch $i\geq 0$ we generate a sequence 
$\{x_{i}^{k}\}_{k=0}^{K_{i}}$  by calling $\SAHBA(\mu_{i},\eps_{i},M_0^{(i)},x^{0}_{i})$ with $\mu_{i}=\frac{\eps_{i}}{4\nu}$ until the stopping condition is reached. We know that this inner procedure terminates after at most $\K_{II}(\eps_{i},x^{0}_{i})$ iterations. Store the values $x^{K_{i}}_{i}$ and $M_{K_{i}}^{(i)}$, and set $x^{i+1}_{0}\equiv x^{K_{i}}_{i}$, as well as $M_{0}^{(i+1)}\equiv M_{K_{i}}^{(i)}/2$. Updating the parameters to $\mu_{i+1}$ and $\eps_{i+1}$, we restart by calling procedure $\SAHBA(\mu_{i+1},\eps_{i+1},M_0^{(i+1)},x^{0}_{i+1})$ anew. This is formalized in Algorithm \ref{alg:RestartSAHBA}. 
%%%%%%%%%%%%%%%%%%%%%%%%%%%%
\begin{algorithm}[h!]
\caption{Restarting $\SAHBA$}
\label{alg:RestartSAHBA}
\SetAlgoLined
\KwData{ $h \in\scrH_{\nu}(\setK)$, $\eps_{0}>0$, $x_0^{0}\in\setX$ -- $4\nu$-analytic center, $M_0^{(0)}\geq 144 \eps_0$.}
\KwResult{Point $\hat{x}_i$, dual variables $\hat{y}_i$, $\hat{s}_i = \nabla f(\hat{x}_i) -\bA^{\ast}\hat{y}_i$.}
%Set $L_{0}^{0} > 0$ -- initial guess for $M$\; 
%$\hat{x}_{0}=x^{0},\hat{M}_{0}=L_{0}$, $\mu_{0}=\frac{\eps_0}{4\nu}$\;
\For{$i=0,1,\ldots$} 
{ Set $\eps_{i}=2^{-i}\eps_{0}$, $\mu_{i}=\frac{\eps_i}{4\nu}$\; 
Obtain $(\hat{x}_{i},\hat{y}_{i},\hat{s}_{i},\hat{M}_{i})$ from $\SAHBA(\mu_{i},\eps_i,M_0^{(i)},x^{0}_{i})$\;
 %
% $(x_{i}^{K_{i}},L_{i}^{K_{i}})$ from $\AHBA(\mu_{i},\eps_i,x^{0}_{i})$\; 
 Set $x_{i+1}^{0}=\hat{x}_{i}$ and $M_0^{(i+1)}=\hat{M}_{i}/2$.
	}
\end{algorithm}
%%%%%%%%%%%%%%%%%	
\begin{theorem}\label{th:ComplexityPathfollowingSAHBA}
Let Assumptions \ref{ass:1}, \ref{ass:barrier}, \ref{ass:2ndorder} hold. 
Then, for any $\eps \in (0,\eps_0)$, Algorithm \ref{alg:RestartSAHBA} finds an $(\eps,\frac{\max\{M,M_0^{(0)}\}\eps}{8\nu})$-2KKT point for problem \eqref{eq:Opt} in the sense of Definition \ref{def:eps_SOKKT} after no more than $I(\eps)\eqdef\lceil \log_{2}(\eps_{0}/\eps)\rceil+1$ restarts and at most 
$
\left\lceil 841(f(x^{0})-f_{\min}(\setX)+\eps_0)\nu^{3/2}\eps^{-3/2}\sqrt{2\max\{M,M_0^{(0)}\}}\right\rceil
$
%841 \frac{1536}{(2\sqrt{2}-1)
 iterations of $\SAHBA$.
\end{theorem}
\begin{proof}
Let us consider a restart $i \geq 0$ and mimic the proof of Theorem \ref{Th:SOAHBA_conv} with the substitution $\eps \to \eps_i$,  $\mu \to \mu_i = \eps_i/(4\nu)$, $M_0 \to M_0^{(i)}=\hat{M}_{i-1}/2$, $\underline{L} = 144  \eps \to \underline{L}_i = 144  \eps_i$, $\bar{M}=\max\{M,M_0\} \to \bar{M}_i=\max\{M,M_0^{(i)}\}$, $x^0 \to x^{0}_{i}=\hat{x}_{i-1}$. Note that $M_0^{(i)} \geq 144 \eps_i=\underline{L}_i$ for $i\geq 0$. We verify this via induction. By construction $M_0^{(0)}\geq 144 \eps_0$. Assume the bound holds for some $i\geq 1$. Then,  $M_0^{(i+1)}=M_{K_{i}}^{(i)}/2=\max\{L_{K_{i}-1}^{(i)}/2,\underline{L}_{i}\}/2\geq 144 \eps_{i}/2=144 \eps_{i+1}$, where we have used the induction hypothesis and the definition of the sequence $\eps_{i}$.\\
Let $K_i$ be the last iteration of $\SAHBA(\mu_{i},\eps_i,M_0^{(i)},x^{0}_{i})$ meaning that the stopping criterion does not hold at the inner iterations $k=0,\ldots,K_i-1$. From the analysis following equation \eqref{eq:SO_per_iter_proof_7}, we obtain
\begin{equation} \label{eq:SO_PF_proof_1} 
K_i \frac{\eps_i^{3/2}}{192 \nu^{3/2} \sqrt{2\bar{M}_i}} \leq F_{\mu_i}(x^{0}_i)-F_{\mu_i}(x^{K_i}_i).
\end{equation}
Using that $\mu_i$ is a decreasing sequence and $x_{0}^{0}$ is a $4\nu$-analytic center, we see
\begin{align}
F_{\mu_{i+1}}(x^{0}_{i+1})&=F_{\mu_{i+1}}(x^{K_i}_{i})=f(x^{K_i}_{i}) + \mu_{i+1} h(x^{K_i}_{i})=F_{\mu_{i}}(x^{K_i}_{i}) + (\mu_{i+1} - \mu_{i}) h(x^{K_i}_{i}) \notag \\
&\stackrel{\eqref{eq:analytic_center}}{\leq} F_{\mu_{i}}(x^{K_i}_{i}) + (\mu_{i+1} - \mu_{i}) (h(x_0^0)-4\nu)\notag \\
& \stackrel{\eqref{eq:SO_PF_proof_1}}{\leq}F_{\mu_i}(x^{0}_i)-K_i \frac{\eps_i^{3/2}}{192 \nu^{3/2} \sqrt{2\bar{M}_i} }  + (\mu_{i+1} - \mu_{i}) (h(x_0^0)-4\nu).
\label{eq:SO_PF_proof_2} 
\end{align}
Let $I=I(\eps)=\left\lceil \log_2 \frac{\eps_0}{\eps} \right\rceil+1$. By Theorem \ref{Th:SOAHBA_conv} applied to the restart $I-1$, we see that $\SAHBA(\mu_{I-1},\eps_{I-1},M_0^{(I-1)},x^{0}_{I-1})$ outputs an $(\eps_{I-1},\frac{\bar{M}_{I-1}\eps_{I-1}}{8\nu})$-2KKT point for problem \eqref{eq:Opt} in the sense of Definition \ref{def:eps_SOKKT}. Since $\eps_{I-1}=\eps$ and, for all $i\geq 1$,
\begin{align}
\bar{M}_i &= \max\{M,M_0^{(i)}\} = \max\{M,\hat{M}_{i-1}/2\} = \max\{M,M_{K_{i-1}}^{(i-1)}/2\} \notag\\
&= \max\{M,\max\{L_{K_{i-1}-1}^{(i-1)}/2,\underline{L}_{i-1}\}/2\} \notag\\
&\leq  \max\{M,\max\{\bar{M}_{i-1},M_0^{(i-1)}\}/2\} \leq  \max\{M,\bar{M}_{i-1}\} \notag\\
&\leq ... \leq \max\{M,\bar{M}_{0}\} \leq  \max\{M,M_0^{(0)}\}.\label{eq:SO_PF_proof_3} 
\end{align}
it follows that actually we generate an $(\eps,\frac{\max\{M,M_{0}^{(0)}\}\eps}{8\nu})$-2KKT point. Summing inequalities \eqref{eq:SO_PF_proof_2} for all the performed restarts $i=0,...,I-1$ and rearranging the terms, we obtain
\begin{align*}
\sum_{i=0}^{I-1} K_i \frac{\eps_i^{3/2}}{192 \nu^{3/2} \sqrt{2\bar{M}_i} }  & \leq F_{\mu_0}(x^{0}_0) - F_{\mu_{I}}(x^{0}_{I}) + (\mu_{I} - \mu_{0}) (h(x_0^0)-4\nu) \notag \\
&=f(x^{0}_0) + \mu_0 h(x^{0}_0)- f(x^{0}_{I}) - \mu_{I}h(x^{0}_{I}) + (\mu_{I} - \mu_{0}) (h(x_0^0)-4\nu) \notag \\
& \stackrel{\eqref{eq:analytic_center}}{\leq} f(x^{0}_0) - f_{\min}(\setX) + \mu_0 h(x^{0}_0) - \mu_{I}h(x^{0}_{0}) + 4\mu_{I} \nu + (\mu_{I} - \mu_{0}) (h(x_0^0)-4\nu) \notag \\ 
& \leq f(x^{0}_0) - f_{\min}(\setX) + 4\mu_0 \nu = f(x^{0}_0) - f_{\min}(\setX) + \eps_0,
\end{align*}
where in the last steps we have used the coupling $\mu_{0}=\eps_{0}/\nu$. From this inequality, using \eqref{eq:SO_PF_proof_3}, we obtain
\begin{align}
K_i \leq  (f(x^{0}) - f_{\min}(\setX)+  \eps_0) \cdot \frac{192\nu^{3/2}\sqrt{2\bar{M}_i}}{\eps_i^{3/2}} \leq \frac{C}{\eps_i^{3/2}},
\label{eq:SO_PF_proof_4} 
\end{align}
where $C \equiv 192(f(x^{0})-f_{\min}(\setX)+\eps_0)\nu^{3/2}\sqrt{2\max\{M,M_0^{(0)}\}}$. Finally, we obtain that the total number of iterations of procedures $\SAHBA(\mu_{i},\eps_{i},M_0^{(i)},x_{i}^{0}),0\leq i \leq I-1$, to reach accuracy $\eps$ is at most
\begin{align*}
\sum_{i=0}^{I-1}K_{i}&\leq \sum_{i=0}^{I-1} \frac{C}{\eps_i^{3/2}} \leq \frac{C}{\eps_0^{3/2}} \sum_{i=0}^{I-1} (2^i)^{3/2} \\
&\leq \frac{C}{\eps_0^{3/2}} \cdot \frac{2^{3/2\cdot(2+\log_2(\frac{\eps_0}{\eps}))}-1}{2^{3/2}-1}\leq \frac{8C}{(\sqrt{8}-1)\eps^{3/2}}\\
&<\frac{841(f(x^{0})-f_{\min}(\setX)+\eps_0)\nu^{3/2}\sqrt{2\max\{M,M_0^{(0)}\}}}{\eps^{3/2}}.
\end{align*} 
\end{proof}

%----------------------------------------------------------------------
%%% CONCLUSIONS
%----------------------------------------------------------------------
\section{Conclusion}
\label{sec:conclusion}
%----------------------------------------------------------------------
%%% Conclusions
%----------------------------------------------------------------------
%!TEX root = ./HBAConicMain.tex
%

We derived Hessian-barrier algorithms based on first- and second-order information on the objective $f$. We performed a detailed analysis of their worst-case iteration complexity in order to find a suitably defined approximate KKT point. Under weak regularity assumptions and in presence of general conic constraints, our Hessian-barrier algorithms share the best known complexity rates in the literature for first- and second-order  approximate KKT points. Our methods are characterized by a decomposition approach of the feasible set which leads to numerically efficient subproblems at each their iteration. Several open questions for the future remain. First, our iterations assume that the subproblems are solved exactly, and for practical reasons this should be relaxed. Second, we mentioned that $\AHBA$ can be interpreted as a discretization of the Hessian-barrier gradient system \cite{ABB04}, but the exact relationship is not explored yet. This, however, could be an important step towards understanding acceleration techniques of $\AHBA$, akin to accelerated methods for the cubic regularized Newton method. Furthermore, the cubic-regularized version has no corresponding continuous-time version yet. It will be very interesting to investigate this question further. Additionally, the question of convergence of the trajectory $(x^{k})_{k\geq 0}$ generated by either scheme is open. Another interesting direction for future research would be to allow for  higher-order Taylor expansions in the subproblems in order to boost convergence speed further, similar to \cite{CarGouToi19}. 

\section*{Acknowledgments}
We would like to thank Yurii Nesterov, Anton Rodomanov, Nikita Doikov, Giovanni Grapiglia and Radu Dragomir for fruitful discussions that allowed to improve the quality of the paper. M. Staudigl acknowledges financial support from the COST Action CA16228 "European Network for Game Theory".

\appendix
\section{More results on Self-concordant barriers}
\label{app:barrier}
%----------------------------------------------------------------------
%%% Appendix - SC
%----------------------------------------------------------------------
% !TEX root = ./HBAConicMain.tex

The \emph{dual cone} $\bar{\setK}^{\ast}$ is defined as $\bar{\setK}^{\ast}\eqdef\{s\in\setE^{\ast}\vert\inner{s,x}\geq 0\;\forall x\in\bar{\setK}\}$, and the \emph{dual barrier} $h_{\ast}(s)\eqdef\sup_{x\in\setK}\{\inner{-s,x}-h(x)\}$ for $s\in\bar{\setK}^{\ast}$. From \cite[Thm 3.3.1]{Ren01} we know that if $h\in\scrH_{\nu}(\setK)$, then $h_{\ast}\in\scrH_{\nu}(\setK^{\ast})$. Moreover, 
\begin{align}
&x \in \setK \Rightarrow -\nabla h(x) \in \setK^{\ast},\label{eq:relations_1}\\
&s=-\nabla h(x)\iff \nabla h_{\ast}(s)=-x\Rightarrow \nabla^{2}h_{\ast}(s)=[\nabla^{2}h(x)]^{-1}. \label{eq:relations}
\end{align}
We will also need the following properties listed in \cite[Lemma 5.4.3]{Nes18}.
\begin{proposition}
\label{prop:logSCB}
Let $h\in\scrH_{\nu}(\setK)$, $x\in\setK$, $t>0$ and recall that $H(x)=\nabla^{2}h(x)$.
Then,
\begin{align}
& \nabla^2 h(t x) = t^{-2} \nabla^2 h(x), \label{eq:log_hom_scb_hess_homog_prop}\\ 
&-\inner{\nabla h(x),x} = \nu, \label{eq:log_hom_scb_hess_prop}\\
&\norm{x}_{x}^2 = \inner{ H(x) x ,x } = \nu, \quad \inner{ \nabla h(x), [H(x)]^{-1}\nabla h(x) } = \nu. \label{eq:log_hom_scb_norm_prop}
\end{align}
\end{proposition}

\section{Useful inequalities}
\label{sec:Appendix}
%----------------------------------------------------------------------
%%% Appendix - Inequalities
%----------------------------------------------------------------------
% !TEX root = ./HBAConicMain.tex

Consider the cone $\setK_{\text{NN}}$ with the standard log-barrier $h(x)=-\sum_{i=1}^n \ln(x_i)$ which has Hessian $H(x)=\diag[x_{1}^{-2},\ldots,x_{n}^{-2}]=\XX^{-2}$. 
Let $\mathbf{V}=\diag[v_{1},\ldots,v_{n}]=\diag(v)$, $z \in \R^n$, and $x\in\setK_{\text{NN}}$. Then, 
\begin{align}
&\norm{\mathbf{V}z}_{\infty} \leq \norm{\mathbf{V}z} \leq \norm{v}_{x} \cdot \norm{z}_x^* \label{eq:technical_1}\\
& \norm{\mathbf{V}z}_{\infty} \leq \norm{v}_{x} \cdot \norm{\XX z}_{\infty}. \label{eq:technical_2}
\end{align}
The first inequality in \eqref{eq:technical_1} is trivial. Let us prove the second inequality.
Indeed, we have
\begin{align*}
\norm{\mathbf{V}z}^2 &= \sum_{i=1}^n (v_iz_i)^2 = \sum_{i=1}^n (v_i/x_i)^2\cdot (x_iz_i)^2 \leq \left(\sum_{i=1}^n (v_i/x_i)^2 \right)\cdot \left(\sum_{i=1}^n(x_iz_i)^2 \right) \\
& = \inner{H(x)v,v} \cdot \inner{[H(x)]^{-1}z,z} = \norm{v}_{x}^2 \cdot (\norm{z}_x^*)^2,
\end{align*}
which finishes the proof of \eqref{eq:technical_1}. 
For the inequality \eqref{eq:technical_2}, we have, donoting by $v/x$ the componentwise division of $v$ by $x$,
\begin{align*}
\norm{\mathbf{V}z}_{\infty} &= \norm{\mathbf{V}\XX^{-1}\XX z}_{\infty} \leq \norm{v/x}_{\infty} \cdot \norm{\XX z}_{\infty} \leq \norm{v/x} \cdot\norm{\XX z}_{\infty}= \norm{\XX^{-1}v} \cdot\norm{\XX z}_{\infty} = \norm{H(x)^{1/2}v} \cdot\norm{\XX z}_{\infty}\\
& = \norm{v}_x \cdot \norm{\XX z}_{\infty}
%  \sum_{i=1}^n (v_iz_i)^2 = \sum_{i=1}^n (v_i/x_i)^2\cdot (x_iz_i)^2 \leq \left(\sum_{i=1}^n (v_i/x_i)^2 \right)\cdot \left(\sum_{i=1}^n(x_iz_i)^2 \right) \\
%& = \inner{H(x)v,v} \cdot \inner{[H(x)]^{-1}z,z} = \norm{v}_{x}^2 \cdot (\norm{z}_x^*)^2,
\end{align*}

\bibliographystyle{plainnat}
\bibliography{mybib}

\begin{thebibliography}{71}
\providecommand{\natexlab}[1]{#1}
\providecommand{\url}[1]{\texttt{#1}}
\expandafter\ifx\csname urlstyle\endcsname\relax
  \providecommand{\doi}[1]{doi: #1}\else
  \providecommand{\doi}{doi: \begingroup \urlstyle{rm}\Url}\fi

\bibitem[Adler and Monteiro(1991)]{AdlMont91}
Ilan Adler and Renato~DC Monteiro.
\newblock Limiting behavior of the affine scaling continuous trajectories for
  linear programming problems.
\newblock \emph{Mathematical Programming}, 50\penalty0 (1-3):\penalty0 29--51,
  1991.

\bibitem[Agarwal et~al.(2017)Agarwal, Allen-Zhu, Bullins, Hazan, and
  Ma]{agarwal2017finding}
Naman Agarwal, Zeyuan Allen-Zhu, Brian Bullins, Elad Hazan, and Tengyu Ma.
\newblock Finding approximate local minima faster than gradient descent.
\newblock pages 1195--1199. ACM, 2017.
\newblock ISBN 145034528X.

\bibitem[Alizadeh and Goldfarb(2003)]{AliGol03}
Farid Alizadeh and Donald Goldfarb.
\newblock Second-order cone programming.
\newblock \emph{Mathematical programming}, 95\penalty0 (1):\penalty0 3--51,
  2003.

\bibitem[Alvarez et~al.(2004)Alvarez, Bolte, and Brahic]{ABB04}
Felipe Alvarez, J{\'e}r{\^o}me Bolte, and Olivier Brahic.
\newblock Hessian {Riemannian} gradient flows in convex programming.
\newblock \emph{SIAM Journal on Control and Optimization}, 43\penalty0
  (2):\penalty0 477--501, 2004.

\bibitem[Andreani et~al.(2021)Andreani, Fukuda, Haeser, Santos, and
  Secchin]{AndFukHaeSanSec21}
R.~Andreani, E.~H. Fukuda, G.~Haeser, D.~O. Santos, and L.~D. Secchin.
\newblock On the use of jordan algebras for improving global convergence of an
  augmented lagrangian method in nonlinear semidefinite programming.
\newblock \emph{Computational Optimization and Applications}, 79\penalty0
  (3):\penalty0 633--648, 2021.
\newblock \doi{10.1007/s10589-021-00281-8}.
\newblock URL \url{https://doi.org/10.1007/s10589-021-00281-8}.

\bibitem[Andreani et~al.(2019)Andreani, Fukuda, Haeser, Santos, and
  Secchin]{Andreani:2019uf}
Roberto Andreani, Ellen~H Fukuda, Gabriel Haeser, Daiana~O Santos, and
  Leonardo~D Secchin.
\newblock Optimality conditions for nonlinear second-order cone programming and
  symmetric cone programming.
\newblock \emph{Optimization online}, 2019.

\bibitem[Bayer and Lagarias(1989{\natexlab{a}})]{BayLag89II}
D.~Bayer and J.~Lagarias.
\newblock The nonlinear geometry of linear programming. ii. legendre transform
  coordinates and central trajectories.
\newblock \emph{Transactions of the American Mathematical Society},
  314:\penalty0 527--581, 1989{\natexlab{a}}.

\bibitem[Bayer and Lagarias(1989{\natexlab{b}})]{BayLag89}
D.~A. Bayer and J.~C. Lagarias.
\newblock The nonlinear geometry of linear programming. i. affine and
  projective scaling trajectories.
\newblock \emph{Trans. Amer. Math. Soc}, pages 499--526, 1989{\natexlab{b}}.

\bibitem[Ben-Tal and Nemirovski(2001)]{BenNem01}
Aharon Ben-Tal and Arkadi Nemirovski.
\newblock \emph{Lectures on Modern Convex Optimization}.
\newblock Society for Industrial and Applied Mathematics, 2021/01/07 2001.
\newblock ISBN 978-0-89871-491-3.
\newblock \doi{doi:10.1137/1.9780898718829}.
\newblock URL \url{https://doi.org/10.1137/1.9780898718829}.

\bibitem[Bian and Chen(2013)]{Bian:2013vd}
Wei Bian and Xiaojun Chen.
\newblock Worst-case complexity of smoothing quadratic regularization methods
  for non-lipschitzian optimization.
\newblock \emph{SIAM Journal on Optimization}, 23\penalty0 (3):\penalty0
  1718--1741, 2013.

\bibitem[Bian and Chen(2015)]{BiaChe15}
Wei Bian and Xiaojun Chen.
\newblock Linearly constrained non-lipschitz optimization for image
  restoration.
\newblock \emph{SIAM Journal on Imaging Sciences}, 8\penalty0 (4):\penalty0
  2294--2322, 2015.
\newblock \doi{10.1137/140985639}.
\newblock URL \url{https://doi.org/10.1137/140985639}.

\bibitem[Bian et~al.(2015)Bian, Chen, and Ye]{BiaCheYe15}
Wei Bian, Xiaojun Chen, and Yinyu Ye.
\newblock Complexity analysis of interior point algorithms for non-lipschitz
  and nonconvex minimization.
\newblock \emph{Mathematical Programming}, 149\penalty0 (1):\penalty0 301--327,
  2015.
\newblock \doi{10.1007/s10107-014-0753-5}.
\newblock URL \url{https://doi.org/10.1007/s10107-014-0753-5}.

\bibitem[Birgin and Mart{\'\i}nez(2018)]{birgin2017regularization}
E.~G. Birgin and J.~M. Mart{\'\i}nez.
\newblock On regularization and active-set methods with complexity for
  constrained optimization.
\newblock \emph{SIAM Journal on Optimization}, 28\penalty0 (2):\penalty0
  1367--1395, 2018.
\newblock \doi{10.1137/17M1127107}.
\newblock URL \url{https://doi.org/10.1137/17M1127107}.

\bibitem[Birgin and Mart{\'\i}nez(2020)]{birgin2017complexity}
E.~G. Birgin and J.~M. Mart{\'\i}nez.
\newblock Complexity and performance of an augmented lagrangian algorithm.
\newblock \emph{Optimization Methods and Software}, 35\penalty0 (5):\penalty0
  885--920, 2020.
\newblock \doi{10.1080/10556788.2020.1746962}.
\newblock URL \url{https://doi.org/10.1080/10556788.2020.1746962}.

\bibitem[Bolte and Teboulle(2003)]{BolTeb03}
J.~Bolte and M.~Teboulle.
\newblock Barrier operators and associated gradient-like dynamical systems for
  constrained minimization problems.
\newblock \emph{SIAM Journal on Control and Optimization}, 42\penalty0
  (4):\penalty0 1266--1292, 2003.
\newblock \doi{10.1137/S0363012902410861}.
\newblock URL \url{https://doi.org/10.1137/S0363012902410861}.

\bibitem[Bomze et~al.(2019)Bomze, Mertikopoulos, Schachinger, and
  Staudigl]{HBA-linear}
Immanuel~M Bomze, Panayotis Mertikopoulos, Werner Schachinger, and Mathias
  Staudigl.
\newblock Hessian barrier algorithms for linearly constrained optimization
  problems.
\newblock \emph{SIAM Journal on Optimization}, 29\penalty0 (3):\penalty0
  2100--2127, 2019.

\bibitem[Carmon et~al.(2017)Carmon, Duchi, Hinder, and
  Sidford]{carmon2017convex}
Yair Carmon, John~C Duchi, Oliver Hinder, and Aaron Sidford.
\newblock Convex until proven guilty: Dimension-free acceleration of gradient
  descent on non-convex functions.
\newblock pages 654--663. JMLR. org, 2017.

\bibitem[Carmon et~al.(2019{\natexlab{a}})Carmon, Duchi, Hinder, and
  Sidford]{CarDucHinSid19}
Yair Carmon, John~C. Duchi, Oliver Hinder, and Aaron Sidford.
\newblock Lower bounds for finding stationary points i.
\newblock \emph{Mathematical Programming}, 2019{\natexlab{a}}.
\newblock \doi{10.1007/s10107-019-01406-y}.
\newblock URL \url{https://doi.org/10.1007/s10107-019-01406-y}.

\bibitem[Carmon et~al.(2019{\natexlab{b}})Carmon, Duchi, Hinder, and
  Sidford]{CarDucHinSid19b}
Yair Carmon, John~C. Duchi, Oliver Hinder, and Aaron Sidford.
\newblock Lower bounds for finding stationary points ii: first-order methods.
\newblock \emph{Mathematical Programming}, 2019{\natexlab{b}}.
\newblock \doi{10.1007/s10107-019-01431-x}.
\newblock URL \url{https://doi.org/10.1007/s10107-019-01431-x}.

\bibitem[Cartis et~al.(2012{\natexlab{a}})Cartis, Gould, and
  Toint]{cartis2012complexity}
C.~Cartis, N.I.M. Gould, and Ph.L. Toint.
\newblock Complexity bounds for second-order optimality in unconstrained
  optimization.
\newblock \emph{Journal of Complexity}, 28\penalty0 (1):\penalty0 93--108,
  2012{\natexlab{a}}.
\newblock ISSN 0885-064X.
\newblock \doi{https://doi.org/10.1016/j.jco.2011.06.001}.
\newblock URL
  \url{https://www.sciencedirect.com/science/article/pii/S0885064X11000537}.

\bibitem[Cartis et~al.(2011)Cartis, Gould, and Toint]{CarGouToi11a}
Coralia Cartis, Nicholas I.~M. Gould, and Philippe~L. Toint.
\newblock Adaptive cubic regularisation methods for unconstrained optimization.
  part i: motivation, convergence and numerical results.
\newblock \emph{Mathematical Programming}, 127\penalty0 (2):\penalty0 245--295,
  2011.
\newblock \doi{10.1007/s10107-009-0286-5}.
\newblock URL \url{https://doi.org/10.1007/s10107-009-0286-5}.

\bibitem[Cartis et~al.(2012{\natexlab{b}})Cartis, Gould, and
  Toint]{CarGouToi12}
Coralia Cartis, Nicholas~IM Gould, and Ph~L Toint.
\newblock An adaptive cubic regularization algorithm for nonconvex optimization
  with convex constraints and its function-evaluation complexity.
\newblock \emph{IMA Journal of Numerical Analysis}, 32\penalty0 (4):\penalty0
  1662--1695, 2012{\natexlab{b}}.

\bibitem[Cartis et~al.(2018)Cartis, Gould, and Toint]{CarGouToi18}
Coralia Cartis, Nick I.~M. Gould, and Philippe~L. Toint.
\newblock Second-order optimality and beyond: Characterization and evaluation
  complexity in convexly constrained nonlinear optimization.
\newblock \emph{Foundations of Computational Mathematics}, 18\penalty0
  (5):\penalty0 1073--1107, 2018.
\newblock \doi{10.1007/s10208-017-9363-y}.
\newblock URL \url{https://doi.org/10.1007/s10208-017-9363-y}.

\bibitem[Cartis et~al.(2019{\natexlab{a}})Cartis, Gould, and
  Toint]{cartis2019optimality}
Coralia Cartis, Nicholas I.~M. Gould, and Philippe~L. Toint.
\newblock Optimality of orders one to three and beyond: characterization and
  evaluation complexity in constrained nonconvex optimization.
\newblock \emph{Journal of Complexity}, 53:\penalty0 68--94,
  2019{\natexlab{a}}.

\bibitem[Cartis et~al.(2019{\natexlab{b}})Cartis, Gould, and
  Toint]{CarGouToi19}
Coralia Cartis, Nick~I. Gould, and Philippe~L. Toint.
\newblock Universal regularization methods: Varying the power, the smoothness
  and the accuracy.
\newblock \emph{SIAM Journal on Optimization}, 29\penalty0 (1):\penalty0
  595--615, 2021/04/10 2019{\natexlab{b}}.
\newblock \doi{10.1137/16M1106316}.
\newblock URL \url{https://doi.org/10.1137/16M1106316}.

\bibitem[Chares(2009)]{Chares:2009wb}
Robert Chares.
\newblock \emph{Cones and interior-point algorithms for structured convex
  optimization involving powers andexponentials}.
\newblock PhD thesis, UCL-Universit\'{e} Catholique de Louvain
  Louvain-la-Neuve, Belgium, 2009.

\bibitem[Chen et~al.(2014)Chen, Ge, Wang, and Ye]{Chen:2014wx}
Xiaojun Chen, Dongdong Ge, Zizhuo Wang, and Yinyu Ye.
\newblock Complexity of unconstrained $l_2$-$l_p$ minimization.
\newblock \emph{Mathematical Programming}, 143\penalty0 (1):\penalty0 371--383,
  2014.
\newblock \doi{10.1007/s10107-012-0613-0}.
\newblock URL \url{https://doi.org/10.1007/s10107-012-0613-0}.

\bibitem[Conn et~al.(2000)Conn, Gould, and Toint]{conn2000trust}
Andrew Conn, Nicholas Gould, and Philippe Toint.
\newblock \emph{Trust Region Methods}.
\newblock Society for Industrial and Applied Mathematics, 2000.

\bibitem[Curtis et~al.(2017)Curtis, Robinson, and Samadi]{CurRobSam17}
Frank~E Curtis, Daniel~P Robinson, and Mohammadreza Samadi.
\newblock A trust region algorithm with a worst-case iteration complexity of
  $\mathcal{O}(\epsilon^{-3/2})$ for nonconvex optimization.
\newblock \emph{Mathematical Programming}, 162\penalty0 (1-2):\penalty0 1--32,
  2017.

\bibitem[Curtis et~al.(2018)Curtis, Robinson, and Samadi]{curtis2018complexity}
Frank~E. Curtis, Daniel~P. Robinson, and Mohammadreza Samadi.
\newblock Complexity analysis of a trust funnel algorithm for equality
  constrained optimization.
\newblock \emph{SIAM Journal on Optimization}, 28\penalty0 (2):\penalty0
  1533--1563, 2018.
\newblock \doi{10.1137/16M1108650}.
\newblock URL \url{https://doi.org/10.1137/16M1108650}.

\bibitem[De~Klerk(2006)]{De-Klerk:2006aa}
Etienne De~Klerk.
\newblock \emph{Aspects of semidefinite programming: interior point algorithms
  and selected applications}, volume~65.
\newblock Springer Science \& Business Media, 2006.
\newblock ISBN 0306478196.

\bibitem[Doikov and Nesterov(2021)]{doikov2021minimizing}
Nikita Doikov and Yurii Nesterov.
\newblock Minimizing uniformly convex functions by cubic regularization of
  newton method.
\newblock \emph{Journal of Optimization Theory and Applications}, 189\penalty0
  (1):\penalty0 317--339, Apr 2021.
\newblock ISSN 1573-2878.
\newblock \doi{10.1007/s10957-021-01838-7}.
\newblock URL \url{https://doi.org/10.1007/s10957-021-01838-7}.

\bibitem[Dvurechensky et~al.(2022)Dvurechensky, Safin, Shtern, and
  Staudigl]{Dvurechensky:2022tu}
Pavel Dvurechensky, Kamil Safin, Shimrit Shtern, and Mathias Staudigl.
\newblock Generalized self-concordant analysis of frank--wolfe algorithms.
\newblock \emph{Mathematical Programming}, pages 1--69, 2022.

\bibitem[Fan and Li(2001)]{FanLi01}
Jianquing Fan and Runze Li.
\newblock Variable selection via nonconcave penalized likelihood and its oracle
  properties.
\newblock \emph{Journal of the American Statistical Association}, 96\penalty0
  (456):\penalty0 1348--1360, 2001.

\bibitem[Faraut and Koranyi(1994)]{FarKor94}
Jacques Faraut and Adam Koranyi.
\newblock \emph{Analysis on symmetric cones}.
\newblock Oxford mathematical monographs. Oxford University Press, 1994.

\bibitem[Faybusovich(2008)]{Fay08}
L.~Faybusovich.
\newblock Several jordan-algebraic aspects of optimization.
\newblock \emph{Optimization}, 57\penalty0 (3):\penalty0 379--393, June 2008.
\newblock ISSN 0233-1934.
\newblock \doi{10.1080/02331930701523510}.
\newblock URL \url{https://doi.org/10.1080/02331930701523510}.

\bibitem[Faybusovich and Lu(2006)]{FayLu06}
Leonid Faybusovich and Ye~Lu.
\newblock Jordan-algebraic aspects of nonconvex optimization over symmetric
  cones.
\newblock \emph{Applied Mathematics and Optimization}, 53\penalty0
  (1):\penalty0 67--77, 2006.
\newblock ISSN 1432-0606.
\newblock URL \url{https://doi.org/10.1007/s00245-005-0835-0}.

\bibitem[Fiacco and McCormick(1968)]{FiacMcCo68}
Anthony~V. Fiacco and G.~P. McCormick.
\newblock \emph{Nonlinear Programming: Sequential Unconstrained Minimization
  Techniques}.
\newblock John Wiley \& Sons, New York, NY, USA, 1968.
\newblock Reprinted by SIAM Publications in 1990.

\bibitem[Foucart and Lai(2009)]{Fou09}
Simon Foucart and Ming-Jun Lai.
\newblock Sparsest solutions of underdetermined linear systems via
  $\ell_{q}$-minimization for $0<q\leq 1$.
\newblock \emph{Applied and Computational Harmonic Analysis}, 26\penalty0
  (3):\penalty0 395--407, 2009.
\newblock \doi{https://doi.org/10.1016/j.acha.2008.09.001}.
\newblock URL
  \url{https://www.sciencedirect.com/science/article/pii/S1063520308000882}.

\bibitem[Ge et~al.(2011)Ge, Jiang, and Ye]{GeJiaYe11}
Dongdong Ge, Xiaoye Jiang, and Yinyu Ye.
\newblock A note on the complexity of $l_p$ minimization.
\newblock \emph{Mathematical Programming}, 129\penalty0 (2):\penalty0 285--299,
  2011.
\newblock \doi{10.1007/s10107-011-0470-2}.
\newblock URL \url{https://doi.org/10.1007/s10107-011-0470-2}.

\bibitem[Ghadimi and Lan(2016)]{GhaLan16}
Saeed Ghadimi and Guanghui Lan.
\newblock Accelerated gradient methods for nonconvex nonlinear and stochastic
  programming.
\newblock \emph{Mathematical Programming}, 156\penalty0 (1):\penalty0 59--99,
  2016.
\newblock \doi{10.1007/s10107-015-0871-8}.
\newblock URL \url{https://doi.org/10.1007/s10107-015-0871-8}.

\bibitem[Grapiglia and Yuan(2020)]{grapiglia2020complexity}
Geovani~Nunes Grapiglia and Ya-xiang Yuan.
\newblock {On the complexity of an augmented Lagrangian method for nonconvex
  optimization}.
\newblock \emph{IMA Journal of Numerical Analysis}, 41\penalty0 (2):\penalty0
  1546--1568, 07 2020.
\newblock ISSN 0272-4979.
\newblock \doi{10.1093/imanum/draa021}.
\newblock URL \url{https://doi.org/10.1093/imanum/draa021}.

\bibitem[Griewank(1981)]{Gri81}
Andreas Griewank.
\newblock The modification of newton's method for unconstrained optimization by
  bounding cubic terms.
\newblock Technical report, Department of Applied Mathematics and Theoretical
  Physics, University of Cambridge., 1981.

\bibitem[G{\"u}ler and Tun{\c c}el(1998)]{GulTun98}
Osman G{\"u}ler and Levent Tun{\c c}el.
\newblock Characterization of the barrier parameter of homogeneous convex
  cones.
\newblock \emph{Mathematical Programming}, 81\penalty0 (1):\penalty0 55--76,
  1998.
\newblock \doi{10.1007/BF01584844}.
\newblock URL \url{https://doi.org/10.1007/BF01584844}.

\bibitem[Guminov et~al.(2021)Guminov, Dvurechensky, Tupitsa, and
  Gasnikov]{guminov2021combination}
Sergey Guminov, Pavel Dvurechensky, Nazarii Tupitsa, and Alexander Gasnikov.
\newblock On a combination of alternating minimization and {N}esterov's
  momentum.
\newblock In \emph{Proceedings of the 38th International Conference on Machine
  Learning}, volume 139 of \emph{PMLR}, pages 3886--3898. PMLR, 2021.
\newblock URL \url{http://proceedings.mlr.press/v139/guminov21a.html}.

\bibitem[Haeser et~al.(2019)Haeser, Liu, and Ye]{HaeLiuYe18}
Gabriel Haeser, Hongcheng Liu, and Yinyu Ye.
\newblock Optimality condition and complexity analysis for linearly-constrained
  optimization without differentiability on the boundary.
\newblock \emph{Mathematical Programming}, 178\penalty0 (1):\penalty0 263--299,
  Nov 2019.
\newblock ISSN 1436-4646.
\newblock \doi{10.1007/s10107-018-1290-4}.
\newblock URL \url{https://doi.org/10.1007/s10107-018-1290-4}.

\bibitem[Hauser and G{\"u}ler(2002)]{HauGul02}
Raphael~A. Hauser and Osman G{\"u}ler.
\newblock Self-scaled barrier functions on symmetric cones and their
  classification.
\newblock \emph{Foundations of Computational Mathematics}, 2\penalty0
  (2):\penalty0 121--143, 2002.
\newblock \doi{10.1007/s102080010022}.
\newblock URL \url{https://doi.org/10.1007/s102080010022}.

\bibitem[Helmke and Moore(1996)]{HelMoo96}
Uwe Helmke and John~B. Moore.
\newblock \emph{Optimization and Dynamical Systems}.
\newblock Communications \& Control Engineering. Springer Berlin Heidelberg,
  1996.

\bibitem[Hinder and Ye(2018)]{hinder2018worst-case}
Oliver Hinder and Yinyu Ye.
\newblock Worst-case iteration bounds for log barrier methods for problems with
  nonconvex constraints.
\newblock \emph{arXiv:1807.00404}, 2018.

\bibitem[Huang et~al.(2009)Huang, Ma, Xie, and Zhang]{HUANG:2009}
Jian Huang, Shuange Ma, Huiliange Xie, and Cun-Hui Zhang.
\newblock A group bridge approach for variable selection.
\newblock \emph{Biometrika}, 96\penalty0 (2):\penalty0 339--355, 2022/09/03/
  2009.
\newblock URL \url{http://www.jstor.org/stable/27798828}.

\bibitem[Ji et~al.(2013)Ji, Sze, Zhou, So, and Ye]{JiSoSzeYe13}
S.~Ji, K.~Sze, Z.~Zhou, A.~M. So, and Y.~Ye.
\newblock Beyond convex relaxation: A polynomial-time non-convex optimization
  approach to network localization.
\newblock In \emph{2013 Proceedings IEEE INFOCOM}, pages 2499--2507, 2013.
\newblock \doi{10.1109/INFCOM.2013.6567056}.

\bibitem[Jia et~al.(2022)Jia, Liang, Shen, and Zhang]{Jia22}
Xiaojing Jia, Xin Liang, Chungen Shen, and Lei-Hong Zhang.
\newblock Solving the cubic regularization model by a nested restarting lanczos
  method.
\newblock \emph{SIAM Journal on Matrix Analysis and Applications}, 43\penalty0
  (2):\penalty0 812--839, 2022.
\newblock \doi{10.1137/21M1436324}.
\newblock URL \url{https://doi.org/10.1137/21M1436324}.

\bibitem[Lan(2020)]{lan2020first}
Guanghui Lan.
\newblock \emph{First-order and Stochastic Optimization Methods for Machine
  Learning}.
\newblock Springer Nature, 2020.

\bibitem[Laurent and Rendl(2005)]{LauRen05}
Monique Laurent and Franz Rendl.
\newblock Semidefinite programming and integer programming.
\newblock \emph{Handbooks in Operations Research and Management Science},
  12:\penalty0 393--514, 2005.

\bibitem[Liu et~al.(2017)Liu, Yao, Li, and Ye]{LiLiuYaoYe17}
Hongcheng Liu, Tao Yao, Runze Li, and Yinyu Ye.
\newblock Folded concave penalized sparse linear regression: sparsity,
  statistical performance, and algorithmic theory for local solutions.
\newblock \emph{Mathematical Programming}, 166\penalty0 (1):\penalty0 207--240,
  2017.
\newblock \doi{10.1007/s10107-017-1114-y}.
\newblock URL \url{https://doi.org/10.1007/s10107-017-1114-y}.

\bibitem[Louren{\c c}o et~al.(2018)Louren{\c c}o, Fukuda, and
  Fukushima]{LouFukMas18}
Bruno~F. Louren{\c c}o, Ellen~H. Fukuda, and Masao Fukushima.
\newblock Optimality conditions for problems over symmetric cones and a simple
  augmented lagrangian method.
\newblock \emph{Mathematics of Operations Research}, 43\penalty0 (4):\penalty0
  1233--1251, 2021/09/05 2018.
\newblock \doi{10.1287/moor.2017.0901}.
\newblock URL \url{https://doi.org/10.1287/moor.2017.0901}.

\bibitem[Lu and Yuan(2007)]{LuYua07}
Ye~Lu and Ya‐Xiang Yuan.
\newblock An interior‐point trust‐region algorithm for general symmetric
  cone programming.
\newblock \emph{SIAM Journal on Optimization}, 18\penalty0 (1):\penalty0
  65--86, 2020/08/03 2007.
\newblock \doi{10.1137/040611756}.
\newblock URL \url{https://doi.org/10.1137/040611756}.

\bibitem[Molzahn and Hiskens(2019)]{Mol19}
Daniel~K. Molzahn and Ian~A. Hiskens.
\newblock A survey of relaxations and approximations of the power flow
  equations.
\newblock \emph{Foundations and Trends{\textregistered} in Electric Energy
  Systems}, 4\penalty0 (1-2):\penalty0 1--221, 2019.
\newblock ISSN 2332-6557.
\newblock \doi{10.1561/3100000012}.
\newblock URL \url{http://dx.doi.org/10.1561/3100000012}.

\bibitem[Nesterov and Nemirovski(1994)]{NesNem94}
Yu. Nesterov and A.~Nemirovski.
\newblock \emph{Interior Point Polynomial methods in Convex programming}.
\newblock SIAM Publications, 1994.

\bibitem[Nesterov and Todd(1997)]{NesTod97}
Yu.~E. Nesterov and M.~J. Todd.
\newblock Self-scaled barriers and interior-point methods for convex
  programming.
\newblock \emph{Mathematics of Operations Research}, 22\penalty0 (1):\penalty0
  1--42, 2020/06/29 1997.
\newblock \doi{10.1287/moor.22.1.1}.
\newblock URL \url{https://doi.org/10.1287/moor.22.1.1}.

\bibitem[Nesterov(2018)]{Nes18}
Yurii Nesterov.
\newblock \emph{Lectures on Convex Optimization}, volume 137 of \emph{Springer
  Optimization and Its Applications}.
\newblock Springer International Publishing, 2018.

\bibitem[Nesterov and Polyak(2006)]{NesPol06}
Yurii Nesterov and Boris Polyak.
\newblock Cubic regularization of newton method and its global performance.
\newblock \emph{Mathematical Programming}, 108\penalty0 (1):\penalty0 177--205,
  2006.
\newblock ISSN 1436-4646.
\newblock \doi{10.1007/s10107-006-0706-8}.
\newblock URL \url{http://dx.doi.org/10.1007/s10107-006-0706-8}.

\bibitem[Nesterov et~al.(2020)Nesterov, Gasnikov, Guminov, and
  Dvurechensky]{nesterov2020primal-dual}
Yurii Nesterov, Alexander Gasnikov, Sergey Guminov, and Pavel Dvurechensky.
\newblock Primal-dual accelerated gradient methods with small-dimensional
  relaxation oracle.
\newblock \emph{Optimization Methods and Software}, pages 1--28, 2020.
\newblock \doi{10.1080/10556788.2020.1731747}.
\newblock URL \url{https://doi.org/10.1080/10556788.2020.1731747}.

\bibitem[Nocedal and Wright(2000)]{NocWri00}
Jorge Nocedal and Stephen~J. Wright.
\newblock \emph{Numerical Optimization}.
\newblock Springer, 2nd edition, 2000.

\bibitem[O'Neill and Wright(2020)]{NeiWr20}
Michael O'Neill and Stephen~J Wright.
\newblock A log-barrier {N}ewton-{CG} method for bound constrained optimization
  with complexity guarantees.
\newblock \emph{IMA Journal of Numerical Analysis}, 12/27/2020 2020.
\newblock \doi{10.1093/imanum/drz074}.
\newblock URL \url{https://doi.org/10.1093/imanum/drz074}.

\bibitem[Renegar(2001)]{Ren01}
James Renegar.
\newblock \emph{A Mathematical View of Interior-Point Methods in Convex
  Optimization}.
\newblock Society for Industrial and Applied Mathematics, 2001.
\newblock \doi{10.1137/1.9780898718812}.
\newblock URL \url{https://epubs.siam.org/doi/abs/10.1137/1.9780898718812}.

\bibitem[Schmieta and Alizadeh(2003)]{Schmieta2003}
S.~H. Schmieta and F.~Alizadeh.
\newblock Extension of primal-dual interior point algorithms to symmetric
  cones.
\newblock \emph{Mathematical Programming}, 96\penalty0 (3):\penalty0 409--438,
  2003.
\newblock ISSN 1436-4646.
\newblock \doi{10.1007/s10107-003-0380-z}.
\newblock URL \url{https://doi.org/10.1007/s10107-003-0380-z}.

\bibitem[Tseng(2007)]{Tse07}
Paul Tseng.
\newblock Second-order cone programming relaxation of sensor network
  localization.
\newblock \emph{SIAM Journal on Optimization}, 18\penalty0 (1):\penalty0
  156--185, 2007.

\bibitem[Tseng et~al.(2011)Tseng, Bomze, and Schachinger]{TseBomSch11}
Paul Tseng, Immanuel~M. Bomze, and Werner Schachinger.
\newblock A first-order interior-point method for linearly constrained smooth
  optimization.
\newblock \emph{Mathematical Programming}, 127\penalty0 (2):\penalty0 399--424,
  2011.
\newblock ISSN 1436-4646.
\newblock \doi{10.1007/s10107-009-0292-7}.
\newblock URL \url{http://dx.doi.org/10.1007/s10107-009-0292-7}.

\bibitem[Wen et~al.(2018)Wen, Chu, Liu, and Qiu]{SurveyNonconvex}
Fei Wen, Lei Chu, Peilin Liu, and Robert~C. Qiu.
\newblock A survey on nonconvex regularization-based sparse and low-rank
  recovery in signal processing, statistics, and machine learning.
\newblock \emph{IEEE Access}, 6:\penalty0 69883--69906, 2018.
\newblock \doi{10.1109/ACCESS.2018.2880454}.

\bibitem[Ye(1992)]{Ye92}
Yinyu Ye.
\newblock On affine scaling algorithms for nonconvex quadratic programming.
\newblock \emph{Mathematical Programming}, 56\penalty0 (1):\penalty0 285--300,
  1992.
\newblock \doi{10.1007/BF01580903}.
\newblock URL \url{https://doi.org/10.1007/BF01580903}.

\end{thebibliography}
\end{document}